\documentclass{article}

 \usepackage[preprint]{arxiv}

\usepackage[utf8]{inputenc} %
\usepackage[T1]{fontenc}    %
\usepackage{hyperref}       %
\usepackage{url}            %
\usepackage{booktabs}       %
\usepackage{amsfonts}       %
\usepackage{nicefrac}       %
\usepackage{microtype}      %
\usepackage{xcolor}         %

\usepackage{bbm}
\usepackage{algorithm}
\usepackage{algorithmic}
\usepackage{comment}
\usepackage{nicematrix}
\usepackage{mathtools}
\usepackage{subfigure}
\usepackage{xcolor}
\usepackage[capitalize,noabbrev]{cleveref}
\usepackage{lastpage}
\usepackage{amsmath,amssymb}
\usepackage{amsthm}
\usepackage[table]{xcolor}
\usepackage{tikz}
\usetikzlibrary{decorations.pathreplacing,tikzmark,calc}

\newtheorem{theorem}{Theorem}
\newtheorem{lemma}[theorem]{Lemma}
\newtheorem{proposition}[theorem]{Proposition}

\theoremstyle{definition}
\newtheorem{definition}{Definition}
\newtheorem{assumption}{Assumption}
\newtheorem{condition}{Condition}

\theoremstyle{remark}
\newtheorem{remark}{Remark}

\title{Fast and Efficient Parallel Sampling Using Higher Order Langevin Dynamics}

\author{%
  Jaideep Mahajan\thanks{These authors contributed equally to this work.} \\
  Department of Statistics\\
  University of Illinois Urbana-Champaign\\
  \texttt{jaideep3@illinois.edu}
  \And
  Kaihong Zhang\footnotemark[1] \\
  Department of Statistics\\
  University of Illinois Urbana-Champaign\\
  \texttt{kaihong5@illinois.edu}
  \And
  Feng Liang\\
  Department of Statistics\\
  University of Illinois Urbana-Champaign\\
  \texttt{liangf@illinois.edu}
  \And
  Jingbo Liu\\
  Department of Statistics\\
  University of Illinois Urbana-Champaign\\
  \texttt{jingbol@illinois.edu}
}

\begin{document}

\maketitle

\begin{abstract}
We study parallel sampling from high-dimensional strongly log-concave distributions. Langevin-based samplers converge rapidly in continuous time, but their discretizations are typically sequential and often require polynomially many steps in the dimension \(d\), the target accuracy \(\varepsilon^{-1}\), or both. Picard-based parallel sampling methods reduce this sequential depth to polylogarithmic scale by solving for many time-discretization points in parallel; however, existing guarantees often require a polynomial number of processors, leading to substantial memory and gradient-evaluation costs in high dimensions.

We show that higher-order Langevin structure can reduce this parallel resource burden while preserving polylogarithmic sequential depth. Our method combines arbitrary-order Langevin dynamics with blockwise Lagrange polynomial interpolation. This sharper discretization reduces the number of parallel points required to achieve a target accuracy. Our results cover both higher-order smooth potentials and ridge-separable potentials, including models such as Bayesian logistic regression and two-layer neural networks, and improve upon the space complexity of the current literature on parallel log-concave sampling.
\end{abstract}

\section{Introduction}

Sampling from high-dimensional log-concave distributions is a central algorithmic primitive in statistics, machine learning, and scientific computing \citep{robert2004monte, marin2007bayesian, nakajima2019variational}. Formally, the goal is to generate approximate samples from a target density
\begin{equation} \label{targetfun:p*}
p^*(x) \;\propto\; \exp(-U(x)), \qquad x \in \mathbb{R}^d,
\end{equation}
where $U : \mathbb{R}^d \to \mathbb{R}$ is a smooth, strongly convex potential function. Directly computing the normalizing constant \( Z = \int_{\mathbb{R}^d} e^{-U(x)}\,dx \) is generally infeasible, since numerical integration scales exponentially with the dimension $d$. This motivates the study of sampling algorithms that avoid explicit evaluation of $Z$.

A widely used approach is based on Langevin diffusion, which admits \(p^*\) as its stationary distribution and converges exponentially fast in continuous time under mild assumptions. In practice, however, the diffusion must be discretized, and standard schemes such as the Euler-Maruyama method introduce bias that can be controlled by taking small step sizes. Consequently, sequential Langevin algorithms often require a polynomial number of iterations in the dimension \(d\), the target accuracy \(\varepsilon^{-1}\), or both. A large literature has sought to improve this dependence through sharper discretizations and structured dynamics, including overdamped methods, underdamped methods (see \cite{Chewi26Book}), and higher-order variants \citep{mou2021high}. While these methods exploit smoother trajectories and auxiliary variables to reduce discretization error, they are still typically implemented sequentially, with each update depending on the previous one.

This sequential structure is natural from the viewpoint of numerical discretization, but it can be limiting in modern large-scale computation, where parallel hardware is often readily available. In optimization and machine learning, parallelism is routinely used to accelerate expensive computations by distributing gradient evaluations, data batches, or parameter updates across multiple processors \citep{recht2011hogwild,  abadi2016tensorflow, you2017large, jain2018parallelizing}. This motivates an analogous question for sampling: can parallel computation reduce the number of sequential rounds needed to generate an accurate sample?

Indeed, recent approaches \citep{lee2018algorithmic, shen2019randomized,  anari2024fast} have shown that Picard-based discretization methods are naturally parallelizable. In particular, \citet{anari2024fast} shows that the required sequential depth can be reduced to polylogarithmic dependence on \(d\) and \(\varepsilon^{-1}\). At a high level, one takes an integration interval of length \(h=\Omega(1)\) and divides it into \(M\) smaller subintervals. Rather than advancing through these subintervals one at a time, the algorithm introduces approximate path values at the corresponding \(M\) grid points and refines them jointly using Picard iterations. During each Picard round, drift evaluations based on the current path approximation can be performed in parallel across the grid. Thus, the sequential depth is governed by the number of Picard rounds rather than by the number of subintervals. Since Picard iteration converges geometrically under suitable contraction conditions, this can lead to substantially smaller sequential depth. Moreover, because the discretization error is controlled by the "effective step size" \(h/M\), increasing \(M\) improves accuracy without increasing the number of sequential time steps as in a purely sequential method.

However, this speedup comes with an important resource cost. Each of the \(M\) processors must store the relevant state variables and evaluate gradients of the potential during each Picard iteration. In existing guarantees, \(M\) typically has polynomial dependence on \(d\) and \(\varepsilon^{-1}\). Thus, even if the sequential depth is only polylogarithmic, each Picard round can require substantial computational power and memory, especially in high-dimensional regimes. This motivates the central question of our work:

\begin{quote}
\emph{Can one retain the polylogarithmic sequential depth of parallel log-concave sampling while reducing the required parallel resources?}
\end{quote}

We answer this question affirmatively. Our contributions are as follows:

\begin{itemize}
    \item We propose a general higher-order Langevin dynamics framework of arbitrary order $K$. We develop a unified analytical framework for studying these dynamics, extending existing analyses of overdamped and underdamped Langevin algorithms to general order in the parallel sampling setting.

    \item Existing Picard-based parallel sampling schemes typically approximate the drift $\nabla U$ by a constant value over each effective step of size $h/M$. By leveraging higher-order dynamics, we replace piecewise-constant drift approximations on each effective step with Lagrange-polynomial interpolation, improving local accuracy. Consequently, our method requires fewer parallel points than prior Picard-based approaches while preserving polylogarithmic sequential depth. 
\end{itemize}

We analyze two statistically motivated regimes: higher-order smooth potentials, where smooth trajectories enable sharper discretization guarantees \citep{mou2021high}, and ridge-separable potentials, which arise in Bayesian logistic regression and two-layer neural networks \citep{mangoubi2017rapid, mangoubi2018dimensionally, lee2018algorithmic, mou2021high, chen2023does}. In the latter case, the target depends on \(x\in\mathbb{R}^d\) only through \(r\) ridge directions, often with \(r\ll d\), making it possible to obtain more efficient dimension-dependent guarantees. Table~\ref{tab:complexity-comparison} compares our bounds with prior parallel sampling results.

\begin{table}[t]
\caption{
Comparison of adaptive and space complexity across Langevin dynamics results.
In the complexity bounds, \(\widetilde O(\cdot)\) hides logarithmic factors in \(d\) and \(\varepsilon^{-1}\), while \(\widetilde{O}_{\log\log}(\cdot)\) preserves the displayed polylogarithmic dependence while suppressing secondary \(\log\log(d/\varepsilon^2)\)-type
factors when present. Throughout, constants depending only on problem parameters such as \(m,L,L_{\max}\), and \(\Lambda_{\Theta}\) are suppressed.
}
\label{tab:complexity-comparison}
\centering
\small
\setlength{\tabcolsep}{5pt}
\renewcommand{\arraystretch}{1.15}
\begin{tabular}{@{}p{0.39\linewidth}ccc@{}}
\toprule
Method & Metric & Adaptive complexity & Space complexity \\
\midrule

\citet{shen2019randomized}\\
Underdamped Langevin diffusion
& \(W_2\)
& \(\widetilde{O}_{\log\log}\!\left(\log^2\!\frac{d}{\varepsilon^2}\right)\)
& \(\widetilde O\!\left(\frac{d^{3/2}}{\varepsilon}\right)\) \\

\citet{anari2024fast}\\
Underdamped Langevin diffusion
& TV
& \(\widetilde{O}_{\log\log}\!\left(\log^2\!\frac{d}{\varepsilon^2}\right)\)
& \(\widetilde O\!\left(\frac{d^{3/2}}{\varepsilon}\right)\) \\

\citet{anari2024fast}\\
Overdamped Langevin diffusion
& KL
& \(\widetilde{O}_{\log\log}\!\left(\log^2\!\frac{d}{\varepsilon^2}\right)\)
& \(\widetilde O\!\left(\frac{d^2}{\varepsilon^2}\right)\) \\

\citet{zhou2024parallel}\\
Overdamped Langevin diffusion
& KL
& \(\widetilde O_{\log\log}\!\left(\log\!\frac{d}{\varepsilon^2}\right)\)
& \(\widetilde O\!\left(\frac{d^2}{\varepsilon^2}\right)\) \\

\citet{yu2025parallelized}\\
Underdamped Langevin diffusion
& \(W_2\)
& \(\widetilde{O}_{\log\log}\!\left(\log^2\!\frac{d}{\varepsilon^2}\right)\)
& \(\widetilde O\!\left(\frac{d^{3/2}}{\varepsilon}\right)\) \\

\rowcolor{blue!20}
Ours (Higher-order Langevin, Order \(K\))\\
Higher-order smoothness 
& \(W_2\)
& \(\widetilde{O}_{\log\log}\!\left(\log^2\!\frac{d}{\varepsilon^2}\right)\)
& \(\widetilde O\!\left(
d^{3/2}\varepsilon^{-\frac{1}{K-1}}
\right)\) \\

\rowcolor{blue!20}
Ours (Higher-order Langevin, Order \(K\))\\
Ridge-separable, \(r=O(\sqrt d)\)
& \(W_2\)
& \(\widetilde{O}_{\log\log}\!\left(\log^2\!\frac{d}{\varepsilon^2}\right)\)
& \(\widetilde O\!\left(
d^{5/4}\varepsilon^{-\frac{1}{K-1}}
\right)\) \\

\rowcolor{blue!20}
Ours (Higher-order Langevin, Order \(K\))\\
Ridge-separable, \(r=O(\log d)\)
& \(W_2\)
& \(\widetilde{O}_{\log\log}\!\left(\log^2\!\frac{d}{\varepsilon^2}\right)\)
& \(\widetilde O\!\left(
d^{1+\frac{1}{2K-2}}\varepsilon^{-\frac{1}{K-1}}
\right)\) \\

\bottomrule
\end{tabular}
\end{table}

\subsection{Related work}

The work most closely related to ours is the recent study of \citet{dang2025high}, which develops a high-order method based on Taylor expansions combined with a splitting scheme. Their algorithm constructs high-order expansions using derivatives of \(\nabla U\) up to order \(K\), and the resulting method is sequential in nature. In addition, their Condition~H2 imposes a stronger structural assumption on the potential than the assumptions used in our work.

Several recent works study parallel sampling for log-concave distributions using Picard-based discretization schemes. \citet{anari2024fast} study parallel sampling under a log-Sobolev inequality, rather than strong log-concavity, while \citet{shen2019randomized} and \citet{yu2025parallelized} consider smooth log-concave potentials in the Langevin setting. A related Picard-based approach for Hamiltonian Monte Carlo appears in \citet{lee2018algorithmic}. More recently, \citet{zhou2024parallel} focus on improving the sequential depth of parallel sampling from \(\log^2(d/\varepsilon)\) to \(\log(d/\varepsilon)\). These works motivate the use of Picard iterations as a mechanism for reducing sequential depth in sampling algorithms.

There is also a large body of work on sequential log-concave sampling \cite{dalalyan2017theoretical, cheng2018underdamped, durmus2019high, vempala2019rapid, dalalyan2020sampling, ma2021there, srinivasan2025poisson}. Splitting-based approaches \citep{mou2021high,  sanz2021wasserstein, paulin2024correction} improve discretization accuracy, but the resulting sequential complexity typically remains polynomial in \(d\), \(\varepsilon^{-1}\), or both. Another line of work develops high-accuracy samplers that achieve polylogarithmic dependence on \(\varepsilon^{-1}\), while retaining polynomial dependence on \(d\). Prominent examples include analyses of MALA \citep{dwivedi2019log, chen2020fast, wu2022minimax, altschuler2024faster}, proximal sampling methods \citep{wibisono2019proximal, chewi2020exponential, salim2020primal, jiang2021mirror, ahn2021efficient, fan2023improved}, and the algorithmic warm-start framework of \citet{altschuler2024faster}. These methods are sequential, and hence use only \(\Theta(d)\) memory per chain, so they are not directly comparable to our parallel setting; nevertheless, their total number of gradient evaluations provides a useful point of comparison.

Finally, several recent works study Picard-based or parallelized frameworks for diffusion-model sampling \citep{gupta2024faster, chen2024accelerating, li2025faster, gatmiry2026high}. This setting is different from Langevin Monte Carlo: the target score is not available as an oracle, but is instead learned or approximated from data through a trained score model. Consequently, their main sources of error include score-estimation and denoising errors, whereas our focus is on sampling from a known log-concave density given access to \(\nabla U\). These works are therefore complementary to the framework studied here.

\subsection{Organization}
The remainder of the paper is organized as follows. 
Section~\ref{sec:prelim} introduces the notation and assumptions that will be used throughout. We then develop the higher-order Langevin dynamics and present the discretization scheme in Section~\ref{Sec: Higher_Order_Langevin_Dynamics}. Our main convergence results are stated in Section~\ref{sec:main}, with a brief outline of the proof strategy. Finally, Section~\ref{sec:discussion} concludes with a discussion, and technical lemmas are deferred to the appendix.

\section{Preliminaries}
\label{sec:prelim}

\begin{definition}[Adaptive Complexity and Space Complexity]
For a sampling algorithm, we measure \emph{adaptive complexity} as the number of sequential, non-parallelizable evaluations of \(\nabla U \). We use \emph{space complexity} to describe the amount of memory required during sampling. Following prior work, space complexity is measured in words rather than bits (see \cite{chen2024accelerating}).
\end{definition}

We denote the standard Euclidean norm by $\|\cdot\|$ and the inner product 
by $\langle \cdot, \cdot \rangle$. 
To quantify higher-order derivatives, we next introduce the notion of a tensor spectral norm.

\begin{definition}[Tensor spectral norm]
For a vector–valued $k$-th order tensor $T$, we define
\begin{equation*}
\|T\|_{\mathrm{tsr}}
:= \sup_{v_1,\ldots,v_{k-1}\in S^{d-1}}
   \big\|T\cdot[v_1,\ldots,v_{k-1}]\big\|_2 ,
\end{equation*}
where $S^{d-1}$ is the unit sphere in $\mathbb{R}^d$, and 
``$\cdot$'' denotes the natural tensor contraction.
\end{definition}

\paragraph{Assumptions on $U$}
We make the following assumptions on the potential function $U$ in \eqref{targetfun:p*}.

\begin{assumption}[Strong convexity and smoothness]
\label{assump:U-strong-smooth}
There exist positive constants $m\le L$ such that
\begin{align*}
\frac{m}{2}\,\|x' - x\|^2
&\;\le\;
U(x') - U(x)
  - \langle \nabla U(x), x' - x\rangle
\;\le\;
\frac{L}{2}\,\|x' - x\|^2 ,
\qquad \forall\, x,x' \in \mathbb{R}^d .
\end{align*}
\end{assumption}

Assumption~\ref{assump:U-strong-smooth} is equivalent to $m I_d \preceq \nabla^2 U(x) \preceq L I_d$.

\begin{assumption}[Centered potential]
\label{assump:centered}
Without loss of generality, assume $\nabla U(0)=0$ and $U(0)=0$, 
which can be enforced by shifting $U$.
\end{assumption}

\paragraph{Convergence metrics and notations}
Consider an iterative algorithm that generates a random vector $\widehat{X}(nh)$ at the $n$-th step, corresponding to time $t=nh$ with step size $h$. Let $\hat{\pi}^{(n)}$ denote the law of $\hat{X}(nh)$. We study convergence $\hat{\pi}^{(n)} \to \pi$, where $\pi$ is the target measure with density $p^*$. The distance between two measures is quantified using the \emph{Wasserstein-2} distance:
\begin{equation*}
\mathcal{W}^2_2(p,q) = \inf_{\gamma \in \Gamma(p,q)}
         \int_{\mathbb{R}^{Kd} \times \mathbb{R}^{Kd}}
           \|x-y\|^2 \, d\gamma(x,y),
\end{equation*}
where $\Gamma(p,q)$ is the set of couplings of $p$ and $q$, with each coupling $\gamma$ being a joint distribution whose marginals are $p$ and $q.$

We use standard asymptotic notation to express convergence rates: for $f,g:\mathbb{R}^n \to [0,\infty)$, we write 
$f=O(g)$ if $\exists\,C>0$ with $f(x)\le C g(x)$, 
$f=\Theta(g)$ if $\exists\,c_1,c_2>0$ with $c_1 g(x) \leq f(x)\leq  c_2 g(x)$, 
and $f=\tilde{O}(g)$ if $\exists\,C_1,C_2>0$ with 
$f(x)\le C_1 g(x)\bigl(\log(1+g(x))\bigr)^{C_2}$ (correspondingly for \(\widetilde \Theta\)), all for sufficiently large $\|x\|$. Unless stated otherwise, the implicit constants may depend on fixed problem parameters such as \(m,L,L_{\max}, \Lambda_{\Theta}, K\), and on norm-equivalence constants associated with the positive definite Lyapunov matrix \(S\), all of which are introduced below.

\section{Higher-order Langevin dynamics and its discretization}
\label{Sec: Higher_Order_Langevin_Dynamics}

We begin with the family of SDEs of the form
\begin{equation}\label{eq:SDE-D+Q}
 \mathrm{d}X(t) \;=\; -(D + Q)\,\nabla H(X(t))\,\mathrm{d}t \;+\; \sqrt{2 D}\,\mathrm{d}B_t,
\end{equation}
where $D$ is a constant positive semidefinite matrix and $Q$ is a constant skew-symmetric matrix. It can be shown \citep{shi2012relation, ma2015complete} that for any such choice of $D$ and $Q$, the SDE \eqref{eq:SDE-D+Q} admits 
\[
p^*(x) \;\propto\; \exp\,(\,-H(x))
\]
as the invariant distribution. Notably, both the underdamped Langevin dynamics \citep{cheng2018underdamped} and the 
more recent third-order Langevin dynamics \citep{mou2021high} are special cases of this framework. Here we consider general $K$-th order dynamics. For arbitrary $K \ge 2$, 
we expand the ambient space as $X = (X_1, X_2, \ldots, X_K),$ where $ X_i \in \mathbb{R}^d$ and
 $X_1$ denotes the variable of interest, and define
$H(X) = U(X_1) + \tfrac{1}{2} \sum_{i=2}^K \|X_i\|^2.$ Also, introduce the matrices
\begingroup
\setlength{\arraycolsep}{5pt}
\renewcommand{\arraystretch}{0.95}
\begin{align}
D =
\bigl(\operatorname{diag}(0,\ldots,0,\gamma)_{K\times K}\otimes I_d\bigr), \quad
Q =
\begin{bmatrix}
0      & -\,1     & 0        & \cdots & 0 \\
1      & 0        & -\,\gamma & \ddots & \vdots \\
0      & \gamma   & 0        & \ddots & 0 \\
\vdots & \ddots   & \ddots   & \ddots & -\,\gamma \\
0      & \cdots   & 0        & \gamma & 0
\end{bmatrix}_{K\times K}
\otimes I_d ,\label{eq:D-and-Q}
\end{align}
\endgroup
for some $\gamma >0$. The resulting $K$-th order Langevin dynamics on $\mathbb{R}^{Kd}$ take the form
\begin{equation}
\mathrm{d} X(t) = -\,\bigl(D+Q\bigr)
   \begin{pmatrix}
     \nabla U(X_{1}(t))\\
     X_{2}(t)\\
     \vdots\\
     X_{K}(t)
   \end{pmatrix} dt + \sqrt{2D}\, \mathrm{d} B_t, 
\label{eq: high_order_langevin}
\end{equation}
and the marginal law of $X_1$ under its invariant distribution is the target measure. Thus, the framework in \eqref{eq: high_order_langevin} provides a unified formulation of higher-order Langevin dynamics, generalizing both underdamped and third-order cases.

\subsection{Discretization} \label{sec:discretization}

We briefly describe the discretization used by our parallel sampler; full details, as well as the intuition behind the construction, are given in Appendix~\ref{appendix-B}. Consider one step over \([0,h]\) for \eqref{eq: high_order_langevin}:
\begin{equation}
    dX(t)=AX(t)\,dt+g(X(t))\,dt+\sqrt{2D}\,dB_t,
    \label{eq:main-high-order-sde}
\end{equation}
where \(J := \operatorname{diag}(0,1,\ldots,1)\otimes I_d\), \(A:=-(D+Q)J\) and \(g(X):=-(e_2\otimes I_d)\nabla U(X_1)\). Let \(\Delta:=h/M\), partition \([0,h]\) into blocks
\[
B_i=[i\Delta,(i+1)\Delta],
\qquad i=0,\ldots,M-1,
\]
and place \(K-1\) equispaced nodes
\[
t_{i,j}=i\Delta+c_j\Delta,
\qquad
c_j=\frac{j-1}{K-2},
\qquad j=1,\ldots,K-1 ,
\]
with the convention \(c_1=0\) when \(K=2\). For \(t\in[0,\Delta]\), the variation-of-constants formula on block \(B_i\) gives
\begin{equation}
X(t_{i,1}+t)
=
e^{tA}X(t_{i,1})
+
\int_{t_{i,1}}^{t_{i,1}+t}
e^{(t_{i,1}+t-s)A}g(X(s))\,ds
+
W_{i,1}(t),
\label{eq:main-local-voc}
\end{equation}
where
\[
W_{i,1}(t)
:=
\int_{t_{i,1}}^{t_{i,1}+t}
e^{(t_{i,1}+t-s)A}\sqrt{2D}\,dB_s .
\]
For \(t,t'\in[0,\Delta]\),
\begin{equation}
\operatorname{Cov}\bigl(W_{i,1}(t),W_{i,1}(t')\bigr)
=
2\int_0^{t\wedge t'}
e^{(t-s)A}D e^{(t'-s)A^\top}\,ds .
\label{eq: sample_gaussian}
\end{equation}

For each block \(B_i\), we evaluate \eqref{eq:main-local-voc} at the collocation times
\(t=c_j\Delta\), replacing \(g(X(s))\) by its Lagrange interpolant on \(B_i\),
while the Gaussian increments
\(\{W_{i,1}(t_{i,j}-t_{i,1})\}_{j=1}^{K-1}\) are sampled jointly according to
the centered Gaussian law with covariance \eqref{eq: sample_gaussian}. Specifically, on each block we use the following approximation of \(g(X(s))\):
\[
P_{B_i}(s;\widehat X)
:=
\sum_{\ell=1}^{K-1}
\ell_{i,\ell}(s)\,
g\!\left(\widehat X(t_{i,\ell})\right),
\qquad s\in B_i .
\]
Freezing this interpolant at the previous Picard iterate gives the local update (see Fig.~\ref{fig:Algorithm})
\begin{equation}
\widehat X^{[k]}(t_{i,j})
=
e^{(t_{i,j}-t_{i,1})A}
\widehat X^{[k]}(t_{i,1})
+
\int_{t_{i,1}}^{t_{i,j}}
e^{(t_{i,j}-s)A}
P_{B_i}(s;\widehat X^{[k-1]})\,ds
+
W_{i,1}(t_{i,j}-t_{i,1}).
\label{eq:main-local-picard-update}
\end{equation}
Equivalently, unrolling the block recursion gives the global form
\begin{equation}
\widehat X^{[k]}(t_{i,j})
=
e^{t_{i,j}A}\widehat X(0)
+
\sum_{m=0}^{i}
\sum_{\ell=1}^{K-1}
\alpha_{i,j}^{m,\ell}
g\!\left(\widehat X^{[k-1]}(t_{m,\ell})\right)
+
\sum_{m=0}^{i}\widetilde W_{i,j}^{m},
\qquad k=1,\ldots,\nu ,
\label{eq:main-global-picard-update}
\end{equation}
where
\begin{equation}
\alpha_{i,j}^{m,\ell}
:=
\int_{B_m\cap[0,t_{i,j}]}
e^{(t_{i,j}-s)A}
\ell_{m,\ell}(s)\,ds ,
\label{eq:main-global-alpha}
\end{equation}
and
\[
\widetilde W_{i,j}^{m}
:=
e^{(t_{i,j}-t_{m,K-1})A}
W_{m,1}(t_{m,K-1}-t_{m,1}),
\quad m<i,
\qquad
\widetilde W_{i,j}^{i}
:=
W_{i,1}(t_{i,j}-t_{i,1}).
\]

For equal-length blocks and fixed reference nodes, the global weights
\(\alpha_{i,j}^{m,\ell}\) depend only on \(A\), \(\Delta\), and the relative source-target node locations, and can therefore be precomputed. The Gaussian increments are sampled independently across blocks and jointly within each block using the covariance formula above. Thus, each Picard round can be implemented in parallel over the collocation grid, as summarized in Algorithm~\ref{alg:picard_sampler}. Moreover, the drift matrix \(A\) inherits the Kronecker structure of \(D,Q\), and \(J\): it acts only through a \(K\times K\) block matrix tensored with \(I_d\). Consequently, the matrix exponential and covariance computations are carried out in this small block space and then applied coordinatewise across the \(d\) dimensions. Thus, for fixed order \(K\), each Picard round uses \(M\) parallel processors and has space complexity \(O(Md)\), coming from storing the blockwise collocation states and gradients.

\begin{figure}[t]
    \centering
    \includegraphics[width=0.9\linewidth]{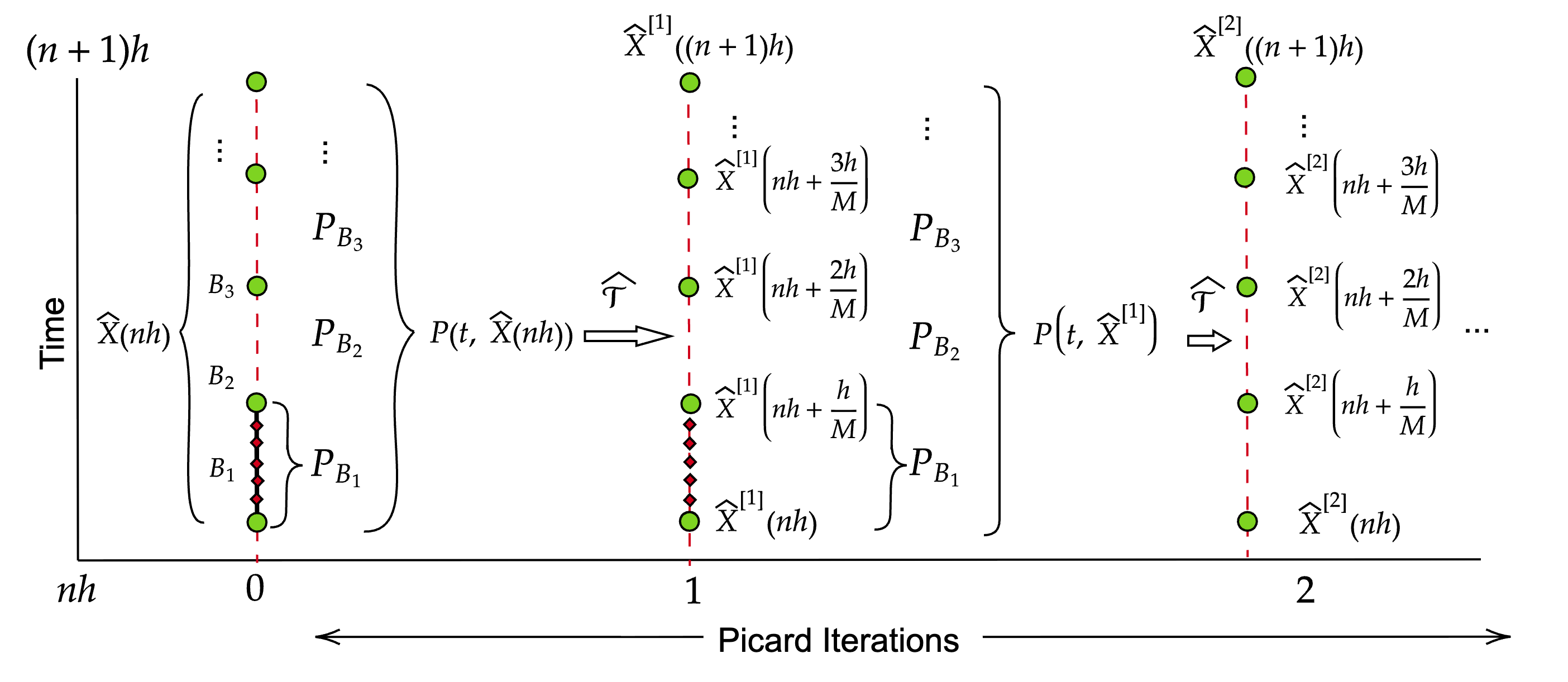}
    \caption{Schematic representation of Picard iterations. 
The procedure begins with constant values, which are used to form a 
Lagrange interpolant. Applying the update operator $\widehat {\mathcal{T}}$ at the 
interpolation nodes yields the next set of  values, which then feed 
into the following Picard iteration.}
    \label{fig:Algorithm}
\end{figure}

\begin{algorithm}[t]
\caption{Blockwise higher-order Langevin Monte Carlo}
\label{alg:picard_sampler}
\begin{algorithmic}[1]
\REQUIRE Gradient map \(\nabla U\), order \(K\), initial state \(\widehat X(0)\in\mathbb{R}^{Kd}\), number of steps \(N\), step size \(h\), number of blocks \(M\), Picard iterations \(\nu\), friction parameter \(\gamma\).
\STATE Partition \([0,h]\) into blocks \(B_i=[i h/M,(i+1)h/M]\), \(i=0,\ldots,M-1\), and place \(K-1\) collocation nodes \(t_{i,j}\) on each block.
\STATE Precompute the weights \(\alpha_{i,j}^{m,\ell}\) in \eqref{eq:main-global-alpha}.
\FOR{\(n=0,\ldots,N-1\)} 
   \FOR{\(m=0,\ldots,M-1\) \textbf{in parallel}}
    \STATE Set \(\widehat X^{[0]}(nh+t_{m,j})\gets \widehat X(nh)\) for \(j=1,\ldots,K-1\), and jointly sample \(\{W_{m,1}(t_{m,j}-t_{m,1})\}_{j=1}^{K-1}\) using \eqref{eq: sample_gaussian}.
\ENDFOR
    \FOR{\(k=1,\ldots,\nu\)}
        \FOR{\(i=0,\ldots,M-1\), \(j=1,\ldots,K-1\) \textbf{in parallel}}
            \STATE Compute \(\widehat X^{[k]}(nh+t_{i,j})\) using \eqref{eq:main-global-picard-update}.
        \ENDFOR
    \ENDFOR
    \STATE Set \(\widehat X((n+1)h)\gets \widehat X^{[\nu]}(nh+h)\).
\ENDFOR
\RETURN \(\widehat X(Nh)\).
\end{algorithmic}
\end{algorithm}

\section{Main results} \label{sec:main}

In this section, we present informal versions of our main convergence theorems for the Higher-order Langevin Monte Carlo method in Algorithm~\ref{alg:picard_sampler}. First, we state the smoothness condition used to control the blockwise Lagrange interpolation error.

\begin{assumption}[Higher–order smoothness]
\label{assump:higher-smooth}
There exist constants $L_1,\ldots,L_{K-1}$ such that
\begin{equation*}
\big\|D^i \nabla U(x)\big\|_{\mathrm{tsr}} \;\le\; L_i,
\qquad \forall\, x\in\mathbb{R}^d,\; i\le K-1,
\end{equation*}
that is, the $i$-th order derivatives of $\nabla U$ are uniformly bounded in tensor spectral norm.
\end{assumption}

The effectiveness of higher-order polynomial approximations relies on higher-order smoothness of the underlying function. Assumption~\ref{assump:higher-smooth} ensures this  by bounding the higher-order derivatives of $U$, which in turn controls the Lagrange-interpolation remainder when approximating the drift.

\begin{theorem}\label{thm:main-W2-simplified}
Assume that \(U\) satisfies Assumptions~\ref{assump:U-strong-smooth}--\ref{assump:higher-smooth}, and fix \(K \ge 2\).
Run Algorithm~\ref{alg:picard_sampler} with \(M\) parallel blocks using \(K-1\) equispaced nodes and \(\nu\) Picard iterations per step.  Then, for any target accuracy \(\varepsilon \in (0,1)\), it suffices to choose
\[
h = \Theta(1), \qquad
\nu = \Theta(\log M), \qquad
M = \widetilde{\Theta}\!\left(\sqrt{d}\,\varepsilon^{-\frac{1}{K-1}}\right), \qquad
N = \Theta\!\left(\log\!\Bigl(\frac{d}{\varepsilon^2}\Bigr)\right),
\]
to guarantee
\[
W_2^2\bigl(\hat\pi^{(N)}, \pi\bigr)
\;\le\; \varepsilon^2.
\]
\end{theorem}

\begin{remark}
Our method retains polylogarithmic adaptive complexity while improving the accuracy dependence in the parallel space complexity to \(\varepsilon^{-1/(K-1)}\), without worsening the dimension dependence compared with existing parallel samplers. See Table~\ref{tab:complexity-comparison}.
\end{remark}

\begin{remark}
For underdamped Langevin dynamics ($K=2$), the best known query complexity is 
$\widetilde O(d^{1/3}\,\varepsilon^{-2/3})$~\citep{shen2019randomized}. 
The third-order scheme of \citet{mou2021high}, under $\alpha$-th order smoothness assumptions, 
achieves $\widetilde O(d^{1/4}\,\varepsilon^{-1/2} + d^{1/2}\varepsilon^{-1/(\alpha-1)})$. 
Our general result shows that for every $K > 3$, the proposed sampler 
achieves $\varepsilon$-accuracy within
$
NM\nu = \widetilde{\Theta} \,(d^{\frac{1}{2}}\,\varepsilon^{-\frac{1}{K-1}})
$
gradient evaluations, thereby improving the dependence on the accuracy 
parameter~$\varepsilon$ compared to existing Langevin-based methods. 
The dimension exponent is suboptimal relative to the $d^{1/3}$
scaling achieved by \cite{shen2019randomized} for the underdamped Langevin dynamics.
\end{remark}

\begin{remark}
\label{rem:kappa_dependence}
The condition-number dependence in our bounds is implicit mainly in the continuous-time contraction estimate. We work with the existence of a positive definite Lyapunov matrix \(S\), rather than an explicit closed-form construction, and the constants \(C_S\) and the associated norm-equivalence factors may depend on \(\kappa=L/m\). Making this dependence explicit is an interesting direction for future work.
\end{remark}

Higher-order smoothness can be restrictive in practical high-dimensional models because it requires control of full derivative tensors in all directions. We therefore also study ridge-separable potentials, a practically important class that appears in logistic regression, generalized linear models, and many Bayesian posterior inference problems. Unlike the general higher-order smoothness assumption, which controls full derivative tensors in arbitrary directions of \(\mathbb R^d\), ridge-separable structure assumes that the nonlinear part of the potential depends on \(x\) only through projections \(\theta_k^\top x\). Thus the main nonlinear variation is concentrated along \(r\) directions \(\{\theta_k\}_{k=1}^r\), while the ambient dimension \(d\) may be much larger. In statistical settings, \(r\) often corresponds to the number of observations, so regimes with \(r\ll d\) arise naturally. This motivates the following assumption.

\begin{assumption}[Ridge-separable potential]\label{assump:ridge-separable}
Let \(\boldsymbol{\Theta}=\{\theta_1,\ldots,\theta_r\}\subset S^{d-1}\), with \(r\ll d\). We assume that
\[
U(x)=U_{\mathrm{ridge}}(x)+\frac{\eta}{2}\|x\|^2,
\qquad
U_{\mathrm{ridge}}(x)=\sum_{k=1}^r \phi_k(\theta_k^\top x),
\]
where each \(\phi_k:\mathbb R\to\mathbb R\) is \(K-1\) times continuously differentiable, and there are constants \(L_1,\ldots,L_{K-1}\), independent of \(d\) and \(r\), such that
\[
\max_{1\le k\le r}\|\phi_k^{(j)}\|_\infty < L_j,
\qquad j=1,\ldots,K-1.
\]
Moreover, let \(\Theta\in\mathbb R^{r\times d}\) be the matrix whose \(k\)-th row is \(\theta_k^\top\). We assume
\[
\lambda_{\max}(\Theta\Theta^\top)\le \Lambda_\Theta,
\]
where \(\Lambda_\Theta\) is independent of \(d\) and \(r\). 
\end{assumption}
\begin{remark}
The bounded derivative condition in Assumption~\ref{assump:ridge-separable} is mild in common examples; logistic regression, for instance, satisfies it since the derivatives of the logistic loss are uniformly bounded. The spectral condition on \(\Theta\) is also natural: Typically, \(\lambda_{\max}(\Theta\Theta^\top)\) is of the same order as the smoothness constant \(L\). In particular, when the ridge directions are sampled independently and uniformly from \(S^{d-1}\), standard random matrix concentration implies \( \lambda_{\max}(\Theta\Theta^\top)=O(1) \)
with high probability in the regime \(r\ll d\).
\end{remark}

\begin{theorem}\label{thm:main-W2-simplified-ridge}
Assume that \(U\) satisfies Assumptions~\ref{assump:U-strong-smooth}--\ref{assump:centered} and \ref{assump:ridge-separable}, and fix \(K \ge 2\).
Run Algorithm~\ref{alg:picard_sampler} with \(M\) parallel blocks using \(K-1\) equispaced nodes and \(\nu\) Picard iterations per step.  Then, for any target accuracy \(\varepsilon \in (0,1)\), it suffices to choose
\[
h = \Theta(1), \quad
\nu = \Theta(\log M), \quad
M = \widetilde{\Theta}\!\left(
\max\left\{\sqrt r,\; d^{\frac{1}{2K-2}}\right\}
\varepsilon^{-\frac{1}{K-1}}
\right), \quad
N = \Theta\!\left(\log\!\Bigl(\frac{d}{\varepsilon^2}\Bigr)\right),
\]
to guarantee
\[
W_2^2\bigl(\hat\pi^{(N)}, \pi\bigr)
\;\le\;\varepsilon^2.
\]
\end{theorem}

\begin{remark}
The processor requirement in Theorem~\ref{thm:main-W2-simplified-ridge} interpolates between the number of ridge directions and the ambient dimension. In particular, when \(r\ll d\), the dimension-dependent factor improves from the generic \(\sqrt d\) scaling to \( \max\{\sqrt r,\ d^{1/(2K-2)}\}.\) Thus, in the worst case over \(r\), the ridge-separable bound never exceeds the \(\sqrt d\) scaling, while for small \(r\) it can be substantially smaller. Table~\ref{tab:complexity-comparison} compares the space complexity for \(r=O(\sqrt d)\) and \(r=O(\log d)\). These choices also imply that the query complexity of our sampler is \(\widetilde{\Theta}\!\left(d^{1/4}\varepsilon^{-\frac{1}{K-1}}\right)\) and \(\widetilde{\Theta}\!\left(d^{\frac{1}{2K-2}}\varepsilon^{-\frac{1}{K-1}}\right)\), respectively.
\end{remark}

\paragraph{Proof overview}
\label{sec:proof_outline}

We give a short overview of the proof of Theorem~\ref{thm:main-W2-simplified}. The proof of the ridge-separable result, Theorem~\ref{thm:main-W2-simplified-ridge}, follows the same decomposition, with only the interpolation-error bound replaced by its ridge-separable analogue. The proof is based on a path-space operator view. For an initial point \(y\), let
\(\mathcal T_y\) denote the map whose fixed point is the exact higher-order
Langevin path on \([0,h]\). In other words, \(X_y^\ast=\mathcal T_y[X_y^\ast]\)
is simply the continuous-time solution initialized at \(y\). Similarly, let
\(\widehat{\mathcal T}_y\) denote the discretized operator obtained by replacing
\(\nabla U\) with its blockwise Lagrange interpolant, and write
\(\widehat X_y^\ast=\widehat{\mathcal T}_y[\widehat X_y^\ast]\) for its fixed
point. The Picard iterations used in Algorithm~\ref{alg:picard_sampler} are precisely
iterations of \(\widehat{\mathcal T}_y\). Starting from
\(y=\widehat X(nh)\), the algorithm runs \(\nu\) Picard steps and outputs
\[
\widehat X((n+1)h)=\widehat X_y^{[\nu]}(h).
\]
Appendix~\ref{Appendix-F2} shows that this operator formulation is equivalent to
Algorithm~\ref{alg:picard_sampler}.

Let \(E_n:=\mathbb E\|X(nh)-\widehat X(nh)\|_S^2\), where \(S\succ0\) is the Lyapunov matrix used for the continuous-time contraction. Under a synchronous coupling and with \(X(0)\sim\pi\), where \(\pi\) denotes the target distribution with density \(p^\ast\), we have
\[
W_2^2(\pi,\widehat \pi^{(N)})
\le
\|S^{-1}\|_{\mathrm{op}}\,E_N .
\]
The key step is a one-step decomposition. Adding and subtracting the exact and discretized fixed-point paths initialized at \(\widehat X(nh)\) gives

\begin{align}
E_{n+1}
\leq&
(1+c h)\,
\underbrace{
\mathbb E\|X^\ast_{X(nh)}(h)-X^\ast_{\widehat X(nh)}(h)\|_S^2
}_{\text{(I) continuous-time contraction}}
+
\left(2+\frac{2}{c h}\right)
\underbrace{
\mathbb E\|X^\ast_{\widehat X(nh)}(h)-\widehat X^\ast_{\widehat X(nh)}(h)\|_S^2
}_{\text{(II) interpolation error}}
\nonumber\\
&+
\left(2+\frac{2}{c h}\right)
\underbrace{
\mathbb E\|\widehat X^\ast_{\widehat X(nh)}(h)-\widehat X_{\widehat X(nh)}^{[\nu]}(h)\|_S^2
}_{\text{(III) Picard error}},
\label{eq:one-step-ED}
\end{align}
for some constant \(c>0\). Each term is controlled separately. Term (I) is bounded using exponential contraction of the higher-order Langevin dynamics in the \(S\)-norm:
\[
\mathbb E\|X^\ast_{X(nh)}(h)-X^\ast_{\widehat X(nh)}(h)\|_S^2
\le e^{-2C_S h}E_n .
\]
Term (II) is controlled by the blockwise Lagrange interpolation remainder. The error is governed by
\[
I_U
:=
\sup_{s\in[nh,(n+1)h]}
\|\nabla U(\widehat X^\ast_1(s))-P(s;\widehat X^\ast_1)\|^2,
\]
and the higher-order smoothness assumption yields
\[
\mathbb E I_U \le C_{\mathrm{IR}}\,\frac{N h^{2K-2}}{M^{2K-2}}\,d^{K-1}.
\]
Term (III) is controlled by the contraction of the discretized Picard operator:
\[
\|\widehat{\mathcal T}_y[X]-\widehat{\mathcal T}_y[Y]\|_\infty
\le \rho\|X-Y\|_\infty,
\qquad \rho=O(h)<1,
\]
which gives a geometric Picard error of order
\[
\mathbb E\sup_{t\in[0,h]}
\|\widehat X^\ast(t)-\widehat X^{[\nu]}(t)\|^2
\lesssim
\frac{\rho^{2\nu}}{(1-\rho)^2}(h^2E_n+dh).
\]

Combining the three estimates yields a one-step recursion for \(E_n\). Solving this recursion over \(n=0,\ldots,N\), and then choosing \(h\), \(M\), \(N\), and \(\nu\) so that the contraction, interpolation, and Picard errors are each at most \(\varepsilon^2/3\), proves Theorem~\ref{thm:main-W2-simplified}.

\section{Discussion}
\label{sec:discussion}

We introduced a blockwise higher-order Langevin Monte Carlo method that combines
arbitrary-order Langevin dynamics with Picard iteration and Lagrange
interpolation. Our results show that one can retain polylogarithmic adaptive
complexity while reducing the number of parallel blocks required by existing
Picard-based sampling schemes. The improvement comes from exploiting
higher-order structure to obtain sharper local discretization error on each
effective step. Our analysis applies both to general higher-order smooth potentials and to
ridge-separable potentials. In the ridge-separable case, the bounds can exploit
the number of ridge directions \(r\), yielding sharper dimension dependence
when \(r\ll d\). This suggests that parallel higher-order Langevin methods may
be especially useful for structured statistical models, including Bayesian
generalized linear models and related posterior sampling problems.

There are several limitations and trade-offs. First, while the generic theorem requires ambient higher-order smoothness, our ridge-separable result shows that this requirement can be weakened for structured potentials. Extending this type of structure-aware analysis beyond ridge-separable models remains an important direction for future work. Second, the condition-number dependence in our bounds is implicit through the continuous-time contraction estimate and the Lyapunov matrix \(S\). Making this dependence explicit is an interesting direction for future work.

\newpage
\bibliographystyle{apalike}
\bibliography{ref}

\newpage

\appendix

\section*{Appendix roadmap}
\label{app:roadmap}

The appendix is organized as follows. Appendix~\ref{appendix-A} recalls the
continuous and discretized operators used throughout the analysis.
Appendix~\ref{appendix-B} gives additional details for
Algorithm~\ref{alg:picard_sampler}, including the blockwise collocation system,
the Gaussian increments, and the equivalence between the algorithm and Picard
iterates of the discretized operator. Appendix~\ref{Appendix-C} proves the
continuous-time \(S\)-norm contraction. Appendix~\ref{Appendix-D} proves the
main convergence theorems by combining the one-step error decomposition with
the continuous-time contraction, interpolation-error, and Picard-error bounds.
Appendix~\ref{Appendix-E} contains the technical estimates needed for the
blockwise interpolation error. Appendix~\ref{Appendix-F} proves the existence and contraction of
the Picard fixed point and controls the finite-iteration error. Appendix~\ref{Appendix-G} provides the technical bounds used in
Appendix~\ref{Appendix-E} for the interpolation error, including derivative and
moment bounds for the fixed-point path. Appendix~\ref{Appendix-H} presents the ridge-separable refinement, which improves
the interpolation term while reusing the same continuous-time contraction and
Picard arguments. Finally, Appendix~\ref{Appendix-I} collects auxiliary lemmas
used throughout the proof.

\section{Settings and notations}
\label{appendix-A}
We begin by defining the main operators used throughout the appendix.

\paragraph{Continuous time operator} \(\mathcal T_y:\mathcal{C}\!\left([0,h],\mathbb{R}^{Kd}\right)\to \mathcal{C}\!\left([0,h],\mathbb{R}^{Kd}\right)\):
\begin{equation}
    \big(\mathcal{T}_y [X]\big)(t)
        := y - \int_{0}^{t} (D+Q)\,\nabla H\big(X(s)\big)\,\mathrm{d} s
        +
        \int_{0}^{t} \sqrt{2D}\,\mathrm{d} B_s,
        \quad 0 \le t \le h . \label{eq:operator_T}
\end{equation}

\paragraph{Discretized operator} \(\widehat{\mathcal T}_y:\mathcal{C}\!\left([0,h],\mathbb{R}^{Kd}\right)\to \mathcal{C}\!\left([0,h],\mathbb{R}^{Kd}\right)\), where for given process $X$, $\widehat{\mathcal T}_y[X]$ is the solution $\tilde{X}$ to the following SDE:
\begin{equation}
\tilde{X}(t)
        = y + \int_{0}^{t} A\,\tilde{X}(s)\,ds \;-\; \int_{0}^{t} P(s;X)\,ds \;+\; \int_{0}^{t} \sqrt{2D}\,dB_s,
        \label{eq: operator_hat_T2}
\end{equation}
where the blockwise Lagrange interpolant \(P(s;X)\) is defined by
\[
P(s;X)
:=
\sum_{i=0}^{M-1}\mathbf 1_{\{s\in B_i\}}\,P_{B_i}(s;X),
\]
with
\[
P_{B_i}(s;X)
:=
\sum_{\ell=1}^{K-1}
\ell_{i,\ell}(s)\,
g\bigl(X(t_{i,\ell})\bigr),
\qquad s\in B_i ,
\]
$J:=\mathrm{diag}(0,1,\dots,1)\otimes I_d\in\mathbb{R}^{Kd\times Kd}$, $e_2=(0,1,0,\ldots,0)^\top\in\mathbb{R}^{K}$, $A\;:=\;-\,(D+Q)\,J$, and \(g(X):=-(e_2\otimes I_d)\nabla U(X_1)\). Moreover, we denote the fixed point of $\mathcal{T}_y$ and $\widehat {\mathcal{T}}_y$ by $X^\ast_y$ and $\widehat X^\ast_y$ correspondingly.

\paragraph{Picard iterates}
Given \(y\), define the Picard sequence on \([0,h]\) by
\[
\widehat X_y^{[0]}(t)\equiv y,\quad
\widehat X_y^{[\nu+1]} := \widehat{\mathcal T}_y\big[\widehat X_y^{[\nu]}\big],\ \nu\ge0.
\]
Starting from \(y=\widehat X(nh)\), Algorithm~\ref{alg:picard_sampler} outputs the next state by evaluating the \(\nu\)-th Picard iterate at the end of the step:
\(\widehat X((n{+}1)h)=\widehat X_y^{[\nu]}(h).\)

\paragraph{Additional notations}
For either Assumption~\ref{assump:higher-smooth} or Assumption~\ref{assump:ridge-separable}, define
\[
L_{\max}:=\max\left\{1,\max_{1\le i\le K-1}L_i\right\}.
\]

Unless otherwise indicated, we write \(a \lesssim b\) to mean that there exists a constant \(C>0\) such that \(a\le Cb\). The constant \(C\) may depend on the fixed problem parameters \(\gamma,m,L,L_{\max}, \Lambda_{\Theta},K\) and the norm-equivalence constants associated with \(S\), but is independent of
\(r, d,N,M,\nu,\varepsilon\), and the step size \(h\). Here, \(N\) is the total number of steps in
Algorithm~\ref{alg:picard_sampler}, \(M\) is the number of blocks, \(\nu\) is
the number of Picard iterations, \(h\) is the step size, and \(d\) is the
dimension of the target distribution \(\pi\propto e^{-U(x)}\). In the
ridge-separable setting, \(\Theta\in\mathbb R^{r\times d}\) denotes the data
matrix with rows \(\theta_1^\top,\ldots,\theta_r^\top\), and \(
\Lambda_\Theta:=\lambda_{\max}(\Theta\Theta^\top).
\)

For a vector \(x \in \mathbb{R}^d\), we use \(\|x\|\) and $\|x\|_2$ to denote its Euclidean norm. 
For a matrix \(A \in \mathbb{R}^{M\times N}\), we use \(\|A\|\) to denote the operator norm induced by the Euclidean norm, i.e.
$\|A\| = \sup_{x \neq 0} \frac{\|Ax\|}{\|x\|}$. For any positive semidefinite matrix \(S\) and vector \(x \in \mathbb{R}^d\), we define the $S$-norm as follows:
$\|x\|_S := \sqrt{x^\top S x}$.
For a matrix \(A\), we denote by \(\lambda_{\min}(A)\), \(\lambda_{\max}(A)\), and \(\lambda(A)\) its smallest eigenvalue, largest eigenvalue, and the set of all its eigenvalues, respectively.
The symbols \(\preceq\) and \(\succeq\) denote the Loewner order between matrices, 
and \(\prec\) and \(\succ\) denote the corresponding strict inequalities.

For two probability distributions \(p\) and \(q\), we denote by \(\Gamma(p, q)\) the set of all their couplings, i.e., the set of joint distributions on the product space whose marginals are \(p\) and \(q\), respectively.
For a complex number \(x \in \mathbb{C}\), \(\mathrm{Re}(x)\) denotes its real part.

For the Lagrange basis $\{\tilde\ell_j\}_{j=1}^{K-1}$ defined on interval $[0,1]$, we define the standard Lebesgue constant
\[
\Gamma_\phi \;:=\; \sup_{\tau\in[0,1]}\sum_{j=1}^{K-1} \big|\tilde\ell_j(\tau)\big|.
\]

\section{Additional details of Algorithm~\ref{alg:picard_sampler}}
\label{appendix-B}

In this section, we analyze a single update of the algorithm. For simplicity, we focus on the step over the interval \([0,h]\). Recall that the higher-order Langevin dynamics \eqref{eq:main-high-order-sde} can be written as
\begin{align}
    dX(t) \;=\; A\,X(t)\,dt \;+\; g(X(t))\,dt \;+\; \sqrt{2\,D}\,dB_t.\; \label{eq-appendix:high_order_langevin2}
\end{align}
We partition the interval $[0,h]$ into $M$ blocks
\[
B_i := \Bigl[\frac{ih}{M},\, \frac{(i+1)h}{M}\Bigr], \qquad i=0,\dots,M-1,
\]
and on each block $B_i$ we place $K-1$ equispaced collocation nodes
\[
t_{i,j} := \frac{ih}{M} + \frac{j-1}{K-2}\frac{h}{M}, 
\qquad j=1,\dots,K-1.
\]

For \(t\in[0,h/M]\), the variation-of-constants formula on the block \(B_i\) starting at
\(t_{i,1}\) gives
\begin{equation}
 X(t_{i,1}+t)
=
e^{tA}X(t_{i,1})
+
\int_{t_{i,1}}^{t_{i,1} +t} e^{(t_{i,1} +t-s)A}g\bigl(X(s)\bigr)\,ds
+
W_{i,1}(t),   
\label{eq: alg_total_variation}
\end{equation}
where
\[
W_{i,1}(t)
:=
\int_{t_{i,1}}^{t_{i,1}+t}
e^{(t_{i,1}+t-s)A}\sqrt{2D}\,dB_s .
\]
Then \(W_{i,1}(t)\) is a centered Gaussian random vector. More generally, for \(t,t'\in[0,h/M]\),
\[
\operatorname{Cov}\bigl(W_{i,1}(t),W_{i,1}(t')\bigr)
=
2\int_0^{t\wedge t'}
e^{(t-s)A}D e^{(t'-s)A^\top}\,ds .
\]

We approximate the nonlinear term \(g\) by a blockwise Lagrange interpolant. 
On each block \(B_i\), let \(\{\ell_{i,j}\}_{j=1}^{K-1}\) be the Lagrange basis
associated with the nodes \(\{t_{i,j}\}_{j=1}^{K-1}\). Thus, for \(s\in B_i\),
\[
\ell_{i,j}(s)
:=
\prod_{\substack{m=1\\m\neq j}}^{K-1}
\frac{s-t_{i,m}}{t_{i,j}-t_{i,m}} .
\]
Then, for \( s \in B_{i}\), we approximate
\[
g(X(s))
\approx P_{B_i}(s, X) = 
\sum_{j=1}^{K-1}
\ell_{i,j}(s)\,g\bigl(X(t_{i,j})\bigr),
\qquad s\in B_i.
\]

Plugging the blockwise interpolant \(P_{B_i}\) into
\eqref{eq: alg_total_variation}, we obtain the discretized path
\[
 \widehat X(t_{i,1}+t)
=
e^{tA} \widehat X(t_{i,1})
+
\int_{t_{i,1}}^{t_{i,1} +t} e^{(t_{i,1} +t-s)A}\, P_{B_i}(s, \widehat X) \,ds
+
W_{i,1}(t), \qquad t \in [0, h/M].
\]
This equation is still implicit. Indeed, the interpolant
\(P_{B_i}(\widehat X,s)\) is constructed from the unknown collocation values
\(\{\widehat X(t_{i,j})\}_{j=1}^{K-1}\) on the same block:
\[
P_{B_i}(s, \widehat X)
=
\sum_{j=1}^{K-1}
\ell_{i,j}(s)\,
g\bigl(\widehat X(t_{i,j})\bigr).
\]
Thus, evaluating the discretized path at the collocation nodes
\(t=t_{i,j}-t_{i,1}\) gives a nonlinear system for the unknown stage values
\(\{\widehat X(t_{i,j})\}_{j=1}^{K-1}\). This is the fixed-point system induced
by the blockwise Lagrange discretization.

We solve this system by Picard iteration. Starting from an initial guess for the
collocation values, we freeze the nonlinear term \(P_{B_i}\) at the previous
Picard iterate and solve the resulting linear variation-of-constants equation.
Repeating this procedure yields a sequence of approximations that converges to
the fixed point under the contraction condition established later.

More formally, we initialize the algorithm at step $n$ as follows:
\[
\widehat X^{[0]}(nh + t_{i,j}) := \widehat X(nh),
\qquad i=0,\dots,M-1,\;\; j=1,\dots,K-1,
\]
and iterate for $k=1,\dots,\nu$ for block \(B_i\):
\begin{align*}
\widehat X^{[k]}(nh + t_{i,j})
&=
e^{(t_{i,j}-t_{i,1})A} \widehat X^{[k]}(nh + t_{i, 1})
+
\int_{t_{i,1}}^{t_{i,j}} e^{(t_{i,j}-s)A}\, P_{B_i}(s, \widehat X^{[k-1]}) \,ds\,\\
& \qquad+
W_{i,1}(t_{i,j}-t_{i,1}).
\end{align*}

Expanding the blockwise interpolant,
\[
P_{B_i}\bigl(s, \widehat X^{[k-1]}\bigr)
=
\sum_{\ell=1}^{K-1}
\ell_{i,\ell}(s)\,
g\!\left(\widehat X^{[k-1]}(nh+t_{i,\ell})\right),
\qquad s\in B_i.
\]
Therefore,
\[
\int_{t_{i,1}}^{t_{i,j}}
e^{(t_{i,j}-s)A}
P_{B_i}\bigl(\widehat X^{[k-1]},s\bigr)\,ds
=
\sum_{\ell=1}^{K-1}
\alpha_{j\ell}^{(i)}
g\!\left(\widehat X^{[k-1]}(nh+t_{i,\ell})\right),
\]
where
\[
\alpha_{j\ell}^{(i)}
:=
\int_{t_{i,1}}^{t_{i,j}}
e^{(t_{i,j}-s)A}\ell_{i,\ell}(s)\,ds.
\]
Thus the Picard update can be written as
\begin{equation}
  \widehat X^{[k]}(nh+t_{i,j})
=
e^{(t_{i,j}-t_{i,1})A}
\widehat X^{[k]}(nh+t_{i,1})
+
\sum_{\ell=1}^{K-1}
\alpha_{j\ell}^{(i)}
g\!\left(\widehat X^{[k-1]}(nh+t_{i,\ell})\right)
+
W_{i,1}(t_{i,j}-t_{i,1}).  
\label{eq:algo_final_expression}
\end{equation}

After \(\nu\) Picard iterations, we set the numerical value at the end of the step to be the \(\nu\)-th Picard iterate evaluated at the final node:
\[\widehat X((n+1)h):=\widehat X^{[\nu]}(nh+h).
\]
Note that the weights \(\alpha_{j\ell}^{(i)}\) in \eqref{eq:algo_final_expression}
depend only on the matrix \(A\), the block
length \(\Delta:=h/M\), and the relative locations of the collocation nodes. Indeed, writing
\[
    c_j:=\frac{j-1}{K-2},\qquad t_{i,j}=i\Delta+c_j\Delta,
\]
and using the change of variables \(s=i\Delta+\Delta u\), the local Lagrange basis satisfies
\[
    \ell_{i,\ell}(s)=\widetilde\ell_\ell(u),
\]
where \(\widetilde\ell_\ell\) is the corresponding reference basis on \([0,1]\). Hence
\[
    \alpha_{j\ell}^{(i)}
    =
    \int_{t_{i,1}}^{t_{i,j}} e^{(t_{i,j}-s)A}\ell_{i,\ell}(s)\,ds
    =
    \Delta
    \int_0^{c_j}
    e^{\Delta(c_j-u)A}\widetilde\ell_\ell(u)\,du .
\]
Thus, for equal-length blocks with the same relative collocation nodes, the coefficients are
independent of the block index \(i\). We may write \(\alpha_{j\ell}^{(i)}=\alpha_{j\ell}\), and
these reference-block weights can be precomputed before running the Picard iterations.

Moreover, for each block \(B_i\), the Gaussian vector
\[
\bigl(W_{i,1}(t_{i,1}-t_{i,1}),\ldots,
W_{i,1}(t_{i,K-1}-t_{i,1})\bigr)
\]
can be generated and stored independently across \(i=0,\ldots,M-1\), according to the covariance formula above. Then, the updates for the above discretizations can be computed quickly in parallel. We also note that the drift matrix \(A\) inherits the Kronecker structure of \(D,Q\), and \(J\): it acts only through a \(K\times K\) block matrix tensored with \(I_d\). Consequently, the matrix exponential and covariance computations are carried out in this small block space and then applied coordinatewise across the \(d\) dimensions.

\subsection{Alternative representation of Algorithm~\ref{alg:picard_sampler}}
\label{Appendix-F2}

We verify that Algorithm~\ref{alg:picard_sampler} can be viewed as evaluating the Picard iterates of the discretized operator \(\widehat{\mathcal T}_y\) \eqref{eq: operator_hat_T2} at the blockwise collocation nodes.

\begin{lemma}[Equivalence of Algorithm~\ref{alg:picard_sampler} and Picard iterates]
\label{lemma:equivalent_algorithm}

Fix a time step \(n\), set \(y=\widehat X(nh)\), and define the Picard sequence
\[
    \widehat Y_y^{[0]}(t)\equiv y,
    \qquad
    \widehat Y_y^{[k+1]}
    :=
    \widehat{\mathcal T}_y[\widehat Y_y^{[k]}],
    \qquad k\ge0.
\]
Let \(\widehat X^{[k]}(nh+t_{i,j})\) denote the stage values produced by
Algorithm~\ref{alg:picard_sampler} at Picard index \(k\). Then, for every
\(k\ge0\), \(i=0,\ldots,M-1\), and \(j=1,\ldots,K-1\),
\[
    \widehat Y_y^{[k]}(t_{i,j})
    =
    \widehat X^{[k]}(nh+t_{i,j}).
\]
In particular,
\[
    \widehat Y_y^{[\nu]}(h)
    =
    \widehat X^{[\nu]}(nh+h)
    =
    \widehat X((n+1)h).
\]
\end{lemma}

\begin{proof}
Let \(X\in \mathcal C([0,h],\mathbb R^{Kd})\), and write
\[
    \widetilde X:=\widehat{\mathcal T}_y[X].
\]
By the variation-of-constants formula, for any \(0\le s\le t\le h\),
\[
\widetilde X(t)
=
e^{(t-s)A}\widetilde X(s)
+
\int_s^t e^{(t-u)A}P(u;X)\,du
+
\int_s^t e^{(t-u)A}\sqrt{2D}\,dB_{nh+u}.
\]
Here \(P(u;X)\) is the blockwise Lagrange approximation of the nonlinear drift on the
step \([nh,(n+1)h]\).

Now fix a block \(B_i\) and evaluate the preceding identity with
\(s=t_{i,1}\) and \(t=t_{i,j}\). Since \(P(u;X)=P_{B_i}(u;X)\) for \(u\in B_i\), we obtain
\[
\widetilde X(t_{i,j})
=
e^{(t_{i,j}-t_{i,1})A}\widetilde X(t_{i,1})
+
\int_{t_{i,1}}^{t_{i,j}}
e^{(t_{i,j}-u)A}P_{B_i}(u;X)\,du
+
W_{n,i}(t_{i,j}),
\]
where
\[
    W_{n,i}(t_{i,j})
    :=
    \int_{t_{i,1}}^{t_{i,j}}
    e^{(t_{i,j}-u)A}\sqrt{2D}\,dB_{nh+u}.
\]
On the block \(B_i\), the local Lagrange interpolant has the form
\[
    P_{B_i}(u;X)
    =
    \sum_{\ell=1}^{K-1}
    \ell_{i,\ell}(u)\,
    g\!\left(X(t_{i,\ell})\right).
\]
Therefore,
\[
\int_{t_{i,1}}^{t_{i,j}}
e^{(t_{i,j}-u)A}P_{B_i}(u;X)\,du
=
\sum_{\ell=1}^{K-1}
\alpha_{j\ell}
g\!\left(X(t_{i,\ell})\right),
\]
where
\[
    \alpha_{j\ell}
    :=
    \int_{t_{i,1}}^{t_{i,j}}
    e^{(t_{i,j}-u)A}\ell_{i,\ell}(u)\,du .
\]
Since all blocks have the same length and the same relative collocation nodes, these
coefficients are independent of \(i\) after the affine change of variables from \(B_i\) to
\([0,1]\). Thus,
\[
\widetilde X(t_{i,j})
=
e^{(t_{i,j}-t_{i,1})A}\widetilde X(t_{i,1})
+
\sum_{\ell=1}^{K-1}
\alpha_{j\ell}
g\!\left(X(t_{i,\ell})\right)
+
W_{n,i}(t_{i,j}).
\]

We now prove the equivalence by induction on the Picard index. For \(k=0\), both
constructions initialize the stage values at the same constant value \(y=\widehat X(nh)\),
so the claim is immediate.

Assume the claim holds at Picard index \(k\). Apply the preceding display with
\(X=\widehat Y_y^{[k]}\) and \(\widetilde X=\widehat Y_y^{[k+1]}\). Then
\[
\widehat Y_y^{[k+1]}(t_{i,j})
=
e^{(t_{i,j}-t_{i,1})A}
\widehat Y_y^{[k+1]}(t_{i,1})
+
\sum_{\ell=1}^{K-1}
\alpha_{j\ell}
g\!\left(\widehat Y_y^{[k]}(t_{i,\ell})\right)
+
W_{n,i}(t_{i,j}).
\]
By the induction hypothesis,
\[
    g\!\left(\widehat Y_y^{[k]}(t_{i,\ell})\right)
    =
    g\!\left(\widehat X^{[k]}(nh+t_{i,\ell})\right).
\]
Moreover, the value \(\widehat Y_y^{[k+1]}(t_{i,1})\) is the updated value at the beginning
of block \(B_i\), which is exactly the block-start value used by
Algorithm~\ref{alg:picard_sampler} at Picard index \(k+1\). Hence the last display matches
the blockwise update in Algorithm~\ref{alg:picard_sampler}, and so
\[
    \widehat Y_y^{[k+1]}(t_{i,j})
    =
    \widehat X^{[k+1]}(nh+t_{i,j}).
\]
This proves the induction step.

Taking \(i=M-1\) and \(j=K-1\), we have \(t_{M-1,K-1}=h\). Therefore,
\[
    \widehat Y_y^{[\nu]}(h)
    =
    \widehat X^{[\nu]}(nh+h)
    =
    \widehat X((n+1)h),
\]
which completes the proof.
\end{proof}

\section{Continuous-time contraction}
\label{Appendix-C}

In this section, we establish the contractivity of the continuous-time higher-order
Langevin dynamics \eqref{eq: high_order_langevin}. This contraction property is the key ingredient that gives exponential
decay of the error under the exact dynamics. Since the Wasserstein distance is measured in
the Euclidean norm, we prove contraction in a weighted norm
\[
    \|x\|_S^2 := x^\top Sx ,
\]
where \(S \succ 0\). Because \(S\) is positive definite, this norm is equivalent to the
Euclidean norm:
\[
    \lambda_{\min}(S)\|x\|^2
    \le
    \|x\|_S^2
    \le
    \lambda_{\max}(S)\|x\|^2 .
\]
Thus, exponential contraction in the \(S\)-norm implies exponential contraction in the
Euclidean norm up to constants depending on the condition number of \(S\). Let \(b(x)=-(D+Q)\nabla H(x)\) denote the drift, and let \(J_b(x)\) be its Jacobian. Throughout this section, we work with a symmetric positive definite matrix \(S\succ0\) satisfying the inequality
\begin{equation}
S J_b(x)+J_b(x)^\top S\preceq -2C_S S,
\qquad \forall x\in\mathbb{R}^{Kd},
\label{eq:S_matrix_property}
\end{equation}
for some constant \(C_S>0\) depending only on \(m,L,\gamma\), and \(K\). The existence of such a matrix \(S\) is verified in the next subsection. Assuming this inequality holds, we now state the resulting contraction estimate induced by the Lyapunov function \(x \mapsto x^\top Sx\).

\begin{proposition}\label{prop-appendix:convergence_S-norm}
Let the processes $\{X(t)\}$ and $\{X^{\ast}(t)\}$ follow the $K$th-order Langevin process in \eqref{eq: high_order_langevin} with sufficiently high $\gamma$. Assume that the potential $U(X)$ follows Assumption~\ref{assump:U-strong-smooth}. Let the initial conditions be $X(0)$ and $X^{\ast}(0)\in\mathbb{R}^{Kd}$.  
Then there exists a coupling
\[
\bar{\zeta}\in\Gamma\bigl(p_t(X(t)\mid X(0)),\,p_t^{\ast}(X^{\ast}(t)\mid X^{\ast}(0))\bigr)
\]
of the laws of $X(t)$ and $X^{\ast}(t)$, a symmetric positive definite matrix $S \succ 0$ satisyfing \eqref{eq:S_matrix_property}, such that
\begin{equation} 
\frac{\mathrm d}{\mathrm dt}\,
  (X(t)-X^{\ast}(t))^{\top}S(X(t)-X^{\ast}(t))
\;\le\;-2C_S\,(X(t)-X^{\ast}(t))^{\top}S(X(t)-X^{\ast}(t)),
\end{equation}
for all \((X(t),X^{\ast}(t))\sim\bar{\zeta}\).
\end{proposition}

\begin{proof}

Consider the processes $X(t)$ and $X^{\ast}(t)$ in the statement. Under synchronous coupling, we have:
\begin{align*}
    d(X(t) - X^{\ast}(t))
        &=  -(D+Q)\bigl(\nabla H(X(t)) - \nabla H(X^{\ast}(t))\bigr)\,dt \\
        &= \bigl(b(X(t)) - b(X^{\ast}(t))\bigr)\,dt \\
        &= \Bigg[\int_{0}^{1} \underbrace{J_b\bigl(X^{\ast}(t) + \lambda (X(t) - X^{\ast}(t))\bigr)}_{=:\widetilde J_b(t,\lambda)}\,d\lambda\Bigg]
           (X(t) - X^{\ast}(t))\,dt.
\end{align*}
Here $b(\cdot)$ is the drift and $J_b$ its Jacobian . Since $U$ is $m$-strongly convex and $L$-smooth on $\mathbb{R}^d$, the (vector-valued) mean-value theorem (fundamental theorem of calculus) on open convex sets yields the last equality. 

Consequently, using \eqref{eq:S_matrix_property} we obtain
\begin{align*}
    &\frac{\mathrm d}{\mathrm dt}\,
    (X(t)-X^{\ast}(t))^{\top}S(X(t)-X^{\ast}(t)) \\
    &=
    (X(t)-X^{\ast}(t))^{\top}
    \bigg(
        S\int_{0}^{1} \widetilde J_b(t,\lambda)\,d\lambda
        +
        \int_{0}^{1} \widetilde J_b(t,\lambda)^\top\,d\lambda\,S
    \bigg)
    (X(t)-X^{\ast}(t)) \\
    &=
    \int_{0}^{1} (X(t)-X^{\ast}(t))^{\top}
    \bigl(
        S\widetilde J_b(t,\lambda)
        +
        \widetilde J_b(t,\lambda)^{\top}S
    \bigr)
    (X(t)-X^{\ast}(t))\,d\lambda \\
    &\le
    \int_{0}^{1}
    -2C_S\,(X(t)-X^{\ast}(t))^{\top}S(X(t)-X^{\ast}(t))\,d\lambda \\
    &=
    -2C_S\,(X(t)-X^{\ast}(t))^{\top}S(X(t)-X^{\ast}(t)).
\end{align*}

This finishes the proof of Proposition~\ref{prop-appendix:convergence_S-norm}.
\end{proof}

\begin{remark}
This proposition immediately yields exponential decay of the continuous-time
dynamics. Indeed, by Grönwall's inequality,
\[
    \|X(t)-X^{\ast}(t)\|_S^2
    \le
    e^{-2C_S t}\|X(0)-X^{\ast}(0)\|_S^2 .
\]
Since \(S\succ 0\), the \(S\)-norm is equivalent to the Euclidean norm, and hence this contraction
can be translated into a Wasserstein contraction up to constants depending on \(S\).
\end{remark}

\subsection{Existence of Lyapunov matrix \(S\)}
\label{Appendix-C1}

To prove the existence, we follow the approach developed in
\cite{monmarche2023almost, arnold2014sharp, arnold2020sharp}. We first introduce a convenient reparametrization of the dynamics \eqref{eq: high_order_langevin}, which allows us to apply their results directly.
Write the state as
\[
X(t)=\bigl(X_1(t),\,Y(t)\bigr)\in \mathbb{R}^{d}\times \mathbb{R}^{(K-1)d}.
\]
Then the continuous-time process \eqref{eq: high_order_langevin} can be rewritten as
\begin{equation}\label{eq:alt_sde_system}
  \mathrm{d}X(t) = P\,Y(t)\,\mathrm{d}t,
  \qquad
  \mathrm{d}Y(t)
    = -P^{\top}\nabla U(X(t))\,\mathrm{d}t
      - \gamma\,\tilde{Q}\,Y(t)\,\mathrm{d}t
      + \sqrt{2\gamma\,\tilde{D}}\,\mathrm{d}B_t,
\end{equation}
where \(B_t\) is a standard \((K-1)d\)-dimensional Brownian motion and the
matrices \(P\in \mathbb{R}^{d\times (K-1)d}\),
\(\tilde Q,\tilde D\in \mathbb{R}^{(K-1)d\times (K-1)d}\) are defined by \(P \;=\;
\begin{pmatrix}
I_d & 0 & \cdots & 0
\end{pmatrix},\)
\[
\tilde{Q} \;=\;
\begin{pmatrix}
0   & -I_d & 0    & \cdots & 0 \\
I_d &  0   & -I_d & \ddots & \vdots \\
0   & \ddots & \ddots & \ddots & 0 \\
\vdots & \ddots & I_d & 0 & -I_d \\
0   & \cdots & 0 & I_d & I_d
\end{pmatrix},
\quad
\tilde D
\;=\;
\begin{pmatrix}
0_d &        &        &        & 0 \\
    & \ddots &        &        & \vdots \\
    &        & 0_d    &        & 0 \\
    &        &        & 0_d    & 0 \\
0   & \cdots & 0      & 0      & I_d
\end{pmatrix}.
\]

Let \(b(x_1,y)\) denote the drift of \eqref{eq:alt_sde_system}, then for \(x_1\in\mathbb{R}^{d}\), \(y\in\mathbb{R}^{(K-1)d}\),
(and correspondingly for \(x=(x_1,y)\in\mathbb{R}^{Kd}\)),
\begin{equation}\label{eq:drift_alt_sde}
  b(x_1,y)=b(x)=-(D+Q)\,\nabla H(x).
\end{equation}
We denote by \(J_b\) the Jacobian of this drift \(b\).

To this end, we state the following set of conditions, adapted from prior work (see~\citet{monmarche2023almost}), and later show that our dynamics satisfy them.
\begin{condition}\label{cond:C1}
There exist constants $m,L>0$ such that
\[
\forall\,x\in\mathbb{R}^d,\qquad
m\,I_d \;\preceq\; \nabla^2 U(x) \;\preceq\; L\,I_d .
\]    
\end{condition}

\begin{condition}\label{cond:C2}
There exist $\kappa>0$ and a symmetric positive–definite matrix
$\tilde{N}\in\mathbb{R}^{(K-1)d\times (K-1)d}$ such that
\[
\tilde{N}\tilde{Q}+\tilde{Q}^{\top}\tilde{N} \;\succeq\; 2\kappa\,\tilde{N} .
\]
\end{condition}

\begin{condition}\label{cond:C3}
With respect to the decomposition
\[
\widetilde Q=
\begin{pmatrix}
\widetilde Q_{11} & \widetilde Q_{12}\\
\widetilde Q_{21} & \widetilde Q_{22}
\end{pmatrix},
\qquad
\widetilde Q_{11}\in\mathbb R^{d\times d},
\]
the block \(\widetilde Q_{22}\) is invertible and that the Schur complement
\[
E:=\widetilde Q_{11}
-\widetilde Q_{12}\widetilde Q_{22}^{-1}\widetilde Q_{21}
\]
is symmetric positive definite. Moreover, set \(H:=\tilde{Q}_{12}\tilde{Q}_{22}^{-1}\).
\end{condition}

The following lemma ensures that the higher-order Langevin dynamics \eqref{eq:alt_sde_system} satisfies all the above conditions.
\begin{lemma}
The dynamics \eqref{eq:alt_sde_system} satisfy
Conditions~\ref{cond:C1}–\ref{cond:C3}.
\end{lemma}
\begin{proof}
Condition~\ref{cond:C1} follows directly from the smoothness and strong convexity
assumptions on \(U\) (Assumption~\ref{assump:U-strong-smooth}).

For Condition~\ref{cond:C2}, write \(\widetilde Q=\widetilde Q_{\mathrm{can}}\otimes I_d\) with the
\((K{-}1)\times (K{-}1)\) ``canonical'' backbone matrix
\[
\widetilde Q_{\mathrm{can}}
=\begin{pmatrix}
0 & -1 & 0 & \cdots & 0 \\
1 &  0 & -1& \ddots & \vdots \\
0 & \ddots & \ddots & \ddots & 0 \\
\vdots & \ddots & 1 & 0 & -1 \\
0 & \cdots & 0 & 1 & 1
\end{pmatrix}.
\]
Hence $\lambda(\widetilde Q)=\lambda(\widetilde Q_{\mathrm{can}})$ 
(with $\lambda(\cdot)$ denoting the spectrum, with multiplicities scaled by \(d\)), and eigenvectors of $\widetilde Q$ are $w=v\otimes e_j$, where
$v$ is an eigenvector of $\widetilde Q_{\mathrm{can}}$ and
$\{e_j\}_{j=1}^d$ is the canonical basis of $\mathbb{R}^d$.

Given that $\widetilde Q$ and $\widetilde Q_{\mathrm{can}}$ share the same spectrum, we now show that
\[
\min\{\mathrm{Re}(\lambda):\lambda\in\lambda(\widetilde Q_{\mathrm{can}})\}>0.
\]
Decompose
\[
\widetilde Q_{\mathrm{can}}
=\frac{\widetilde Q_{\mathrm{can}}+\widetilde Q_{\mathrm{can}}^{\!\top}}{2}
+\frac{\widetilde Q_{\mathrm{can}}-\widetilde Q_{\mathrm{can}}^{\!\top}}{2}
=: L+M,
\]
where \(L\) is Hermitian and \(M\) is skew–Hermitian. If
\(\widetilde Q_{\mathrm{can}}x=\lambda x\) for some
\(x=(x_1,\ldots,x_{K-1})\in\mathbb{C}^{K-1}\setminus\{0\}\), then by the standard
Rayleigh-quotient identity for the symmetric part,
\[
\text{Re}(\lambda)=\frac{x^{\!*}Lx}{x^{\!*}x}=\frac{x^2_{K-1}}{\|x\|^2}.
\]
We claim \(x_{K-1}\neq 0\). Indeed, if \(x_{K-1}=0\), then the last row of
\(\widetilde Q_{\mathrm{can}}x=\lambda x\) gives \(x_{K-2}=0\); iterating
backwards yields \(x_{K-3}=0,\ldots,x_1=0\), a contradiction to \(x\neq 0\).
Hence \(x_{K-1}\neq 0\) and thus \(\text{Re}(\lambda)>0\) for every eigenvalue
\(\lambda\) of \(\widetilde Q_{\mathrm{can}}\). The same conclusion then holds
for \(\widetilde Q=\widetilde Q_{\mathrm{can}}\otimes I_d\).

Given the above result that \(\min\{\text{Re}(\lambda):\lambda\in\lambda(\widetilde Q)\}>0\),
the works \citet[ Lemma~4.3]{arnold2014sharp}, 
\citet[ Section~2.1]{arnold2020sharp}, 
and \citet[ Appendix~A]{dang2025high} provide an explicit
construction of a symmetric positive–definite matrix \(\tilde{N}\) such that
\[
\tilde{N}\widetilde Q+\widetilde Q^{\top}\tilde{N} \;\succeq\; 2\kappa\,\tilde{N} .
\]
This satisfies Condition~\ref{cond:C2}.

With $H:=-(I_d,\ldots,I_d)\in\mathbb{R}^{d\times (K-2)d}$ we verify directly that
\(H\,\widetilde Q_{22}=\widetilde Q_{12}\). It follows that
\(E:=\widetilde Q_{11}-\widetilde Q_{12}\widetilde Q_{22}^{-1}\widetilde Q_{21}
=\widetilde Q_{11}-H\,\widetilde Q_{21}=I_d\),
which is symmetric positive definite. this proves Condition~\ref{cond:C3}.
\end{proof}

Since $E= I_d\succ0$ and $\tilde{N}\succ0$, we can fix constants $h_i>0$ ($i=1,\dots,5$) such that
\begin{equation}\label{eq:hi-bounds}
  \tilde{N} P^{\top} P \tilde{N} \;\preceq\; h_1\,\tilde{N},
  \qquad
  \frac{1}{h_2}\,I_d \;\preceq\; E \;\preceq\; h_3\,I_d,
\end{equation}
and
\begin{equation}\label{eq:block-bounds}
  \begin{pmatrix} I_d & -H \\[2pt] 0 & 0 \end{pmatrix} \;\preceq\; h_4\,\tilde{N},
  \qquad
  \begin{pmatrix} I_d & -H \\[2pt] -H^{\top} & 0 \end{pmatrix} \;\preceq\; h_5\,\tilde{N}.
\end{equation}
Note that, here we can take $h_2 =h_3 =1$.

The following theorem, taken from \cite[Theorem~9]{monmarche2023almost}, establishes the existence of a matrix \(S\) satisfying \eqref{eq:S_matrix_property}; see that reference for the proof.

\begin{theorem}\label{thm:contr_S_C_s}
Define
\begin{align}
  \gamma_{0}
  := 2\sqrt{\frac{h_{1}L}{\kappa}}\,
     \max\!\left(\sqrt{h_{2}h_{5}},\,\sqrt{\frac{h_{4}}{\kappa}}\right). \label{eq:gamma-0}
\end{align}
If $\gamma\ge \gamma_{0}$, then  under Conditions~\ref{cond:C1}–\ref{cond:C3} there exist a positive definite matrix $S$ such that the drift of the process
\eqref{eq:alt_sde_system} satisfies 
\[
\forall\, x\in\mathbb{R}^{Kd},\qquad S\,J_b(x) + J_b^{\top}(x)S\,\ \preceq\ -2\,C_s\,S,
\]
 with
\[
  C_s \;=\; \min\!\left(\frac{m}{3h_{3}\gamma},\; \frac{\gamma\kappa}{6}\right).
\]
Moreover, $S$ is a positive definite matrix such that
\[
  \frac{1}{2}
  \begin{pmatrix}
    I_d & 0 \\
    0 & \dfrac{\kappa}{L h_{1}}\,\tilde{N}
  \end{pmatrix}
  \;\preceq\;
  S
  \;\preceq\;
  \frac{3}{2}
  \begin{pmatrix}
    I_d & 0 \\
    0 & \dfrac{\kappa}{L h_{1}}\,\tilde{N}
  \end{pmatrix}.
\]
\end{theorem}

\begin{remark}
Based on the construction of $\tilde{N}$ provided in \citet[ Appendix~A]{dang2025high}, 
both $\|\tilde{N}\|_{\mathrm{op}}$ and $\|\tilde{N}^{-1}\|_{\mathrm{op}}$ 
are bounded by constants that are independent of the ambient dimension~$d$. 
As a result, all constants appearing in inequality~\eqref{eq:hi-bounds}, 
as well as the eigenvalues of the associated matrix~$S$, 
are dimension-free (see \citet[ Appendix~A]{dang2025high}). 
In particular, the parameters $h_i$ and~$\kappa$, 
and consequently the constants 
$C_S$, $\lambda_{\min}(S)$, and $\lambda_{\max}(S)$, 
remain independent of~$d$.
\end{remark}

\section{Proof of main results}
\label{Appendix-D}

With the Lyapunov matrix \(S\) constructed in Appendix~\ref{Appendix-C}, we now turn to the proof of the main results. We begin with the following synchronous coupling bound:
\begin{equation}\label{eq:w2-S}
W_2^2\bigl(\pi,\hat\pi^{(N)}\bigr)
\;\le\;
\|S^{-1}\|_{\mathrm{op}}\,
\mathbb{E}\bigl[\|X(Nh)-\widehat X(Nh)\|_{S}^{2}\bigr],
\end{equation}
where $\pi$ is the stationary law of the continuous dynamics as defined in \eqref{eq: high_order_langevin}, $\hat\pi^{(N)}$ is the law of the output of Algorithm~\ref{alg:picard_sampler} after $N$ steps of size $h$, and both processes are driven by the same Brownian path on each step. Throughout we assume $X(0)\sim\pi$ and $\widehat X(0)=0$, so the initial continuous process is stationary.

\subsection{Proof of Theorem~\ref{thm:main-W2-simplified}}
\label{Appendix-D1}
We provide the formal version of Theorem~1 below. 
\begin{theorem}\label{thm:main-W2-unified}
Assume that \(U\) satisfies
Assumptions~\ref{assump:U-strong-smooth}--\ref{assump:higher-smooth}.
Fix \(K\ge 2\), and run Algorithm~\ref{alg:picard_sampler} with
\(\gamma>\gamma_0\) \eqref{eq:gamma-0}, \(M\) blocks, \(K-1\) equispaced collocation nodes per
block, and \(\nu\) Picard iterations per step. Let the fixed step size be
\(h=h_0\in[h_1',h_1]\), where \(0<h_1'<h_1\) are the thresholds defined in
\eqref{eq: h_1' and h1}; these thresholds depend only on
\(\gamma,m,L,L_{\max},K\). Then there exist constants
\(c_1,c_2,c_3>0\), independent of \(d,N,M,\nu,\varepsilon\), such that for any
\(\varepsilon\in(0,1)\), if
\[
    N
    =
    \left\lceil
    \frac{c_1}{h_0}\log\!\left(\frac{d}{\varepsilon^2}\right)
    \right\rceil,
\]
and
\[
    M \ge c_2\sqrt d\,\varepsilon^{-1/(K-1)}
    \log^{\frac{1}{2K-2}}\!\left(\frac{d}{\varepsilon^2}\right),
    \qquad
    \nu \ge c_3\log M ,
\]
then the output law \(\widehat\pi^{(N)}\) of
Algorithm~\ref{alg:picard_sampler} satisfies
\[
    W_2^2\!\left(\widehat\pi^{(N)},\pi\right)
    \le
    \|S^{-1}\|_{\mathrm{op}}\varepsilon^2 .
\]
\end{theorem}

\begin{proof}
From Proposition~\ref{prop-appendix:combined-global}, for all sufficiently small constant
step sizes \(h>0\), we have
\begin{equation} \label{eq:global-clean-final-recalled}
E_{N}
\;\le\; C_0^*\,e^{-\frac{1}{2}C_S Nh}\,d
\;+\; C^\star_1\,\frac{N\,h^{2K-3}d^{K-1}}{M^{2K-2}}
\;+\; C^\star_2\,\rho^{2\nu - 1}d .
\end{equation}
Fix such a constant step size \(h=h_0>0\), independent of \(d,N,M,\nu\), and
\(\varepsilon\). We choose
\[
    N
    =
    \left\lceil
    \frac{2}{C_S h_0}
    \log\!\left(\frac{3C_0^\ast d}{\varepsilon^2}\right)
    \right\rceil .
\]
Then
\[
C_0^\ast\,e^{-\frac{1}{2}C_S Nh_0}\,d
\;\le\;
\frac{\varepsilon^2}{3}.
\]

Next, since \(\rho<1/10\), choosing
\[
\nu
\;\ge\;
\frac{1}{2}\left(
1+\frac{(2K-2)\log M}{\log(10)}
\right)
\]
ensures that
\[
\rho^{2\nu-1}
\le
\frac{1}{M^{2K-2}}.
\]
Consequently,
\[
C^\star_2\rho^{2\nu-1}d
\le
C^\star_2\frac{d}{M^{2K-2}}.
\]

It remains to choose \(M\). Since \(h_0=\Theta(1)\) and
\[
    N
    =
    O\!\left(
    \log\!\left(\frac{d}{\varepsilon^2}\right)
    \right),
\]
there exists a constant \(C_M>0\), depending only on the fixed problem parameters, such
that if
\[
M
\;\ge\;
C_M\,\sqrt{d}\,\varepsilon^{-\frac{1}{K-1}}\,
\log^{\frac{1}{2K-2}}\!\left(\frac{d}{\varepsilon^2}\right),
\]
then
\[
C^\star_1\,\frac{N\,h_0^{2K-3}d^{K-1}}{M^{2K-2}}
\;\le\;
\frac{\varepsilon^2}{3}
\qquad\text{and}\qquad
C^\star_2\,\frac{d}{M^{2K-2}}
\;\le\;
\frac{\varepsilon^2}{3}.
\]
Combining the three bounds in \eqref{eq:global-clean-final-recalled} gives
\[
W_2^2\bigl(\pi,\hat\pi^{(N)}\bigr)
\;\le\;
\|S^{-1}\|_{\mathrm{op}}\,\varepsilon^2,
\]
which proves Theorem~\ref{thm:main-W2-unified}.
\end{proof}

\subsection{Proof of Theorem~\ref{thm:main-W2-simplified-ridge}}
We provide the formal version of Theorem~2 below. 
\begin{theorem}\label{thm:main-W2-unified-ridge}
Assume that \(U\) satisfies Assumptions~\ref{assump:U-strong-smooth}--\ref{assump:centered} and Assumption~\ref{assump:ridge-separable}. Fix \(K\ge 3\), and run Algorithm~\ref{alg:picard_sampler} for \(\gamma > \gamma_0\) \eqref{eq:gamma-0} with \(M\) blocks, \(K-1\) equispaced collocation nodes per block, and \(\nu\) Picard iterations per step. Let the fixed step size be \(h=h_0\in[h_1',h_1]\), where \(0<h_1'<h_1\) are the thresholds defined in \eqref{eq: h_1' and h1}; these thresholds depend only on \(\gamma,m,L,L_{\max},\Lambda_\Theta,K\). Then there exist constants \(c_1,c_2,c_3>0\), independent of \(d,r,N,M,\nu,\varepsilon\), such that for any \(\varepsilon\in(0,1)\), if
\[
    N
    =
    \left\lceil
    \frac{c_1}{h_0}\log\!\left(\frac{d}{\varepsilon^2}\right)
    \right\rceil,
\]
and
\[
    M \ge c_2 \max\left\{\sqrt r,\; d^{\frac{1}{2K-2}}\right\}\,\varepsilon^{-1/(K-1)}
    \log\!^{\frac{1}{2K-2}}\left(\frac{d}{\varepsilon^2}\right),
    \qquad
    \nu \ge c_3\log M ,
\]
then the output law \(\widehat\pi^{(N)}\) of Algorithm~\ref{alg:picard_sampler} satisfies
\[
    W_2^2\!\left(\widehat\pi^{(N)},\pi\right)
    \le
    \|S^{-1}\|_{\mathrm{op}}\varepsilon^2 .
\]
\end{theorem}

\begin{proof}
The proof is identical to the proof of Theorem~\ref{thm:main-W2-simplified},
except for the interpolation term. Indeed, the continuous-time contraction
term \((\mathrm{I})\) and the Picard iteration term \((\mathrm{III})\) are
unchanged. Only the interpolation residual \((\mathrm{II})\) is improved by
the ridge-separable structure.

By Lemma~\ref{lem_appendix:lagrange-IR-uniform-ridge}, the interpolation error satisfies
\[
\mathbb E I_U
\;\lesssim\;
\frac{N h^{2K-2}}{M^{2K-2}}
\left(
r^{K-1}+d
\right).
\]
 Therefore, after applying the same one-step decomposition and summing the recursion as in Proposition~\ref{prop-appendix:combined-global}, we obtain
\[
E_{N}
\;\le\;
C_0^\ast e^{-\frac12 C_S N h}d
+
C_1^\ast\frac{N h^{2K-3}}{M^{2K-2}}
\left(r^{K-1}+d\right)
+
C_2^\ast\rho^{2\nu-1}d .
\]
 Fix a constant step size \(h=h_0\in[h_1',h_1]\), where \(h_1', h_1\) are the thresholds defined in \eqref{eq: h_1' and h1} independent of \(r, d,N,M,\nu\), and
\(\varepsilon\). We choose
\[
N
=
\left\lceil
\frac{2}{C_S h_0}
\log\left(\frac{3C_0^\ast d}{\varepsilon^2}\right)
\right\rceil .
\]
Then
\[
C_0^\ast e^{-\frac12 C_S N h_0} d
\le
\frac{\varepsilon^2}{3}.
\]

Next, since \(\rho<1/10\), choosing
\[
\nu
\ge
\frac12\left(
1+\frac{(2K-2)\log M}
{\log(10)}
\right)
\]
ensures that
\[
\rho^{2\nu-1}
\le
\frac{1}{M^{2K-2}}.
\]
Consequently,
\[
C_2^\ast \rho^{2\nu-1} d
\le
\frac{C_2^\ast d}{M^{2K-2}}.
\]

It remains to choose \(M\). Since \(h_0=\Theta(1)\) and
\[
N
=
O\!\left(\log\left(\frac{d}{\varepsilon^2}\right)\right),
\]
there exists a constant \(C_M>0\), depending only on the fixed problem
parameters, such that if
\[
M
\ge
C_M
\left[
\frac{
\bigl(r^{K-1}+d\bigr)
\log\left(\frac{d}{\varepsilon^2}\right)
}{\varepsilon^2}
\right]^{\frac{1}{2K-2}},
\]
then
\[
C_1^\ast
\frac{
N h_0^{2K-3}\bigl(r^{K-1}+d\bigr)
}{
M^{2K-2}
}
\le
\frac{\varepsilon^2}{3}, \quad \text{and} \quad C_2^\ast\rho^{2\nu-1}d  \leq \frac{\varepsilon^2}{3}.
\]

Combining the three bounds gives
\[
W_2^2(\pi,\widehat{\pi}^{(N)})
\le
\|S^{-1}\|_{\mathrm{op}}\varepsilon^2.
\]
\end{proof}

\subsection{Global error bound}
\label{Appendix-D2}

In this section, we combine the three one-step estimates: continuous-time contraction,
blockwise Lagrange interpolation error, and finite Picard-iteration error. This yields the
global recursion used in the proof of the main theorem. We state the formal version below.

\begin{proposition}
\label{prop-appendix:combined-global}
Let
\[
    E_n:=\mathbb E\|X(nh)-\widehat X(nh)\|_S^2 .
\]
There exist constants \(0<h_1'\le h_1\), depending only on \(\gamma,m,L,L_{\max}, \Lambda_\Theta, K\), such that for every constant step size \(h\in[h_1',h_1]\), the interpolation-error (Lemma~\ref{lem_appendix:lagrange-IR-uniform}, or
Lemma~\ref{lem_appendix:lagrange-IR-uniform-ridge} in the ridge-separable
setting), Picard-error bounds (Proposition~\ref{prop-appendix:picard-error}) apply. Moreover, there exist constants \(C_0^\ast,C_1^\ast,C_2^\ast>0\), independent of \(h, d, r, N,M,\nu,\varepsilon\), such that for all \(N\ge 0\),
\[
E_{N+1}
\le
C_0^\ast e^{-\frac12 C_S(N+1)h}d
+
C_1^\ast\frac{N h^{2K-3}d^{K-1}}{M^{2K-2}}
+
C_2^\ast\rho^{2\nu-1}d .
\]
\end{proposition}

\begin{proof}

We start from the one-step error decomposition \eqref{eq:one-step-ED}:
\[
E_{n+1}
\le
(1+C_Sh)(\mathrm I)
+
\left(2+\frac{2}{C_Sh}\right)(\mathrm{II})
+
\left(2+\frac{2}{C_Sh}\right)(\mathrm{III}),
\]
where \((\mathrm I)\), \((\mathrm{II})\), and \((\mathrm{III})\) denote the continuous-time
contraction, interpolation-error, and Picard-error terms, respectively.

The continuous-time term is controlled by Proposition~\ref{prop-appendix:convergence_S-norm}:
\[
    (\mathrm I)\le e^{-2C_Sh}E_n .
\]
The interpolation term is controlled by Proposition~\ref{prop:error-bound} together with
Lemma~\ref{lem_appendix:lagrange-IR-uniform}:
\[
    (\mathrm{II})
    \lesssim
    \frac{N h^{2K-1}d^{K-1}}{M^{2K-2}} .
\]
Finally, the Picard term is controlled by Proposition~\ref{prop-appendix:picard-error}:
\[
    (\mathrm{III})
    \lesssim
    \rho^{2\nu}h^2E_n+\rho^{2\nu}dh .
\]

Since \(\rho=2L\Gamma_\phi h\), the factor \((2+2/(C_Sh))\) can be absorbed by replacing
\(\rho^{2\nu}\) with \(\rho^{2\nu-1}\). Therefore, for constants \(C_1,C_2,C_3>0\),
\[
E_{n+1}
\le
\left((1+C_Sh)e^{-2C_Sh}+C_3\rho^{2\nu-1}h^2\right)E_n
+
C_1\frac{N h^{2K-2}d^{K-1}}{M^{2K-2}}
+
C_2\rho^{2\nu-1}dh .
\]

We next show that the coefficient of \(E_n\) is contractive for sufficiently small \(h\).
Since \(\rho<1/10\), there exists a constant \(\widetilde C_a>0\), independent of
\(h,d,n,N\), such that
\[
(1+C_Sh)e^{-2C_Sh}+C_3\rho^{2\nu-1}h^2
\le
e^{-C_Sh}+\widetilde C_a h^2
=:a_h .
\]
Using the elementary bound
\[
    e^{-C_Sh}\le 1-C_Sh+C_S^2h^2,
\]
we obtain
\[
    a_h
    \le
    1-C_Sh+(C_S^2+\widetilde C_a)h^2 .
\]
Let \(\bar h>0\) denote a preliminary upper step-size threshold that collects all
upper step-size restrictions required by the interpolation-error bounds
(Lemma~\ref{lem_appendix:lagrange-IR-uniform}, or
Lemma~\ref{lem_appendix:lagrange-IR-uniform-ridge} in the ridge-separable
setting) and by the Picard-error bound
(Proposition~\ref{prop-appendix:picard-error}). Define
\begin{equation}
  h_1 :=
\min\left\{
\bar h,\,
\frac{1}{4C_S},\,
\frac{C_S}{2(C_S^2+\widetilde C_a)}
\right\},
\qquad
h_1':=c h_1,  
\label{eq: h_1' and h1}
\end{equation}
where \(c\in(0,1)\) is the fixed constant.
Thus, for every constant step size \(h\in[h_1',h_1]\), the moment estimate in
Lemma~\ref{lem:expectaion_power2k} (or Lemma~\ref{lem:expectaion_power2k_ridge} in ridge-separable setting) applies, and all the continuous-time
contraction, interpolation-error, and Picard-error estimates used above still hold. These bounds imply 
\[
    (C_S^2+\widetilde C_a)h^2
    \le
    \frac12 C_Sh,
\]
and hence
\[
    a_h\le 1-\frac12 C_Sh .
\]

It follows that,
\[
E_{n+1}
\le
\left(1-\frac12 C_Sh\right)E_n
+
C_1\frac{N h^{2K-2}d^{K-1}}{M^{2K-2}}
+
C_2\rho^{2\nu-1}dh .
\]
Iterating this recursion gives
\[
E_{N+1}
\le
e^{-\frac12 C_S(N+1)h}E_0
+
\frac{2}{C_Sh}
\left(
C_1\frac{N h^{2K-2}d^{K-1}}{M^{2K-2}}
+
C_2\rho^{2\nu-1}dh
\right).
\]
Finally, since \(\widehat X(0)=0\) and \(X(0)\sim p^\ast\), Lemma~\ref{lem:StationaryMoments}
implies \(E_0=O(d)\). Absorbing constants into
\(C_0^\ast,C_1^\ast,C_2^\ast\), we obtain
\[
E_{N+1}
\le
C_0^\ast e^{-\frac12 C_S(N+1)h}d
+
C_1^\ast\frac{N h^{2K-3}d^{K-1}}{M^{2K-2}}
+
C_2^\ast\rho^{2\nu-1}d .
\]
This completes the proof.
\end{proof}

\section{Interpolation error}
\label{Appendix-E}

In this section, we control the interpolation-error term that appears in the one-step error decomposition. Throughout this section, we work under the higher-order smoothness condition in Assumption~\ref{assump:higher-smooth}.

\begin{proposition}[Interpolation error within one step]\label{prop:error-bound}
Fix a step size $h>0$ and index $n$, and set $y:=\widehat X(nh)$.
Let $X^\ast$ and $\widehat X^\ast$ be the fixed points of $\mathcal T_y$ and $\widehat {\mathcal T}_{y}$ on $[0,h]$ (as defined in \eqref{eq:operator_T} and \eqref{eq: operator_hat_T2}), driven by the same Brownian motion. Then
\begin{equation}\label{eq:interp-core}
\big\|X^\ast(h)-\widehat {X}^\ast(h)\big\|_{S}^{2}
\;\le\;
\frac{9\gamma^2\,\|S\|_{\mathrm{op}}}{C_S}\; h\; I_U,
\end{equation}
where
\[
I_U\;:=\;\sup_{s\in[nh,(n+1)h]}
\big\|\nabla U\big(\widehat {X}^{\ast}_{1}(s)\big)-P_2\big(s,\widehat{X}^{\ast}(s)\big)\big\|^{2},
\]
and \(P_2(s;X)\in\mathbb R^d\) denotes the second \(d\)-dimensional block of
the \(Kd\)-dimensional interpolant \(P(s;X)\).

\end{proposition}

\begin{proof}
Under the synchronous coupling, subtracting the two fixed-point equations gives
\[
\begin{aligned}
\frac{d}{dt}\bigl(X^\ast(t)-\widehat X^\ast(t)\bigr)
&=
-\,(D+Q)
\begin{bmatrix}
\nabla U\bigl(X^\ast_1(t)\bigr)-\nabla U\bigl(\widehat X^\ast_1(t)\bigr)\\
X^\ast_2(t)-\widehat X^\ast_2(t)\\[-3pt]
\vdots\\[-3pt]
X^\ast_K(t)-\widehat X^\ast_K(t)
\end{bmatrix}
-\,(D+Q)
\begin{bmatrix}
r(t)\\[1pt]
0_d\\
\vdots\\
0_d
\end{bmatrix},
\end{aligned}
\]
where
\[
r(t):=\nabla U\bigl(\widehat X^\ast_1(t)\bigr)-P_2\bigl(t;\widehat X^\ast\bigr).
\]
By the mean-value theorem,
\[
\nabla U\bigl(X^\ast_1(t)\bigr)-\nabla U\bigl(\widehat X^\ast_1(t)\bigr)
=
H_t\bigl(X^\ast_1(t)-\widehat X^\ast_1(t)\bigr),
\qquad
mI_d\preceq H_t\preceq LI_d .
\]
Define
\[
J_b(t):=-(D+Q)\operatorname{diag}(H_t,I_d,\ldots,I_d),
\qquad
R(t):=(r(t),0_d,\ldots,0_d)^\top .
\]
Then \eqref{eq:S_matrix_property} implies that \(J_b(t)\) satisfies the Lyapunov inequality
\[
S J_b(t)+J_b(t)^\top S\preceq -2C_S S.
\]

Differentiating \(\|X^\ast(t)-\widehat X^\ast(t)\|_S^2\) and using the above inequality, we obtain
\begin{align*}
\frac{d}{dt}\|X^\ast(t)-\widehat X^\ast(t)\|_S^2
&\le
-2C_S\|X^\ast(t)-\widehat X^\ast(t)\|_S^2  \\
&\quad
+2\|X^\ast(t)-\widehat X^\ast(t)\|_S
\|S^{1/2}(D+Q)R(t)\|.
\end{align*}
Applying Young's inequality to the second term gives
\[
\frac{d}{dt}\|X^\ast(t)-\widehat X^\ast(t)\|_S^2
\le
-C_S\|X^\ast(t)-\widehat X^\ast(t)\|_S^2
+
\frac{\|S\|_{\mathrm{op}}\|D+Q\|^2}{C_S}\|r(t)\|^2 .
\]
Since \(X^\ast(0)=\widehat X^\ast(0)=y\), Grönwall's inequality implies
\[
\|X^\ast(h)-\widehat X^\ast(h)\|_S^2
\le
\frac{\|S\|_{\mathrm{op}}\|D+Q\|^2}{C_S}
\int_0^h \|r(t)\|^2\,dt .
\]
Using \(\|D+Q\|\le 3\gamma\) (Lemma~\ref{lemma:d_plus_q}), we get
\[
\|X^\ast(h)-\widehat X^\ast(h)\|_S^2
\le
\frac{9\gamma^2\|S\|_{\mathrm{op}}}{C_S}
h
\sup_{t\in[0,h]}
\|\nabla U(\widehat X_1^\ast(t))-P_2(t;\widehat X^\ast)\|^2 .
\]
This proves \eqref{eq:interp-core}.
\end{proof}

\subsection{Bounding \(I_U\)}\label{subsection:part1}

We now bound the interpolation residual
\[
    I_U
    :=
    \sup_{t\in[0,h]}
    \left\|
        \nabla U(\widehat X_1^\ast(t))
        -
        P_2(t;\widehat X^\ast)
    \right\|^2 .
\]
The argument has two steps. First, we apply the deterministic Lagrange interpolation
remainder on each block \(B_i\), which reduces the problem to controlling the
\((K-1)\)-th time derivative of
\(t\mapsto \nabla U(\widehat X_1^\ast(t))\). Second, we use the higher-order smoothness
assumption on \(U\) together with moment bounds for the fixed-point path
\(\widehat X^\ast\) to control this derivative uniformly over the step.

Assumption~\ref{assump:higher-smooth} implies the multilinear bound
\[
\|D^{r}U(x)[v_1,\ldots,v_{r-1}]\|
\le
L_{\max}\prod_{j=1}^{r-1}\|v_j\|,
\qquad r=2,\ldots,K,
\]
where \(L_{\max}:=\max{\{1,\max_{1\le r\le K-1}L_r\}}\). 

\begin{lemma}\label{lem: inductive bound of time derivative of nabla U}
Let
\[
g(t):=\nabla U(\widehat X_1^\ast(t)),
\qquad
M_I:=\sup_{s\in I}\|\widehat X^\ast(s)\|,
\]
where \(I=[nh,(n+1)h]\). For every \(1\le r\le K-1\), there exists a constant
\(K_r>0\), independent of \(h, d,N,M,\nu,\varepsilon\), such that
\[
\sup_{t\in I}
\left\|
\frac{d^r}{dt^r}g(t)
\right\|
\le
K_r L_{\max}M_I(L_{\max}+M_I)^{r-1}.
\]
\end{lemma}

\begin{proof}
We prove the claim by induction on \(r\). In fact, we prove simultaneously the following two bounds: for each \(1\le j\le r\), there exist constants \(K_j,A_j>0\), independent of \(d,h,N\), such that
\begin{align}
\sup_{t\in I}
\left\|
\frac{d^j}{dt^j}g(t)
\right\|
&\le
K_jL_{\max}M_I(L_{\max}+M_I)^{j-1},
\label{eq:ind-hyp-g}
\\
\sup_{t\in I}
\left\|
\frac{d^j}{dt^j}\widehat X_1^\ast(t)
\right\|
&\le
A_jM_I(L_{\max}+M_I)^{j-1}.
\label{eq:ind-hyp-X}
\end{align}

For \(j=1\), since \(g(t)=\nabla U(\widehat X_1^\ast(t))\),
\[
\|g'(t)\|
=
\left\|
\nabla^2U(\widehat X_1^\ast(t))
\frac{d}{dt}\widehat X_1^\ast(t)
\right\|
\le
L_{\max}\|\widehat X_2^\ast(t)\|
\le
L_{\max}M_I.
\]
Thus \eqref{eq:ind-hyp-g} holds for \(j=1\). Also, from the chain structure of the dynamics,
\[
\sup_{t\in I}\left\|
\frac{d}{dt}\widehat X_1^\ast(t)\right\| = \sup_{t\in I}\left\|\widehat X_2^\ast(t)
\right\|
\le C M_I,
\]
so \eqref{eq:ind-hyp-X} holds for \(j=1\).

Now assume that \eqref{eq:ind-hyp-g} and \eqref{eq:ind-hyp-X} hold for all orders \(j<r\). We first prove \eqref{eq:ind-hyp-X} at order \(r\). By Lemma~\ref{lem:X1-deriv-bound},
\[
\sup_{t\in I}
\left\|
\frac{d^r}{dt^r}\widehat X^\ast(t)
\right\|
\le
C\left(
M_I+
\sum_{j=0}^{r-2}
\sup_{t\in I}
\left\|
\frac{d^j}{dt^j}g(t)
\right\|
\right).
\]
Using the induction hypothesis for the \(g\)-terms and the fact that
\(L_{\max}+M_I\ge 1\), we get
\[
\sup_{t\in I}
\left\|
\frac{d^r}{dt^r}\widehat X^\ast(t)
\right\|
\le
A_r M_I(L_{\max}+M_I)^{r-1}
\]
for a constant \(A_r\) independent of \(d,h,N\). This proves
\eqref{eq:ind-hyp-X} at order \(r\).

We now prove \eqref{eq:ind-hyp-g} at order \(r\). By Faà di Bruno's formula (\cite{constantine1996multivariate, mishkov2000generalization}),
\[
\left\|
\frac{d^r}{dt^r}g(t)
\right\|
\le
L_{\max}
\sum_{\substack{m_1,\ldots,m_r\ge0\\ \sum_{\ell=1}^r \ell m_\ell=r}}
C_{\mathbf m}
\prod_{\ell=1}^r
\left\|
\frac{d^\ell}{dt^\ell}\widehat X_1^\ast(t)
\right\|^{m_\ell},
\]
where the constants \(C_{\mathbf m}\) depend only on \(r\). Applying
\eqref{eq:ind-hyp-X} for all \(\ell\le r\), each product is bounded by
\[
C
M_I^{\sum_\ell m_\ell}
(L_{\max}+M_I)^{\sum_\ell(\ell-1)m_\ell}.
\]
Since \(\sum_\ell \ell m_\ell=r\), we have
\[
\sum_\ell(\ell-1)m_\ell
=
r-\sum_\ell m_\ell.
\]
Also \(\sum_\ell m_\ell\ge1\) and \(M_I\le L_{\max}+M_I\). Therefore
\[
M_I^{\sum_\ell m_\ell}
(L_{\max}+M_I)^{r-\sum_\ell m_\ell}
\le
M_I(L_{\max}+M_I)^{r-1}.
\]
Since the number of Faà di Bruno terms depends only on \(r\), all constants can be absorbed into \(K_r\). Hence
\[
\sup_{t\in I}
\left\|
\frac{d^r}{dt^r}g(t)
\right\|
\le
K_r L_{\max}M_I(L_{\max}+M_I)^{r-1}.
\]
This proves \eqref{eq:ind-hyp-g} at order \(r\), completing the induction.
\end{proof}

\begin{lemma}[Blockwise interpolation error]
\label{lem_appendix:lagrange-IR-uniform}
Let \(h=\Theta(1)\) be the step size  such that the moment estimate in
Lemma~\ref{lem:expectaion_power2k} applies. Let
\(P_2(t;\widehat X^\ast)\) be the blockwise Lagrange interpolant of
\[
    t\mapsto \nabla U(\widehat X_1^\ast(t))
\]
using \(K-1\) equispaced nodes on each block. Define
\[
I_U
:=
\sup_{t\in[nh,(n+1)h]}
\left\|
\nabla U(\widehat X_1^\ast(t))-P_2(t;\widehat X^\ast)
\right\|^2 .
\]
Under Assumptions~\ref{assump:U-strong-smooth}--\ref{assump:higher-smooth}, with
\(L_{\max}:=\max{\{1,\max_{1\le i\le K-1}L_i\}}\), there exists a constant \(C_{\rm IR}>0\),
independent of \(h, d,N,M,\nu,\varepsilon\), such that,
\[
\mathbb E I_U
\le
C_{\rm IR}
\frac{N h^{2K-2}}{M^{2K-2}(K-1)!^2}
\sum_{m=2}^{2K-2}
\binom{2K-4}{m-2}
L_{\max}^{2K-m}d^{m/2}.
\]
\end{lemma}

\begin{proof}
The standard Lagrange interpolation error estimate applied on each block \(B_i\) gives
\[
\sqrt{I_U}
\le
\frac{(h/M)^{K-1}}{(K-1)!}
\sup_{t\in[nh,(n+1)h]}
\left\|
\frac{d^{K-1}}{dt^{K-1}}
\nabla U(\widehat X_1^\ast(t))
\right\|.
\]
By Lemma~\ref{lem: inductive bound of time derivative of nabla U},
\[
\sup_{t\in[nh,(n+1)h]}
\left\|
\frac{d^{K-1}}{dt^{K-1}}
\nabla U(\widehat X_1^\ast(t))
\right\|
\le
C_{K-1}L_{\max}M_I(L_{\max}+M_I)^{K-2},
\]
where
\[
M_I:=\sup_{t\in[nh,(n+1)h]}\|\widehat X^\ast(t)\|.
\]
Substituting this bound into the interpolation estimate and squaring gives
\[
I_U
\le
C\,
\frac{h^{2K-2}}{M^{2K-2}(K-1)!^2}
L_{\max}^2M_I^2(L_{\max}+M_I)^{2K-4},
\]
where \(C\) is independent of \(h, d,N,M,\nu,\varepsilon\). Expanding the last factor by the binomial theorem, we obtain
\begin{equation}
I_U
\le
C
\frac{h^{2K-2}}{M^{2K-2}(K-1)!^2}
\sum_{m=2}^{2K-2}
\binom{2K-4}{m-2}
L_{\max}^{2K-m}M_I^m .
\label{eq:I_U}
\end{equation}

Using Lemma~\ref{lem:expectaion_power2k}, we have the even-moment bound
\[
\mathbb E[M_I^{2k}] \lesssim Nd^k,
\qquad
M_I:=\sup_{t\in I}\|\widehat X^\ast(t)\|.
\]
Hence, for each \(m\in\{0,\dots,2K-2\}\), choosing \(k=\lceil m/2\rceil\) and using Hölder's inequality, we obtain
\[
\mathbb E[M_I^m]
\le
\bigl(\mathbb E[M_I^{2\lceil m/2\rceil}]\bigr)^{m/(2\lceil m/2\rceil)}
\lesssim N d^{m/2}.
\]
Substituting this back into \eqref{eq:I_U} proves the lemma.
\end{proof}

\section{Picard convergence error}
\label{Appendix-F}

In this section, we control the error incurred by stopping the Picard iteration after a finite
number of steps. Recall that, on each interval \([nh,(n+1)h]\), the discretized operator \(\widehat T_y\) defines an implicit fixed-point equation because the blockwise Lagrange interpolant depends on the unknown values of the path at the collocation nodes. The algorithm approximates this fixed point by Picard iteration.

The analysis has two parts. First, we show that both the continuous operator \(T_y\) and the discretized operator \(\widehat T_y\) admit fixed points, and that \(\widehat T_y\) is a contraction
under the sup norm for sufficiently small step size. Second, we use this contraction to bound the distance between the fixed point \(\widehat X_y^\ast\) and the finite Picard iterate
\(\widehat X_y^{[\nu]}\). This yields the Picard convergence error used in the one-step decomposition.
\subsection{Existence of fixed point}
\label{Appendix-F1}

In this subsection, we provide the proofs of the existence of fixed point for operators $\mathcal{T}_y$ and $\widehat{\mathcal{T}}_y$ as defined in Appendix~\ref{appendix-A}.

\subsubsection{Operator $\mathcal{T}_y$}
\label{Appendix-F11}

\begin{lemma}[Existence and uniqueness of the fixed point]
\label{lem: Tau_existence-uniqueness_fixed_point}
Under Assumption~\ref{assump:U-strong-smooth}, let \(L_H:=\max\{L,1\}\) be the Lipschitz constant of \(\nabla H\). Equip
\(\mathcal{C}([0,h],\mathbb{R}^{Kd})\) with the uniform norm \(\|X\|_\infty:=\sup_{t\in[0,h]}\|X(t)\|
\). There exists
\[
    h_*:=\frac{1}{3\gamma L_H}>0
\]
such that, for any \(0<h<h_*\) and any \(y\in\mathbb{R}^{Kd}\), the operator
\(\mathcal T_y\) admits a unique fixed point in
\(\mathcal{C}([0,h],\mathbb{R}^{Kd})\).
\end{lemma}

\begin{proof}
The space \(\bigl(\mathcal{C}([0,h],\mathbb{R}^{Kd}),\|\cdot\|_\infty\bigr)\) is complete. Let \(X,Y\in\mathcal{C}([0,h],\mathbb{R}^{Kd})\), and define
\(\widetilde X:=\mathcal T_y[X]\), \(\widetilde Y:=\mathcal T_y[Y].\)
The additive noise cancels in the difference, so
\[
\widetilde X(t)-\widetilde Y(t)
= -\int_0^t (D+Q)
\Big(\nabla H(X(s))-\nabla H(Y(s))\Big)\,ds .
\]
Taking norms and using the Lipschitzness of \(\nabla H\), we obtain
\[
\|\widetilde X(t)-\widetilde Y(t)\|
\le
\int_0^t \|D+Q\|\,
\|\nabla H(X(s))-\nabla H(Y(s))\|\,ds
\le
\|D+Q\|\,L_H
\int_0^t \|X(s)-Y(s)\|\,ds .
\]
Thus, for \(t\in[0,h]\),
\[
\|\widetilde X(t)-\widetilde Y(t)\|
\le
\|D+Q\|\,L_H\,t\,\|X-Y\|_\infty
\le
3\gamma L_H h\,\|X-Y\|_\infty,
\]
where in the last inequality we used \(\|D+Q\|\le 3\gamma\) from
Lemma~\ref{lemma:d_plus_q}. Taking the supremum over \(t\in[0,h]\) yields
\[
\|\mathcal T_y[X]-\mathcal T_y[Y]\|_\infty
=
\|\widetilde X-\widetilde Y\|_\infty
\le
3\gamma L_H h\\,\|X-Y\|_\infty.
\]
If \(h<h_\ast:=(3\gamma L_H)^{-1}\), then \(\mathcal T_y\) is a contraction.
By Banach's fixed point theorem, \(\mathcal T_y\) has a unique fixed point in
\(\mathcal{C}([0,h],\mathbb{R}^{Kd})\).
\end{proof}

\subsubsection{Operator $\widehat{\mathcal{T}}_y$}
\label{Appendix-F12}

On each block \(B_i\) we let \(\{\ell_{i,j}\}_{j=1}^{K-1}\) denote the Lagrange basis associated with the
\(K-1\) equispaced nodes on \(B_i\). These local basis functions are obtained from the
reference basis \(\{\widetilde\ell_j\}_{j=1}^{K-1}\) on \([0,1]\) by the affine change of
variables
\[
    \tau=\frac{M}{h}\left(t-\frac{ih}{M}\right).
\]
Consequently, for every \(t\in B_i\),
\[
    \sum_{j=1}^{K-1}|\ell_{i,j}(t)|
    =
    \sum_{j=1}^{K-1}|\widetilde\ell_j(\tau)|
    \le
    \Gamma_\phi,
\]
where
\[
    \Gamma_\phi
    :=
    \sup_{\tau\in[0,1]}
    \sum_{j=1}^{K-1}|\widetilde\ell_j(\tau)|.
\]
Lemma~\ref{lem:bound-for-basis} gives
\[
    \Gamma_\phi
    \le
    \frac{2^{K-2}(K-2)^{K-2}}{(K-2)!},
\]
so \(\Gamma_\phi\) depends only on the number of collocation nodes, equivalently only on
\(K\). We then have the following lemma.

\begin{lemma}[Contraction of the discretized operator]
\label{lem: T_phi_contraction_sup_norm}
Assume Assumption~\ref{assump:U-strong-smooth}. Equip
\(\mathcal{C}([0,h],\mathbb{R}^{Kd})\) with the uniform norm \(
    \|X\|_\infty := \sup_{t\in[0,h]}\|X(t)\|.\)
Then, for every
\[
    0<h\le
    \min\left\{
        \frac{\log 2}{3\gamma},\,
        \frac{1}{20L\Gamma_\phi}
    \right\},
\]
the operator \(\widehat{\mathcal T}_y\) is a contraction with rate $\rho = 2L\,\Gamma_\phi \,h <\frac{1}{10}$ on \(\mathcal{C}\!\left([0,h],\mathbb{R}^{Kd}\right)\) with respect to \(\|\cdot\|_\infty\).
\end{lemma}

\begin{proof}
Let \(\tilde X=\widehat{\mathcal T}_y[X]\) and \(\tilde Y=\widehat{\mathcal T}_y[Y]\). 
Let \(A:=-(D+Q)J\) with \(J:=\mathrm{diag}(0,1,\ldots,1)\otimes I_d\) and \(e_2=(0,1,0,\ldots,0)^\top\in\mathbb{R}^{K}\). Then from \eqref{eq: operator_hat_T2},
\begin{align}
\tilde{X}(t)
=
y + \int_{0}^{t} A\,\tilde{X}(s)\,ds \;-\; \int_{0}^{t} (e_2\otimes I_d)\,P(s;X)\,ds \;+\; \int_{0}^{t} \sqrt{2D}\,dB_s. \label{eq21_block}
\end{align}
Since \(\|e_2\otimes I_d\|=1\),
\begin{align*}
\|\tilde X(t)-\tilde Y(t)\|
&\le \int_{0}^{t} \|A\|\,\|\tilde X(s)-\tilde Y(s)\|\,ds
 \;+\; \int_{0}^{t} \|P(s;X)-P(s;Y)\|\,ds.
\end{align*}

Moreover, we have,
\[
\int_0^t \|P(s;X)-P(s;Y)\|\,ds
\;\le\; \sum_{i=1}^{M} \int_{B_i \cap [0,t]} \|P_{B_i}(s;X)-P_{B_i}(s;Y)\|\,ds.
\]

By Lipschitzness of \(\nabla U\) and the definition of the Lagrange interpolation on each block (using the reparameterization remark), we have for each \(i\),
\begin{align*}
\|P_{B_i}(s;X)-P_{B_i}(s;Y)\|
&\le L \sum_{j=1}^{K-1} |\ell_{i,j}(s)|\,\|X-Y\|_\infty \\
&= L \sum_{j=1}^{K-1} |\tilde\ell_{j}(s)|\,\|X-Y\|_\infty \\
&\le L\,\Gamma_\phi\,\|X-Y\|_\infty.
\end{align*}

Summing over at most \(M\) blocks yields
\[
\int_0^t \|P(s;X)-P(s;Y)\|\,ds
\;\le\; h\,L\,\Gamma_\phi\,\|X-Y\|_\infty.
\]

Therefore,
\[
\|\tilde X(t)-\tilde Y(t)\|
\;\le\; \int_{0}^{t} \|A\|\,\|\tilde X(s)-\tilde Y(s)\|\,ds
\;+\; L\,\Gamma_\phi\, h\,\|X-Y\|_\infty.
\]

Applying Grönwall’s inequality, for \(t\le h\),
\[
\|\tilde X(t)-\tilde Y(t)\|
\;\le\; e^{\|A\| t}\, L\,\Gamma_\phi\, h\,\|X-Y\|_\infty
\;\le\; e^{3\gamma h}\, L\,\Gamma_\phi\, h\,\|X-Y\|_\infty,
\]
where we used Lemma~\ref{lemma:d_plus_q}.

Choose \(h\le \frac{\log 2}{3\gamma}\) so that \(e^{3\gamma h}\le 2\), and \(h\le (20 L \Gamma_\phi)^{-1}\) so that the product is at most \(1/10\). Taking the supremum over \(t\in[0,h]\) yields
\[
\|\widehat{\mathcal T}_y[X] - \widehat{\mathcal T}_y[Y]\|_\infty
\;\le\; \tfrac{1}{10}\,\|X-Y\|_\infty,
\]
which proves the contraction.
\end{proof}

\begin{remark}[Fixed-point interpretation of the interpolated dynamics]
\label{rem:fixed-point-discrete-sde}
Under the stepsize condition of Lemma~\ref{lem: T_phi_contraction_sup_norm}, the operator
\(\widehat{\mathcal T}_y\) is a contraction on
\(\mathcal C([0,h];\mathbb R^{Kd})\). Hence, by Banach's fixed point theorem, it has a unique fixed point, denoted by \(\widehat X_y^\ast\). By the definition of \(\widehat{\mathcal T}_y\), this fixed point satisfies
\[
\widehat X_y^\ast(t)
=
y+\int_0^t A\,\widehat X_y^\ast(s)\,ds
-\int_0^t (e_2\otimes I_d)\,P(s;\widehat X_y^\ast)\,ds
+\int_0^t \sqrt{2D}\,dB_s ,
\]
or equivalently,
\[
d\widehat X(t)
=
A\widehat X(t)\,dt
-(e_2\otimes I_d)P(t;\widehat X)\,dt
+\sqrt{2D}\,dB_t,
\qquad
\widehat X(0)=y.
\]
Thus the Picard iterates generated by \(\widehat{\mathcal T}_y\) converge uniformly on \([0,h]\) to the unique solution of the interpolated dynamics.
\end{remark}

\begin{remark}[Geometric convergence of Picard iterates]\label{rem:geometric-picard}
Under the contraction property in Lemma~\ref{lem: T_phi_contraction_sup_norm}, the Picard iterates converge geometrically in the uniform norm. In particular, for all $\nu\ge 1$,
\[
\|\widehat X^{[\nu]}-\widehat X^{[\nu-1]}\|_\infty
\le \rho^{\nu-1}\,\|\widehat X^{[1]}-\widehat X^{[0]}\|_\infty.
\]
Summing the telescoping series, we obtain
\[
\|\widehat X^{[\nu]}-\widehat X^{[0]}\|_\infty
\le \sum_{k=0}^{\nu-1} \|\widehat X^{[k+1]}-\widehat X^{[k]}\|_\infty
\le \sum_{k=0}^{\nu-1} \rho^k \|\widehat X^{[1]}-\widehat X^{[0]}\|_\infty
= \frac{1-\rho^\nu}{1-\rho}\,\|\widehat X^{[1]}-\widehat X^{[0]}\|_\infty.
\]
In particular, letting $\nu\to\infty$ yields
\[
\|\widehat X_y^\ast-\widehat X^{[0]}\|_\infty
\le \frac{1}{1-\rho}\,\|\widehat X^{[1]}-\widehat X^{[0]}\|_\infty.
\]
\end{remark}

\subsection{Bounding the Picard error}
\label{Appendix-D4}

With the fixed-point existence and contraction estimates in hand, we are now ready to prove the finite Picard-iteration error bound. The goal is to bound the distance between \(\widehat X^\ast_y\) and \(\widehat X^{[\nu]}_y\) uniformly over one step.

\begin{proposition}[Picard iteration error]\label{prop-appendix:picard-error}
Fix a step \(n\), let \(y:=\widehat X(nh)\), and define the Picard iterates
\[
\widehat X_y^{[0]}(t)\equiv y,\qquad
\widehat X_y^{[\nu+1]}:=\widehat{\mathcal T}_y[\widehat X_y^{[\nu]}].
\]
Assume the step-size condition of Lemma~\ref{lem: T_phi_contraction_sup_norm}, so that
\(
\rho:=2Lh\Gamma_\phi\le \frac{1}{10} .
\)
Let \(\widehat X_y^\ast\) be the fixed point of \(\widehat{\mathcal T}_y\). Then
\[
\mathbb E\sup_{t\in[0,h]}
\bigl\|\widehat X_y^\ast(t)-\widehat X_y^{[\nu]}(t)\bigr\|^2
\le
\frac{\rho^{2\nu}}{(1-\rho)^2}
\left[
C h^2\,\mathbb E\|\widehat X(nh)-X(nh)\|^2
+
Cdh
\right],
\]
where \(C>0\) is independent of \(h, d,N,M,\nu,\varepsilon\).
\end{proposition}

\begin{proof}
By the contraction property of \(\widehat{\mathcal T}_y\) (See Lemma~\ref{lem: T_phi_contraction_sup_norm} and  Remark \eqref{rem:geometric-picard}),
\[
\sup_{t\in[0,h]}
\bigl\|\widehat X_y^\ast(t)-\widehat X_y^{[\nu]}(t)\bigr\|
\le
\frac{\rho^\nu}{1-\rho}
\sup_{t\in[0,h]}
\bigl\|\widehat X_y^{[1]}(t)-\widehat X_y^{[0]}(t)\bigr\|.
\]
It remains to bound the first Picard increment. Since
\(\widehat X_y^{[0]}(t)\equiv y=\widehat X(nh)\), the first increment satisfies
\[
\widehat X_y^{[1]}(t)-\widehat X_y^{[0]}(t)
=
\int_0^t A\bigl(\widehat X_y^{[1]}(s)-\widehat X_y^{[0]}(s)\bigr)\,ds
-
\int_0^t (D+Q)
\begin{bmatrix}
\nabla U(\widehat X_1(nh))\\
\widehat X_2(nh)\\
\vdots\\
\widehat X_K(nh)
\end{bmatrix}ds
+
\int_0^t\sqrt{2D}\,dB_s .
\]

By Cauchy--Schwarz, we first obtain
\begin{align}
\mathbb{E}\sup_{t\in[0,h]}
\|\widehat X_y^{[1]}(t)-\widehat X_y^{[0]}(t)\|^2
&\le
3h\|A\|^2\int_0^h
\mathbb{E}\sup_{u\in[0,s]}
\|\widehat X_y^{[1]}(u)-\widehat X_y^{[0]}(u)\|^2\,ds \notag\\
&\quad
+3h^2\|D+Q\|^2
\Big(
\mathbb E\|\nabla U(\widehat X_1(nh))\|^2
+\mathbb E\|\widehat X(nh)\|^2
\Big) \notag\\
& +24\gamma dh . \label{eq:first-picard-increment-pre}
\end{align}
For the last term we used Doob’s \(L^{2}\) inequality and the Itô isometry:
\begin{align*}
\mathbb{E}\,\sup_{t\in[0,h]}\Big\|\int_{0}^{t} \sqrt{2D}\,dB_s\Big\|^{2}
\le 4\,\mathbb{E}\,\Big\|\int_{0}^{h} \sqrt{2D}\,dB_s\Big\|^{2}
= 4\,\mathrm{tr}(2D)\,h
= 8\,\mathrm{tr}(D)\,h
= 8\gamma dh.
\end{align*}

For the second term we add and subtract \(X(nh)\) to get
\begin{align*}
\mathbb{E}\,\|\nabla U(\widehat X_{1}(nh))\|^{2}
&\le 2L^{2}\,\mathbb{E}\,\|\widehat X(nh)-X(nh)\|^{2}
   + 2\,\mathbb{E}\,\|\nabla U(X_{1}(nh))\|^{2},\\
\mathbb{E}\,\|\widehat X(nh)\|^{2}
&\le 2\,\mathbb{E}\,\|\widehat X(nh)-X(nh)\|^{2}
   + 2\,\mathbb{E}\,\|X(nh)\|^{2}.
\end{align*}

Using the Lipschitzness of \(\nabla U\), the centering assumption~\eqref{assump:centered}, and Lemma~\ref{lem:StationaryMoments}, we get,
\[
\mathbb E\|\nabla U(\widehat X_1(nh))\|^2
+\mathbb E\|\widehat X(nh)\|^2
\le
C\,\mathbb E\|\widehat X(nh)-X(nh)\|^2
+
Cd .
\]
Substituting this into \eqref{eq:first-picard-increment-pre} and applying
Grönwall's inequality gives
\[
\mathbb{E}\sup_{t\in[0,h]}
\|\widehat X_y^{[1]}(t)-\widehat X_y^{[0]}(t)\|^2
\le
Ce^{3h^{2}\|A\|^{2}} h^2\,\mathbb E\|\widehat X(nh)-X(nh)\|^2
+
Ce^{3h^{2}\|A\|^{2}}dh,
\]
Combining this with the geometric bound above gives the claim.
\end{proof}

\section{Technical bounds for the interpolation error}
\label{Appendix-G}

\subsection{Moment bounds for the fixed-point path}
\label{subsection:part2}

In this subsection, we derive moment bounds for $\sup_{t\in[0,1]}\|\widehat X^*(nh+th)\|$. 
Before proceeding, we first introduce the following uniform bound for Brownian motion.
\begin{lemma}
\label{lem:nice_events}
Let $(B_t)_{t\ge0}$ be a $d$-dimensional standard Brownian motion,  \(h > 0\) denote the stepsize and \(N \ge 1\) be an integer denoting the maximum number of iterations. For $C_b\ge0$ define the events
\begin{equation}
\mathcal G_n(h,C_b)\;:=\;\Bigl\{\sup_{0\le t\le h}\,\|B_{nh+t}-B_{nh}\|\;\le\;C_b\Bigr\},
\qquad n=0,1,\dots,N-1.
\label{eq:nice_events}
\end{equation}
Then
\[
\mathbb P\!\left(\bigcap_{n=0}^{N-1}\mathcal G_n(h,C_b)\right)
\;\ge\;1-3N\,\exp\!\left(-\frac{C_b^{2}}{6dh}\right).
\]
In particular, for any $\delta\in(0,1)$, choosing
\[
C_b\;=\; \sqrt{\,6dh\,\log\!\frac{3N}{\delta}\,}
\]
ensures
\(
\mathbb P\!\left(\bigcap_{n=0}^{N-1}\mathcal G_n(h,C_b\right)\ge 1-\delta.
\)
\end{lemma}

\begin{proof}
By stationary increments, for each $n$,
\(
\sup_{0\le t\le h}\|B_{nh+t}-B_{nh}\|
\overset{d}= \sup_{0\le t\le h}\|B_t\|.
\)
The one-interval tail bound
\(
\mathbb P\!\left(\sup_{0\le t\le h}\|B_t\|\ge C_b\right)
\le 3\exp\!\left(-C_b^2/(6dh)\right)
\)
(See, i.e., \citet[lemma 34]{chewi2024analysis})
then implies, by a union bound over $n=0,\dots,N-1$,
\[
\mathbb P\!\left(\bigcup_{n=0}^{N-1}\mathcal G_n(h,C_b)^c\right)
\le \sum_{n=0}^{N-1} 3\exp\!\left(-\frac{C_b^2}{6dh}\right)
= 3N\exp\!\left(-\frac{C_b^2}{6dh}\right),
\]
which yields the claim.
\end{proof}

Given a starting point $y$, we analyze the first Picard iterate on $[0,h]$.
On the event where the Brownian increment is suitably bounded (the “nice” event),
the following lemma provides an explicit upper bound on its deviation from $y$.

\begin{lemma}\label{lem: First_Picard_increment_bound}
Let $y \in \mathbb{R}^{Kd}$ and $y_1$ be the first $d$ element of $y$. Define the Picard iterates on $[0,h]$ by
$\widehat X_y^{[0]}(t)\equiv y$ and
$\widehat X_y^{[1]} := \widehat{\mathcal T}_y\big[\widehat X_y^{[0]}\big]$.
On the nice event \eqref{eq:nice_events}, 
 if $h\le (\log 2)/\|A\|$, then:
\[
\sup_{0\le t\le h}\bigl\|\widehat X^{[1]}_y(t)-y\bigr\|
\;\le\;
2\left[
h\,(L+1)\,\|D{+}Q\|\|y\|
+\sqrt{2\gamma}\,C_b
\right].
\]
\end{lemma}

\begin{proof}
\begin{align*}
\sup_{0\le u\le t}\bigl\|\widehat X^{[1]}_y(u)-y\bigr\|
&\le \|A\|\int_0^{t}\sup_{0\le r\le s}\bigl\|\widehat X^{[1]}_y(r)-y\bigr\|\,\mathrm ds \\
&\quad + h\,\|D{+}Q\|\Bigl(L\|y_1\|+\|y\|\Bigr)
     + \sqrt{2\gamma}\,\sup_{0\le u\le h}\bigl\|B_{u}\bigr\|.
\end{align*}
Applying Grönwall’s inequality at $t=h$ gives
\[
\sup_{0\le u\le h}\bigl\|\widehat X^{[1]}_y(u)-y\bigr\|
\le
e^{\|A\|h}\!\left[
h\,\|D{+}Q\|\Bigl(L\|y_1\|+\|y\|\Bigr)
+ \sqrt{2\gamma}\,\sup_{0\le u\le h}\bigl\|B_{u}\bigr\|
\right].
\]
Using \(h\le (\log 2)/\|A\|\) we get the desired bound.
\end{proof}
We now establish bounds on the discrete evolution of the algorithm across steps.
\begin{lemma}
\label{lem:discrete_contraction}
There exist constants \(0<h_1'\le h_1\) \eqref{eq: lem_18_h1}, depending only on the fixed problem
parameters, such that for every constant step size \(h\in[h_1',h_1]\), the
bounds used below hold on nice event \eqref{eq:nice_events}:
\begin{align}
\|\widehat X((n+1)h)\|_S^2 \lesssim d \log \left(\frac{3N}{\delta}\right).
\end{align}
\end{lemma}
\begin{proof}
At each algorithmic step, we focus on the final Picard iteration $\nu$. Note that at the start of each algorithmic step we initialize the Picard iterations at the current state, i.e., set $y := \widehat X(nh)$ (For notational convenience, we omit $y$ in the proof.) Then the following relation holds:
\begin{align}\label{eq: high_prob_1}
\widehat X((n+1)h) = \widehat X^{[\nu]}(h)
&= \widehat X(nh)
   - \int_{0}^{h} (D{+}Q)
      \begin{bmatrix}
         P(s; \widehat X^{[\nu-1]})\\
         \widehat X^{[\nu]}_2(s) \\
         \vdots\\
         \widehat X^{[\nu]}_K(s)
      \end{bmatrix}\mathrm ds
   + \sqrt{2}\!\int_{0}^{h} D\,\mathrm dB_s
\nonumber\\
&=\;
   \underbrace{\Biggl(
   \widehat X(nh)
   - \int_{0}^{h} (D{+}Q)
      \begin{bmatrix}
        \nabla U(\widehat X_1(nh))\\
        \widehat X_2(nh) \\
        \vdots\\
        \widehat X_K(nh)
      \end{bmatrix}\mathrm ds
   \Biggr)}_{\text{(I)}}
   \nonumber\\
   &\qquad+\;
   \underbrace{\Biggl(
   - \int_{0}^{h} (D{+}Q)
      \begin{bmatrix}
        P(s;\widehat X^{[\nu-1]}) - \nabla U(\widehat X_1(nh))\\
        \widehat X^{[\nu]}_2(s) - \widehat X_2(nh)\\[-2pt]
        \vdots\\[-2pt]
        \widehat X^{[\nu]}_K(s) - \widehat X_K(nh)
      \end{bmatrix}\mathrm ds
   \Biggr)}_{\text{(II)}} \nonumber\\
  & \qquad \;+\;
   \sqrt{2}\!\int_{0}^{h}\sqrt{D}\,\mathrm dB_s.
\end{align}

Consider these two terms separately (first term in the $S$-norm):
\begin{align*}
(\mathrm I)
&=\Biggl\|
   \widehat X(nh)
   - \int_{0}^{h} (D{+}Q)
      \begin{bmatrix}
        \nabla U(\widehat X_1(nh))\\
        \widehat X_2(nh)\\
        \vdots\\
        \widehat X_K(nh)
      \end{bmatrix}\mathrm ds
   \Biggr\|_S^2\\
&=\Biggl\|
   \widehat X(nh)
   + \int_{0}^{h} \int_{0}^{1}  J_b(\lambda \widehat{X}(nh)) \mathrm d\lambda 
      \begin{bmatrix}
        \widehat X_1(nh)\\
        \widehat X_2(nh)\\
        \vdots\\
        \widehat X_K(nh)
      \end{bmatrix}\mathrm ds
   \Biggr\|_S^2\\
&= \|\widehat X(nh)\|_S^2
   + 2h \int_{0}^{1}\,\widehat X(nh)^\top SJ_b(\lambda \widehat{X}(nh)) \widehat X(nh) \mathrm d\lambda 
   + h^2 \|S\|\|J_b\|^2 \|\widehat X(nh)\|^2\\
&\le \|\widehat X(nh)\|_S^2
   - 2h\,C_S\|\widehat X(nh)\|_S^2
   + h^2\|S\|\|S^{-1}\|\|J_b\|^2
     \|\widehat X(nh)\|_S^2.
\end{align*}

Here we apply the same mean-value theorem argument as in Proposition~\ref{prop-appendix:convergence_S-norm}, using the two endpoints \(\widehat X(nh)\) and \(\mathbf{0}\), to obtain the final inequality. For
\[
h\le
\min\left\{
\frac{1}{C_S},
\frac{C_S}{4\|S\|\|S^{-1}\|\|J_b\|^2}
\right\},
\]
we have
\[
(1+C_Sh)\|\mathrm I\|_S^2
\le
\left(1-\frac12 C_Sh\right)\|\widehat X(nh)\|_S^2 .
\]

Let us deal with the second term now (in $l_2$ norm):

\begin{alignat}{2}
(\mathrm{II})
&\;=\;&
&\left \| \int_{0}^{h} (D{+}Q)
      \begin{bmatrix}
         P\big(s;\,\widehat X^{[\nu-1]}\big) -  \nabla U\big(\widehat X_1(nh)\big)\\
        \widehat X^{[\nu]}_2(s) - \widehat X_2(nh)\\[-2pt]
        \vdots\\[-2pt]
        \widehat X^{[\nu]}_K(s) - \widehat X_K(nh)
      \end{bmatrix}\,\mathrm ds \right\|_2 \nonumber\\
&\;\leq\;&
&\sum_{i=1}^M \int_{B_i} \left\| (D{+}Q)
      \begin{bmatrix}
         P_{B_i}\big(s;\,\widehat X^{[\nu-1]}\big) -  \nabla U\big(\widehat X_1(nh)\big)\\
         0\\[-2pt]
        \vdots\\[-2pt]
        0
      \end{bmatrix} \right\|_2 \,\mathrm ds \nonumber\\
&&&\hspace{10em}+
        \int_0^h
        \left\|
        (D+Q)
        \begin{bmatrix}
        0\\
        \widehat X^{[\nu]}_2(s)-\widehat X_2(nh)\\
        \vdots\\
        \widehat X^{[\nu]}_K(s)-\widehat X_K(nh)
        \end{bmatrix}
        \right\|_2 \,\mathrm ds \label{eq:ineq0_gen}\\
&\;\leq\;&
&Lh\,\Gamma_{\phi}\,\|D{+}Q\|
      \sup_{s \in [0,h]}  \big\|\widehat X^{[\nu-1]}(s) -  \widehat X(nh)\big\| \nonumber\\
&&&\hspace{10em}+
      h\,\|D{+}Q\|
      \sup_{s \in [0,h]}\big\|  \widehat X^{[\nu]}(s) - \widehat X(nh)\big\|
      \label{eq:ineq1_gen} \\
&\;\leq\;&
&\rho\,\|D{+}Q\|
      \sup_{s \in [0,h]}  \big\|\widehat X^{[\nu-1]}(s) -  \widehat X(nh)\big\| \hspace{20em}\nonumber\\
&&&\hspace{10em}+
      h\,\|D{+}Q\|
      \sup_{s \in [0,h]} \big\|  \widehat X^{[\nu]}(s) - \widehat X(nh)\big\|
      \label{eq:ineq2_gen}
\end{alignat}

\begin{align}
&\;\leq\;&
&\rho\,\|D{+}Q\| \frac{1-\rho^{\nu -1}}{1-\rho}
      \sup_{s \in [0,h]}  \big\|\widehat X^{[1]}(s) -  \widehat X(nh)\big\| \nonumber\\
&&&\hspace{10em}+
      h\,\|D{+}Q\|\frac{1-\rho^{\nu}}{1-\rho}
      \sup_{s \in [0,h]} \big\|  \widehat X^{[1]}(s) - \widehat X(nh)\big\|
      \label{eq:ineq3_gen}\\
&\;\leq\;&
&2\!\left[
      \rho\,\|D{+}Q\| \frac{1-\rho^{\nu-1}}{1-\rho}
      + h\,\|D{+}Q\|\frac{1-\rho^{\nu}}{1-\rho}
      \right] \nonumber\\
&&&\hspace{10em} \times
      \left[
      h\,\|D{+}Q\|\bigl(L\|\widehat X_1(nh)\|+\|\widehat X(nh)\|\bigr)
      + \sqrt{2\gamma}\,C_b
      \right].
      \label{eq:ineq4_gen}
\end{align}

We now explain each inequality in detail:
\begin{itemize}
 \item \textbf{\eqref{eq:ineq0_gen}} follows since \(P(s;X) = P_{B_i}(s;X)\) for \(s \in B_i\), which allows us to decompose the integral blockwise.
  \item \textbf{\eqref{eq:ineq1_gen}} follows from the Lipschitz continuity of \(\nabla U\) and the definition of the interpolant \(P_{B_i}(s;\widehat X^{[\nu-1]})\). Using the identity \(\sum_j \ell_{i,j}(s)=1\), we rewrite the difference and then apply the triangle inequality along with the Lipschitz continuity of \(\nabla U\).
  
  \item \textbf{\eqref{eq:ineq2_gen}} applies the definition of 
  \(\rho = 2Lh\,\Gamma_{\phi}\) used in Lemma~\ref{lem: T_phi_contraction_sup_norm}.
  
  \item \textbf{\eqref{eq:ineq3_gen}} follows from the Picard contraction property. In particular, by Lemma~\ref{lem: T_phi_contraction_sup_norm}, the operator contracts with rate \(\rho = 2Lh\,\Gamma_{\phi}\). Applying the geometric convergence bound in Remark~\ref{rem:geometric-picard}, we obtain the desired estimate via a standard geometric series argument.
  
  \item \textbf{\eqref{eq:ineq4_gen}} follows from Lemma~\ref{lem: First_Picard_increment_bound}, 
  which bounds the first Picard increment.
\end{itemize}

Next, we express $(\mathrm{II})$ in terms of the current state $\widehat X(nh)$, noting that with $\rho := 2L h\,\Gamma_{\phi}\le \tfrac{1}{2}$ we have

\begin{align*}
 (\mathrm{II})   &\leq 2h\,\|D{+}Q\|^2(1+L)\left[\rho\frac{1-\rho^{\nu-1}}{1-\rho} + h\frac{1-\rho^{\nu}}{1-\rho}\right]\|\widehat X(nh)\| \nonumber\\
& \qquad + 2\sqrt{2\gamma}\,C_b \|D{+}Q\|\left[\rho \, \frac{1-\rho^{\nu-1}}{1-\rho} + h \,\frac{1-\rho^{\nu}}{1-\rho}\right] \nonumber\\
&\leq 2h\,\|D{+}Q\|^2(1+L)(2\rho +2h)\|\widehat X(nh)\|
 \;+\; 2\sqrt{2\gamma}\,C_b \,\|D{+}Q\|(2\rho +2h)\, \qquad (\because\rho \leq 1/2)\nonumber \\
 &\leq 4h^2\,\|D{+}Q\|^2(1+L)(1+2L\Gamma_{\phi})\|\widehat X(nh)\| \\
 & \hspace{2cm}\;+\; 4\sqrt{2\gamma}\,C_b \,\|D{+}Q\|(1+2L\Gamma_{\phi}) h \qquad (\because\rho = 2L h\,\Gamma_{\phi})\nonumber \\
 &= C_1 h^2 \|\widehat X(nh)\| + C_2 C_bh,
\end{align*}
where \(C_1,C_2>0\) are the corresponding constants independent of
\(\widehat X(nh)\), \(C_b\), and \(h\).
We now proceed to analyze the recursion in the $S$-norm.  
Combining \eqref{eq: high_prob_1} with the bounds for terms $(\mathrm{I})$ and $(\mathrm{II})$ (Converting into $S$-norm), we obtain
\begin{align}
 \|\widehat X((n+1)h)\|_S^2 
 &\leq \left(1-\frac{C_{S}}{2}h\right)\|\widehat X(nh)\|_{S}^2\nonumber \\\
 &+\left(2+ \frac{2}{C_{{S}}h}\right)(2C_1^2h^4 \|S\|\|S^{-1}\|\|\widehat X(nh)\|_{S}^2 + 2 \|S\| C_2^2C_b^2h^2) \nonumber \\\
 &+ \left(4+ \frac{4}{C_{{S}}h}\right) \gamma C_b^2 \|S\| \nonumber \\\
 &\lesssim \left(1-\frac{C_{S}}{4}h\right) \|\widehat X(nh)\|_{S}^2  + C_b^2 h + \frac{1}{h}  C_b^2 \|S\| \nonumber
\end{align}
where the last inequality is obtained by choosing a sufficiently small stepsize \(h\) such that  

\[
h \le
\min\left\{
\frac{1}{C_S},\,
\frac{C_S}{\sqrt{32}\,C_1\sqrt{\|S\|\,\|S^{-1}\|}}
\right\}.
\]
Let \(h_1>0\) be chosen small enough so that all small-step conditions used in
the proof hold; in particular,
\begin{equation}
h_1
\le
\min\left\{
\frac{1}{C_S},\,
\frac{C_S}{\sqrt{32}\,C_1\sqrt{\|S\|\,\|S^{-1}\|}},\,
\frac{C_S}{4\|S\|\|S^{-1}\|\|J_b\|^2}
\right\}, \quad h_1' =  ch_1,
\label{eq: lem_18_h1}
\end{equation}
for some \(0<c<1\). Throughout this proposition, we consider constant step size \(h \in [h'_1, h_1]\). The  lower bound needed for the lemma (\(h'_1\)) will be fixed when \(h\) is
chosen in the proof of the main theorem. Starting from the initialization $\widehat X(0)=0$, the recursion yields for all $n$,
\begin{align*}
\|\widehat X((n+1)h)\|_S^2 \lesssim d \log \left(\frac{3N}{\delta}\right)
\end{align*}
where we used the definition of \(C_b\), together with the fact that \(h\) is bounded above and below by positive constants.
\end{proof}

Next we bound the Picard fixed point initialized at the current state $\widehat X(nh)$.

\begin{lemma}
\label{lem:high-prob-bound_fixed_point}
Assume the step size condition in Lemma~\ref{lem: First_Picard_increment_bound} and ~\ref{lem:discrete_contraction}. Then the following bound holds with probability at least \(1-\delta\):
\begin{equation}
\label{eq:high-prob-bound}
\sup_{t \in [0,1]}\|\widehat X^\ast_y(nh + th)\| 
\lesssim
\left[\, d \,\log\!\frac{3N}{\delta} \right]^{1/2},
\end{equation}
where $y=\widehat X(nh)$.
\end{lemma}

\begin{proof}
For notational convenience, we suppress the subscript $y$, though we emphasize that the Picard fixed point is initialized at $y=\widehat X(nh)$. Under the event $\bigcap_{n=0}^{N-1}\mathcal G_n(h,C_b)$, we have
\begin{align}
\sup_{t \in [0,1]}\|\widehat X^*(nh + th)\|^2 
&\leq  2\sup_{t \in [0,1]}\|\widehat X^*(nh + th) - \widehat X(nh)\|^2 + 2\|\widehat X(nh)\|^2 \nonumber\\
&\leq  \frac{2}{(1-\rho)^2}\sup_{t \in [0,1]}\|\widehat X^{[1]}(nh + th) - \widehat X(nh)\|^2 + 2\|\widehat X(nh)\|^2 \enspace \text{(by rem \eqref{rem:geometric-picard})}\nonumber\\
&\lesssim \left[2+h^2 \right] \|\widehat X(nh)\|^2
+ C_b^2 \qquad \text{(by Lemma~\ref{lem: First_Picard_increment_bound})},\nonumber
\end{align}

Converting $l_2$ norm into $ S$ norm and applying  Lemma~\ref{lem:discrete_contraction}, we have,
\begin{align}
\sup_{t \in [0,1]}\|\widehat X^*(nh + th)\|^2 &\leq (2+h^2)\|S^{-1}\|\|\widehat{X}(nh)\|_S^2 + C_b^2\nonumber\\
& \lesssim (2+h^2)\|S^{-1}\|d \log \left(\frac{3N}{\delta}\right) + C_b^2\nonumber\\
& \lesssim d \log \left(\frac{3N}{\delta}\right). \nonumber
\label{eq: lem_fixed_pt_bound_temp}
\end{align}
\end{proof}

Finally, we state the moment bound needed to control the interpolation error.
\begin{lemma}
\label{lem:expectaion_power2k}

Assume the step size condition in Lemma~\ref{lem: First_Picard_increment_bound} and ~\ref{lem:discrete_contraction}. Let
\[
Z \;:=\; \sup_{t\in[0,1]}\bigl\|\widehat X^\ast_y(nh+th)\bigr\|,
\]
for $y = \widehat X(nh).$ Then for every integer $k\ge1$,
\[
\mathbb{E}\bigl[Z^{2k}\bigr]
\;\lesssim \; Nd^k.
\]
\end{lemma}
\begin{proof}
For $k\ge1$,
\[
\mathbb{E}\bigl[Z^{2k}\bigr]
= \int_{0}^{\infty}\mathbb{P}\bigl(Z^{2k}>t\bigr)\,dt
= \int_{0}^{\infty}\mathbb{P}\bigl(Z>u\bigr)\,2k\,u^{2k-1}\,du .
\]
By the tail estimate of Lemma~\ref{lem:high-prob-bound_fixed_point}, we obtain
\[
\mathbb{E}\bigl[Z^{2k}\bigr]
\le 6kN \int_{0}^{\infty} u^{2k-1}
    \exp\Bigl(-\frac{u^{2}}{C d}\Bigr)\,du .
\]
for some positive constant $C > 0$ independent of $d, N, \varepsilon$. Let $a:=1/(C d)$. The standard integral
\(
\int_{0}^{\infty} u^{2k-1} e^{-a u^{2}}\,du
= \tfrac{1}{2} a^{-k}\Gamma(k)
\)
yields
\[
\mathbb{E}\bigl[Z^{2k}\bigr]
\lesssim Nd^k
\]
as claimed.
\end{proof}

\subsection{Derivatives of $\widehat  X^{\ast}_1$}
\label{Appendix-E3}

To control the higher-order derivatives of \(\widehat X_1^\ast\), we first record a standard bound on the derivatives of the Lagrange interpolant appearing in the discretized dynamics.  Let $f=f(t)$ be the function of interest, and let $P_t$ denote its Lagrange interpolating polynomial, which serves as an approximation of $f$.  
We consider $k{+}1$ interpolation nodes
\[
t_0,\, t_1,\,\ldots,\, t_k, \qquad t_j = t_0 + jh, \quad j=0,\ldots,k,
\]
with constant step size $h>0$.

In the Newton form \citep{stoer1980introduction}, the interpolating polynomial $P_t$ can be expressed as  
\begin{equation} \label{eq:Lagrange-in-Newton-form}
    P_t \;=\; a_0 \;+\; \sum_{i=1}^{K} a_i \prod_{j=0}^{i-1} (t - t_j),
\end{equation}
where the coefficients $a_i = [f(t_0), f(t_1), \ldots, f(t_i)]$ are the
finite divided differences of $f$. These coefficients satisfy the recursive identity
\begin{equation} \label{eq: Finite divided difference recursive relation}
    [f(t_0), f(t_1), \ldots, f(t_i)]
    \;=\; \frac{[f(t_1), f(t_2), \ldots, f(t_i)] \;-\; [f(t_0), f(t_1), \ldots, f(t_{i-1})]}{t_i - t_0}.
\end{equation}

Having expressed the interpolant in the Newton form, we now establish a few useful lemmas that will be instrumental in the proof.  
In particular, we next derive a more general recursive relation for finite divided differences.

\begin{lemma}[Order-reduction of divided differences] \label{lem: K finite divided in i finite divided}
    Let $t_r = t_0 + r h$ for $r=0,1,\ldots,k$ be equally spaced nodes with step size $h>0$ and let $D_{r}^{(k)} =  f[t_r, t_{r+1}, \ldots, t_{r+k}]$, be the $k$-th order divided difference starting at point $f(t_r)$. For $i \leq k$, let $g_{r} = D_{r}^{i} = f[x_r, \ldots, x_{r+i}]$, be $i$-th order finite divided difference starting at point $f(x_r).$ Then for $m = k-i$, 
    \begin{equation}
        D_{0}^{(k)} = \frac{i!}{k!h^m}\sum_{j=0}^{m}(-1)^{m+j} \binom{m}{j} g_{j}.
    \end{equation}
    This lemma expresses $k$-th order finite divided difference in terms of $i$-th order finite divided differences.
\end{lemma}
\begin{proof}
    We will prove this by induction over $m$. For the base case $m = 1$, i.e. $i = k-1$, we have
    \begin{align}
       D_{0}^{(k)}=[f(t_0), f(t_1),\ldots, f(t_k)] = \frac{1}{kh} \left( -f[t_0, \ldots, t_{k-1}] + f[t_1, \ldots, t_k] \right).
    \end{align}
    This is true by the recursive relation \eqref{eq: Finite divided difference recursive relation}.
Assume the formula holds for \(m=l\), i.e. \(i=k-l\). Using the recursive identity
\[
D_j^{(k-l)}
=
\frac{D_{j+1}^{(k-l-1)}-D_j^{(k-l-1)}}{(k-l)h},
\]
we get
\[
\begin{aligned}
D_0^{(k)}
&=
\frac{(k-l)!}{k!h^l}
\sum_{j=0}^{l}(-1)^{l+j}\binom{l}{j}D_j^{(k-l)}  \\
&=
\frac{(k-l-1)!}{k!h^{l+1}}
\sum_{j=0}^{l}(-1)^{l+j}\binom{l}{j}
\left(D_{j+1}^{(k-l-1)}-D_j^{(k-l-1)}\right).
\end{aligned}
\]
Reindexing the first sum and using Pascal's identity,
\(\binom{l}{j-1}+\binom{l}{j}=\binom{l+1}{j}\), this becomes
\[
D_0^{(k)}
=
\frac{(k-l-1)!}{k!h^{l+1}}
\sum_{j=0}^{l+1}
(-1)^{l+1+j}\binom{l+1}{j}D_j^{(k-l-1)}.
\]
Thus the formula holds for \(m=l+1\), completing the induction.
\end{proof}

\begin{lemma}[Derivative bound for the Lagrange interpolant]
\label{lem:lagrange-derivative-bound}
Let \(P\) be the Lagrange interpolant of \(f\) on \(k+1\) equally spaced nodes
\(t_0,\ldots,t_k\) in an interval \(I=[t_0,t_k]\). Then, for every \(0\le n\le k\),
\[
\sup_{t\in I}\|P^{(n)}(t)\|
\le
C_{p,n}\sup_{t\in I}\|f^{(n)}(t)\|,
\]
where \(C_{p,n}\) depends only on \(k\) and \(n\).
\end{lemma}

\begin{proof}
Write the Newton form of the interpolating polynomial as
\[
P(t)=\sum_{\ell=0}^k a_\ell\prod_{j=0}^{\ell-1}(t-t_j),
\qquad
a_\ell=f[t_0,\ldots,t_\ell].
\]
Differentiating \(n\) times gives
\[
P^{(n)}(t)
=
\sum_{\ell=n}^k
a_\ell
\frac{d^n}{dt^n}\prod_{j=0}^{\ell-1}(t-t_j).
\]
For \(\ell=n+i\), the coefficient \(a_{n+i}\) is an \((n+i)\)-th order divided difference. By Lemma~\ref{lem: K finite divided in i finite divided}, it can be expressed in terms of \(n\)-th order divided differences as
\[
a_{n+i}
=
\frac{n!}{(n+i)!\,h^i}
\sum_{j=0}^{i}(-1)^{i+j}\binom{i}{j}D_j^{(n)} ,
\]
where \(D_j^{(n)}\) denotes an \(n\)-th order divided difference. Hence
\[
\|a_{n+i}\|
\le
\frac{n!}{(n+i)!\,h^i}
\sum_{j=0}^{i}\binom{i}{j}\|D_j^{(n)}\|.
\]
By the mean-value theorem for divided differences,
\[
\|D_j^{(n)}\|
\le
\frac{1}{n!}\sup_{t\in I}\|f^{(n)}(t)\|.
\]
Therefore,
\[
\|a_{n+i}\|
\le
\frac{2^i}{(n+i)!\,h^i}
\sup_{t\in I}\|f^{(n)}(t)\|.
\]

On the other hand, it is easy to verify that, for \(t\in I\),
\[
\left|
\frac{d^n}{dt^n}\prod_{j=0}^{n+i-1}(t-t_j)
\right|
\le
C_{n,i}h^i,
\]
where \(C_{n,i}\) depends only on \(n\) and \(i\). Combining the last two bounds,
\[
\|P^{(n)}(t)\|
\le
\sum_{i=0}^{k-n}
\frac{2^i C_{n,i}}{(n+i)!}
\sup_{t\in I}\|f^{(n)}(t)\|.
\]
Since the sum has finitely many terms depending only on \(k\) and \(n\), we obtain
\[
\sup_{t\in I}\|P^{(n)}(t)\|
\le
C_{p,n}\sup_{t\in I}\|f^{(n)}(t)\|.
\]
\end{proof}

We next establish a bound on the derivatives of the fixed-point trajectory 
$\widehat{X}^{\ast}_1$ within a single interpolation interval. We apply the above lemmas to the Lagrange interpolant \(P_2(t;\widehat X^\ast)\) in Proposition~\ref{prop:error-bound}.

\begin{lemma}\label{lem:X1-deriv-bound}
Let \(I\subset\mathbb R\) be an interval on which the Lagrange interpolant
\(P_t:=P_2(t;\widehat X^\ast)\) is constructed. Fix \(1\le k\le K-1\).
Then there exists a constant \(C_k, \widetilde C_k >0 \), depending only on \(k,\gamma\), and the interpolation nodes, such that,
\[
\left\|\frac{d^k}{dt^k}\widehat X_1^\ast(t)\right\|
\le
\widetilde C_k\left(
\|\widehat X^\ast(t)\|
+
\sum_{q=0}^{k-2}
\sup_{s\in I}
\left\|
\frac{d^q}{ds^q}\nabla U(\widehat X_1^\ast(s))
\right\|
\right).
\]

\end{lemma}

\begin{proof}
The fixed-point equation for \(\widehat X^\ast\), corresponding to the
higher-order dynamics \eqref{eq: high_order_langevin} with \(\nabla U\) replaced by the interpolant \(P_t\), yields a linear chain of equations (See \eqref{eq: operator_hat_T2}). Differentiating the first
coordinate repeatedly, the first few derivatives have the form
\[
\left\|\frac{d}{dt}\widehat X_1^\ast(t)\right\|
=
\|\widehat X_2^\ast(t)\|,
\]
\[
\left\|\frac{d^2}{dt^2}\widehat X_1^\ast(t)\right\|
\le
C_\gamma\left(\|P_t\|+\|\widehat X_3^\ast(t)\|\right),
\]
and
\[
\left\|\frac{d^3}{dt^3}\widehat X_1^\ast(t)\right\|
\le
C_\gamma\left(
\left\|\frac{d}{dt}P_t\right\|
+\|\widehat X_2^\ast(t)\|
+\|\widehat X_4^\ast(t)\|
\right).
\]
Continuing in the same way, for each fixed \(k\le K-1\),
\[
\frac{d^k}{dt^k}\widehat X_1^\ast(t)
=
\sum_{j=1}^{k+1} c_{k,j}\widehat X_j^\ast(t)
+
\sum_{i=0}^{k-2} d_{k,i}\frac{d^i}{dt^i}P_t,
\]
where the coefficients \(c_{k,j},d_{k,i}\) depend only on \(k\) and \(\gamma\).
Taking norms and absorbing the finitely many coefficients into a constant
\(C_k\), we obtain
\[
\left\|\frac{d^k}{dt^k}\widehat X_1^\ast(t)\right\|
\le
C_k\left(
\|\widehat X^\ast(t)\|
+
\sum_{i=0}^{k-2}
\left\|\frac{d^i}{dt^i}P_t\right\|
\right).
\]
Finally, applying Lemma~\ref{lem:lagrange-derivative-bound} with
\(f(t)=\nabla U(\widehat X_1^\ast(t))\) gives
\[
\left\|\frac{d^k}{dt^k}\widehat X_1^\ast(t)\right\|
\le
C_k\|\widehat X^\ast(t)\|
+
C_k
\sum_{i=0}^{k-2} C_{p,i}
\sup_{s\in I}
\left\|
\frac{d^i}{ds^i}\nabla U(\widehat X_1^\ast(s))
\right\|.
\]
We may absorb the constants \(C_k\), \(C_{p,i}\), and the finite multiplicities into a new constant \(\widetilde C_k\), obtaining
\[
\left\|\frac{d^k}{dt^k}\widehat X_1^\ast(t)\right\|
\le
\widetilde C_k
\left(
\|\widehat X^\ast(t)\|
+
\sum_{q=0}^{k-2}
\sup_{s\in I}
\left\|
\frac{d^q}{ds^q}\nabla U(\widehat X_1^\ast(s))
\right\|
\right).
\]
This proves the claim.
\end{proof}
\begin{remark}
The derivative bounds above are stated only for \(k\le K-1\). This is because the Brownian part enters through the last coordinate of the chain, so differentiating the first coordinate more than \(K-1\) times would require differentiating the Brownian term. Hence the interpolation analysis only uses derivatives up to order \(K-1\).
\end{remark}

\section{Ridge-separable case}
\label{Appendix-H}
In this section, we revisit the interpolation error under the ridge-separable structure in Assumption~\ref{assump:ridge-separable}. We focus only on the interpolation component of the analysis, since this is the step where the crude worst-case bounds based on full higher-order smoothness of \(U\) introduce unfavorable dimension dependence. The continuous-time contraction and Picard iteration arguments remain unchanged, so the improvement comes entirely from sharpening the bound on the interpolation residual.

To reiterate, in the ridge-separable setting we consider potentials of the form
\[
U(x)
=
U_{\mathrm{ridge}}(x)
+
\frac{\eta}{2}\|x\|^2,
\qquad
U_{\mathrm{ridge}}(x)
=
\sum_{j=1}^r \phi_j(\theta_j^\top x),
\]
where \(\theta_j\in S^{d-1}\). We write
\[
\Theta
:=
\begin{bmatrix}
\theta_1^\top\\
\vdots\\
\theta_r^\top
\end{bmatrix}
\in\mathbb R^{r\times d},
\]
and assume that
\[
\lambda_{\max}(\Theta\Theta^\top)\le \Lambda_\Theta .
\]
In Assumption~\ref{assump:ridge-separable}, we also define
\[
L_{\max}
:= \max\left\{1, \max_{1 \leq j \leq K-1} L_j\right\}.
\]
The following lemma records the projected tensor bound used in the
ridge-separable interpolation analysis.

\begin{lemma}[Projected Tensor norm]
\label{lem:ridge-projected-l2-tensor-bound}
For an \(j\)-linear tensor \(T:(\mathbb R^d)^j\to \mathbb R^d\), define the projected-\(\ell_2\) tensor norm by
\[
\|T\|_{2,\Theta}
:=
\sup_{\|\Theta v_1\|_2\le 1,\ldots,\|\Theta v_j\|_2\le 1}
\|T[v_1,\ldots,v_j]\|_2 .
\]
Then, for every \(2\le m\le K\) and every \(x\in\mathbb R^d\),
\[
\left\|\nabla^{(m)}U_{\mathrm{ridge}}(x)\right\|_{2,\Theta}
\le
\sqrt{\Lambda_\Theta}\,L_{\max}.
\]
Equivalently, for every \(\Delta_1,\ldots,\Delta_{m-1}\in\mathbb R^d\),
\[
\left\|
\nabla^{(m)}U_{\mathrm{ridge}}(x)
[\Delta_1,\ldots,\Delta_{m-1}]
\right\|_2
\le
\sqrt{\Lambda_\Theta}\,L_{\max}
\prod_{\ell=1}^{m-1}\|\Theta\Delta_\ell\|_2 .
\]
\end{lemma}

\begin{proof}
Repeated differentiation gives
\[
\nabla^{(m)}U_{\mathrm{ridge}}(x)
=
\sum_{j=1}^r
\phi_j^{(m)}(\theta_j^\top x)\theta_j^{\otimes m}.
\]
Hence, for any \(\Delta_1,\ldots,\Delta_{m-1}\in\mathbb R^d\),
\[
\nabla^{(m)}U_{\mathrm{ridge}}(x)
[\Delta_1,\ldots,\Delta_{m-1}]
=
\Theta^\top a,
\]
where
\[
a_j
=
\phi_j^{(m)}(\theta_j^\top x)
\prod_{\ell=1}^{m-1}(\theta_j^\top\Delta_\ell).
\]
Using \(\lambda_{\max}(\Theta\Theta^\top)\le \Lambda_\Theta\), we have
\[
\|\Theta^\top a\|_2
\le
\sqrt{\Lambda_\Theta}\,\|a\|_2 .
\]
Moreover,
\[
\|a\|_2
\le
L_{\max}
\left(
\sum_{j=1}^r
\prod_{\ell=1}^{m-1}
|\theta_j^\top\Delta_\ell|^2
\right)^{1/2}
\le
L_{\max}
\prod_{\ell=1}^{m-1}
\left(
\sum_{j=1}^r|\theta_j^\top\Delta_\ell|^2
\right)^{1/2}.
\]
Since
\[
\left(
\sum_{j=1}^r|\theta_j^\top\Delta_\ell|^2
\right)^{1/2}
=
\|\Theta\Delta_\ell\|_2,
\]
we obtain
\[
\left\|
\nabla^{(m)}U_{\mathrm{ridge}}(x)
[\Delta_1,\ldots,\Delta_{m-1}]
\right\|_2
\le
\sqrt{\Lambda_\Theta}\,L_{\max}
\prod_{\ell=1}^{m-1}\|\Theta\Delta_\ell\|_2.
\]
Taking the supremum over \(\|\Theta\Delta_\ell\|_2\le1\) proves the claim.
\end{proof}

\begin{lemma}[Ridge-separable analogue of Lemma~\ref{lem:X1-deriv-bound}]
\label{lem:Theta-X1-deriv-bound}
Let \(I\subset\mathbb R\) be an interpolation interval. Then, for every \(2\le k\le K-1\), there exist constants
\(\tilde C_k,\tilde C_{k,q}>0\), depending only on \(k,\gamma,\eta\), such that
\[
\sup_{t\in I}
\left\|
\frac{d^k}{dt^k}\Theta \widehat X_1^\ast(t)
\right\|_2
\le
\tilde C_k\sup_{t\in I}\|\Theta \widehat X^\ast(t)\|_2
+
\sqrt{\Lambda_\Theta} \left(\sum_{q=0}^{k-2} \tilde C_{k,q}
\sup_{t\in I}
\left\|
\frac{d^q}{dt^q}\nabla U_{\mathrm{ridge}}(\widehat X_1^\ast(t))
\right\|_2 \right).
\]
\end{lemma}

\begin{proof}
The proof is the projected analogue of Lemma~\ref{lem:X1-deriv-bound}. Applying the fixed linear map \(\Theta\) to the chain equations for the fixed-point trajectory and differentiating repeatedly gives, for every \(2\le k\le K-1\),
\begin{equation}
\left\|
\frac{d^k}{dt^k}\Theta \widehat X_1^\ast(t)
\right\|_2
\le
C_k\left(
\|\Theta \widehat X^\ast(t)\|_2
+
\sum_{i=0}^{k-2}
\left\|
\Theta \frac{d^i}{dt^i}P_t
\right\|_2
\right),
\label{eq:theta^TX_bound}
\end{equation}

where \(P_t\) is the Lagrange interpolant of
\(t\mapsto \nabla U(\widehat X_1^\ast(t))\).

Since
\[
\nabla U(x)=\nabla U_{\mathrm{ridge}}(x)+\eta x,
\]
linearity of the interpolation operator gives
\[
P_t=P_t^{\mathrm{ridge}}+\eta P_t^X,
\]
where \(P_t^{\mathrm{ridge}}\) interpolates \(\nabla U_{\mathrm{ridge}}\), and \(P_t^X\)
interpolates \(\widehat X_1^\ast\). For the ridge part, the Lagrange derivative
bound (Lemma~\ref{lem:lagrange-derivative-bound}) and \(\|\Theta v\|_2\le \sqrt{\Lambda_\Theta}\|v\|_2\) imply
\[
\sup_{t\in I}
\left\|
\Theta \frac{d^i}{dt^i}P_t^{\mathrm{ridge}}
\right\|_2
\le
C_{p,i} \sqrt{\Lambda_\Theta}
\sum_{q=i}^{k-2}
\sup_{t\in I}
\left\|
\frac{d^q}{dt^q}g_{\mathrm{ridge}}(t)
\right\|_2.
\]

For the linear part, the same Lagrange derivative bound gives lower-order terms
in
\[
\sup_{t\in I}
\left\|
\frac{d^q}{dt^q}\Theta \widehat X_1^\ast(t)
\right\|_2,
\qquad q\le k-2.
\]
These are absorbed inductively using the already established bounds for lower
derivatives in \eqref{eq:theta^TX_bound}. Renaming constants yields the claim.
\end{proof}

\begin{lemma}[Ridge-separable analogue of Lemma~\ref{lem: inductive bound of time derivative of nabla U}]
\label{lem:ridge-derivative-bound}
Let
\[
g(t):=\nabla U(X_1^\ast(t)),
\qquad
g_{\mathrm{ridge}}(t):=\nabla U_{\mathrm{ridge}}(X_1^\ast(t)),
\]
and define
\[
M_I:=\sup_{t\in I}\|\Theta X_1^\ast(t)\|_2,
\qquad
N_I:=\sup_{t\in I}\|X^\ast(t)\|.
\]
Write
\[
U(x)=U_{\mathrm{ridge}}(x)+\frac{\eta}{2}\|x\|^2.
\]
Under Assumption~\ref{assump:ridge-separable}, for every \(n\ge1\), there exist constants
\(C_n,C_n'>0\), independent of \(d,h,N\), such that
\[
\sup_{t\in I}
\left\|
\frac{d^n}{dt^n}g(t)
\right\|
\le
C_n L_{\max}\sqrt{\Lambda_\Theta}\,M_I
\bigl(L_{\max}\Lambda_\Theta+M_I\bigr)^{n-1}
+
C_n' N_I .
\]
In particular, for \(n=K-1\),
\[
\sup_{t\in I}
\left\|
\frac{d^{K-1}}{dt^{K-1}}\nabla U(X_1^\ast(t))
\right\|
\le
C_{K-1} L_{\max}\sqrt{\Lambda_\Theta}\,M_I
\bigl(L_{\max}\Lambda_\Theta+M_I\bigr)^{K-2}
+
C_{K-1}'N_I .
\]
\end{lemma}

\begin{proof}

The proof has two parts. First, we control the ridge component
\(g_{\mathrm{ridge}}(t)\). Second, we add the contribution of the quadratic
term.

Set
\[
a:=L_{\max}\sqrt{\Lambda_\Theta},
\qquad
H:=a \sqrt{\Lambda_\Theta}+M_I .
\]
By Lemma~\ref{lem:ridge-projected-l2-tensor-bound}, every multilinear derivative of
\(U_{\mathrm{ridge}}\) satisfies the projected-\(\ell_2\) bound
\[
\left\|
\nabla^{(m)}U_{\mathrm{ridge}}(x)
[\Delta_1,\ldots,\Delta_{m-1}]
\right\|
\le
a\prod_{\ell=1}^{m-1}\|\Theta\Delta_\ell\|_2 .
\]
Combining this estimate with Faà di Bruno's formula gives the same induction as
in Lemma~\ref{lem: inductive bound of time derivative of nabla U}, except that
each derivative of \(X_1^\ast\) is measured after applying \(\Theta\). Using
Lemma~\ref{lem:Theta-X1-deriv-bound} to control these projected derivatives, we
obtain, for every \(n\ge1\),
\[
\sup_{t\in I}
\left\|
\frac{d^n}{dt^n}g_{\mathrm{ridge}}(t)
\right\|
\le
C_n a M_I H^{n-1}.
\]
Here the constant \(C_n\) are independent of \(d, r, \nu, h,N, M\).

Now use
\[
g(t)=g_{\mathrm{ridge}}(t)+\eta X_1^\ast(t).
\]
Therefore,
\[
\sup_{t\in I}
\left\|
\frac{d^n}{dt^n}g(t)
\right\|
\le
\sup_{t\in I}
\left\|
\frac{d^n}{dt^n}g_{\mathrm{ridge}}(t)
\right\|
+
\eta
\sup_{t\in I}
\left\|
\frac{d^n}{dt^n}X_1^\ast(t)
\right\|.
\]
The Euclidean derivative bound from Lemma~\ref{lem:X1-deriv-bound} gives
\[
\sup_{t\in I}
\left\|
\frac{d^n}{dt^n}X_1^\ast(t)
\right\|
\le
C\left(
N_I+
\sum_{i=0}^{n-2}
\sup_{s\in I}
\left\|
\frac{d^i}{ds^i}g(s)
\right\|
\right).
\]
Thus
\[
\sup_{t\in I}
\left\|
\frac{d^n}{dt^n}g(t)
\right\|
\le
C_n aM_IH^{n-1}
+
C\eta N_I
+
C\eta
\sum_{i=0}^{n-2}
\sup_{s\in I}
\left\|
\frac{d^i}{ds^i}g(s)
\right\|.
\]
This recursion closes by induction on \(n\), since all terms in the sum are of
strictly lower order. Hence, after renaming constants,
\[
\sup_{t\in I}
\left\|
\frac{d^n}{dt^n}g(t)
\right\|
\le
C_n aM_IH^{n-1}+C_n'N_I.
\]
Substituting \(a=L_{\max}\sqrt{\Lambda_\Theta}\) and \(H=a\sqrt{\Lambda_\Theta}+M_I\) proves the claim.
\end{proof}

\begin{lemma}[Ridge separable analogue of Lemma~\ref{lem_appendix:lagrange-IR-uniform}]
\label{lem_appendix:lagrange-IR-uniform-ridge}
Let \(h=\Theta(1)\) be the step size  such that the moment estimate in
Lemma~\ref{lem:expectaion_power2k_ridge} applies. On each step \([nh,(n+1)h]\), partition the interval into
\(M\) blocks and let \(P_2(t;\widehat X^\ast)\) be the blockwise
Lagrange interpolant of
\[
t\mapsto \nabla U(\widehat X_1^\ast(t))
\]
using \(K-1\) equispaced nodes on each block. Define
\[
I_U
:=
\sup_{t\in[nh,(n+1)h]}
\left\|
\nabla U(\widehat X_1^\ast(t))-P_2(t;\widehat X^\ast)
\right\|^2 .
\]
Under Assumptions~\ref{assump:U-strong-smooth}-~\ref{assump:centered},and \ref{assump:ridge-separable}, there exists a constant
\(C_{\rm IR}>0\), independent of \(h, r,d,N,M,\nu,\varepsilon\), such that,
\[
\mathbb E I_U
\le
C_{\rm IR}
\frac{N h^{2K-2}}{M^{2K-2}(K-1)!^2}
\left[
\sum_{j=2}^{2K-2}
\binom{2K-4}{j-2}
L_{\max}^{2K-j}
\Lambda_\Theta^{2K-j-1}
r^{j/2}
+d
\right].
\]
\end{lemma}

\begin{proof}
The standard Lagrange interpolation error estimate, applied on each block, gives
\[
\sqrt{I_U}
\le
\frac{(h/M)^{K-1}}{(K-1)!}
\sup_{t\in[nh,(n+1)h]}
\left\|
\frac{d^{K-1}}{dt^{K-1}}
\nabla U(\widehat X_1^\ast(t))
\right\|.
\]
Let
\[
M_I:=\sup_{t\in[nh,(n+1)h]}\|\Theta \widehat X_1^\ast(t)\|_2,
\qquad
N_I:=\sup_{t\in[nh,(n+1)h]}\|\widehat X^\ast(t)\|.
\]
By Lemma~\ref{lem:ridge-derivative-bound},
\[
\sup_{t\in[nh,(n+1)h]}
\left\|
\frac{d^{K-1}}{dt^{K-1}}
\nabla U(\widehat X_1^\ast(t))
\right\|
\le
C\,
L_{\max}\sqrt{\Lambda_\Theta}\,M_I
\bigl(L_{\max}\Lambda_\Theta+M_I\bigr)^{K-2}
+
C'N_I .
\]
Substituting this into the interpolation estimate and using
\((a+b)^2\le 2a^2+2b^2\), we obtain
\[
I_U
\le
C
\frac{h^{2K-2}}{M^{2K-2}(K-1)!^2}
\left[
L_{\max}^2\Lambda_\Theta
M_I^2
\bigl(L_{\max}\Lambda_\Theta+M_I\bigr)^{2K-4}
+
N_I^2
\right].
\]
Expanding the first term by the binomial theorem yields
\[
I_U
\le
C
\frac{h^{2K-2}}{M^{2K-2}(K-1)!^2}
\left[
\sum_{j=2}^{2K-2}
\binom{2K-4}{j-2}
L_{\max}^{2K-j}
\Lambda_\Theta^{2K-j-1}
M_I^j
+
N_I^2
\right].
\]
Finally, using the moment bounds in Lemma~\ref{lem:expectaion_power2k} and Lemma~\ref{lem:expectaion_power2k_ridge}, we get
\[
\mathbb E[M_I^j]\lesssim N r^{j/2},
\qquad
\mathbb E[N_I^2]\lesssim Nd,
\qquad 2\le j\le 2K-2.
\]
Factoring out the common \(N\), we get
\[
\mathbb E I_U
\le
C_{\rm IR}
\frac{N h^{2K-2}}{M^{2K-2}(K-1)!^2}
\left[
\sum_{j=2}^{2K-2}
\binom{2K-4}{j-2}
L_{\max}^{2K-j}
\Lambda_\Theta^{2K-j-1}
r^{j/2}
+d
\right].
\]
This proves the claim.
\end{proof}

\subsection{Ridge-separable Picard contraction}
 Before stating the projected contraction result, we record a simple
commutation property. Let
\[
\bar\Theta:=I_K\otimes \Theta\in\mathbb R^{Kr\times Kd}.
\]
Since the matrices \(A\), \(D\), and \(Q\) have Kronecker-product form with
respect to the \(K\)-block structure, we can write
\[
A=A_K\otimes I_d,\qquad
D+Q=(D_K+Q_K)\otimes I_d,
\]
for suitable \(K\times K\) matrices \(A_K,D_K,Q_K\). Define their projected
counterparts by
\[
A_r:=A_K\otimes I_r,
\qquad
D_r+Q_r:=(D_K+Q_K)\otimes I_r .
\]
Then, using the mixed-product rule for Kronecker products,
\[
\bar\Theta A
=
(I_K\otimes \Theta)(A_K\otimes I_d)
=
A_K\otimes \Theta
=
(A_K\otimes I_r)(I_K\otimes \Theta)
=
A_r\bar\Theta,
\]
and similarly,
\[
\bar\Theta(D+Q)
=
(D_r+Q_r)\bar\Theta .
\]

\begin{lemma}[Ridge-separable analogue of Lemma~\ref{lem: T_phi_contraction_sup_norm}]
For a fixed initialization
\(y\in\mathbb R^{Kd}\), define the projected discretized map
\[
\widehat{\mathcal T}_y^\Theta[X]
:=
\bar\Theta\,\widehat{\mathcal T}_y[X].
\]
Assume that
\[
h\le
\min\left\{
\frac{\log 2}{3\gamma},
\frac{1}{20\Gamma_\phi(\Lambda_\Theta L_{max}+\eta)}
\right\}.
\]
Then, for all paths \(X,Y\in\mathcal C([0,h],\mathbb R^{Kd})\),
\[
\big\|
\widehat{\mathcal T}_y^\Theta[X]
-
\widehat{\mathcal T}_y^\Theta[Y]
\big\|_\infty
\le
\frac{1}{10}
\big\|
\bar\Theta X-\bar\Theta Y
\big\|_\infty .
\]
\end{lemma}
\begin{proof}
Let \(\tilde X=\widehat{\mathcal T}_y[X]\) and \(\tilde Y=\widehat{\mathcal T}_y[Y]\). Then 
\begin{align}
\widehat{\mathcal{T}}^{\Theta}_y[X]
= 
\bar \Theta\tilde{X}(t)
&=
\bar \Theta y + \int_{0}^{t} A_r \bar \Theta\,\tilde{X}(s)\,ds \nonumber\\
&\;-\; \int_{0}^{t} (e_2\otimes I_r) \bar \Theta\,P(s;X)\,ds \;+\; \int_{0}^{t} \bar \Theta\sqrt{2D}\,dB_s. \label{eq21}
\end{align}
Since \(\|e_2\otimes I_r\|=1\),
\begin{align*}
\|\bar \Theta\tilde X(t)-\bar \Theta\tilde Y(t)\|
&\le \int_{0}^{t} \|A_r\|\,\|\bar \Theta\tilde X(s)-\bar \Theta\tilde Y(s)\|\,ds
 \;+\; \int_{0}^{t} \|\Theta P(s;X)-\Theta P(s;Y)\|\,ds.
\end{align*}
We now bound the interpolation term using the block structure. Partition \([0,h]\) into \(M\) blocks (namely $B_i$)  of size \(h/M\), and let \(P_{B_i}\) denote the Lagrange interpolant on block \(B_i\). Then
\[
\int_0^t \|\Theta P(s;X)- \Theta P(s;Y)\|\,ds
\;\le\; \sum_{i=1}^{M} \int_{B_i \cap [0,t]} \|\Theta P_{B_i}(s;X)-\Theta P_{B_i}(s;Y)\|\,ds.
\]

For each \(i\) we have,
\begin{align*}
\|\Theta P_{B_i}(s;X)-\Theta P_{B_i}(s;Y)\|
&\le  \sum_{j=1}^{K-1} |\ell_{i,j}(s)|\,\|\Theta \Theta ^\top \phi'(\Theta X(t_{i,j})) \\
& \hspace{2cm}- \Theta \Theta ^\top \phi'(\Theta Y(t_{i,j})) + \eta\Theta X(t_{i,j}) - \eta\Theta Y(t_{i,j})\|_\infty \\
&= \sum_{j=1}^{K-1} |\tilde\ell_{j}(s)|\,\left(\Lambda_{\Theta} L_{max}\|\Theta X - \Theta Y\|_\infty + \eta\|\Theta X - \Theta Y\|_\infty \right) \\
&\le \,\Gamma_\phi (\Lambda_{\Theta} L_{max} + \eta)\,\|\Theta X- \Theta Y\|_\infty.
\end{align*}
Here, in the second line we used that \(\|\Theta\Theta^\top\|_{op}\le \Lambda_\Theta\) and that \(\phi'\) is \(L_{max}\)-Lipschitz, while in the last line we used the definition of the Lebesgue constant \(\Gamma_\phi\).

Summing over at most \(M\) blocks yields
\[
\int_0^t \|\Theta P(s;X)- \Theta P(s;Y)\|\,ds
\;\le\; h\,\Gamma_\phi (\Lambda_{\Theta} L_{max} + \eta)\,\|\Theta X- \Theta Y\|_\infty.
\]

Therefore,
\begin{align*}
\|\bar \Theta\tilde X(t)-\bar \Theta\tilde Y(t)\|
&\le \int_{0}^{t} \|A_r\|\,\|\bar \Theta\tilde X(s)-\bar \Theta\tilde Y(s)\|\,ds
 \;+\; h\,\Gamma_\phi (\Lambda_{\Theta} L_{max} + \eta)\,\|\Theta X- \Theta Y\|_\infty.
\end{align*}

Applying Grönwall’s inequality, for \(t\le h\),
\begin{align*}
  \|\Theta \tilde X(t)- \Theta \tilde Y(t)\|
\;&\le\   e^{\|A_r\| t}\, h\,\Gamma_\phi (\Lambda_{\Theta} L_{max} + \eta)\,\|\Theta X- \Theta Y\|_\infty \\
& \;\le\; e^{3\gamma h}\, h\,\Gamma_\phi (\Lambda_{\Theta} L_{max} + \eta)\,\|\Theta X- \Theta Y\|_\infty.
\end{align*}

Choose \(h\le \frac{\log 2}{3\gamma}\) so that \(e^{3\gamma h}\le 2\), and \(h\le (20 \Gamma_\phi (\Lambda_{\Theta} L_{max} + \eta))\) so that the product is at most \(1/10\). Taking the supremum over \(t\in[0,h]\) yields the result.
\end{proof}

\begin{remark}[Geometric convergence in the projected norm]
\label{rem:projected-geometric-picard}
Let
\[
\widehat Z^{[\nu]} := \bar\Theta\,\widehat X^{[\nu]}.
\]
Then the projected Picard iterates converge geometrically in the uniform norm. In particular, for all \(\nu\ge 1\),
\[
\|\widehat Z^{[\nu]}-\widehat Z^{[\nu-1]}\|_\infty
\le \rho_\Theta^{\nu-1}\,\|\widehat Z^{[1]}-\widehat Z^{[0]}\|_\infty,
\qquad \rho_\Theta:= 2\, h\,\Gamma_\phi (\Lambda_{\Theta} L_{max} + \eta)\leq \frac{1}{10}.
\]
Summing the telescoping series, we obtain
\[
\|\widehat Z^{[\nu]}-\widehat Z^{[0]}\|_\infty
\le \sum_{k=0}^{\nu-1}\|\widehat Z^{[k+1]}-\widehat Z^{[k]}\|_\infty
\le \sum_{k=0}^{\nu-1}\rho_\Theta^k\,\|\widehat Z^{[1]}-\widehat Z^{[0]}\|_\infty
=
\frac{1-\rho_\Theta^\nu}{1-\rho_\Theta}\,\|\widehat Z^{[1]}-\widehat Z^{[0]}\|_\infty.
\]
In particular, letting \(\nu\to\infty\) yields
\[
\|\widehat Z_y^\ast-\widehat Z^{[0]}\|_\infty
\le \frac{1}{1-\rho_\Theta}\,\|\widehat Z^{[1]}-\widehat Z^{[0]}\|_\infty,
\qquad
\widehat Z_y^\ast:=\bar\Theta\,\widehat X_y^\ast.
\]
\end{remark}

\subsection{Ridge-separable Lyapunov matrix}
Before proceeding, we introduce a projected norm that allows us to establish contraction directly in the reduced ridge space. For this subsection, consider an alternative quadratic potential
\[
\widehat U(x)=\frac{\eta}{2}\|x\|^2
\]
for \(x\in\mathbb R^r\).

We associate this potential with the following reduced \(K\)-th order Langevin
dynamics on \(\mathbb R^{Kr}\):
\[
Z(t)=(X(t),Y(t)),
\qquad
X(t)\in\mathbb R^r,\quad
Y(t)\in\mathbb R^{(K-1)r},
\]
where
\begin{align*}
dX(t) &= \bar P\,Y(t)\,dt,\\
dY(t) &= -\bar P^\top \nabla \widehat U(X(t))\,dt
        - \gamma\,\bar Q\,Y(t)\,dt
        + \sqrt{2\gamma}\,\bar D\,dB_t .
\end{align*}
Here
\[
\bar P=(I_r,0,\ldots,0)\in\mathbb R^{r\times (K-1)r},
\]
and \(\bar Q,\bar D\in\mathbb R^{(K-1)r\times (K-1)r}\) are the \(r\)-dimensional
block analogues of the matrices appearing in the original higher-order
dynamics in Appendix~\ref{Appendix-C1}:
\[
\bar Q=
\begin{pmatrix}
0 & -I_r & 0 & \cdots & 0 \\
I_r & 0 & -I_r & \ddots & \vdots \\
0 & I_r & 0 & \ddots & 0 \\
\vdots & \ddots & \ddots & \ddots & -I_r \\
0 & \cdots & 0 & I_r & I_r
\end{pmatrix},
\qquad
\bar D=
\begin{pmatrix}
0 & & & 0 \\
& \ddots & & \\
& & 0 & \\
0 & & & I_r
\end{pmatrix}.
\]

Let \(\widehat b\) denote the drift of the reduced dynamics. Writing
\(z=(x_1,y_1,\ldots,y_{K-1})\in\mathbb R^{Kr}\), we have
\[
\widehat b(z)
=
-(D_r+Q_r)
\begin{pmatrix}
\eta x_1\\
y_1\\
\vdots\\
y_{K-1}
\end{pmatrix}.
\]
Hence its Jacobian is
\[
\widehat J_b(z)
=
-(D_r+Q_r)
\begin{pmatrix}
\eta I_r & & & \\
& I_r & & \\
& & \ddots & \\
& & & I_r
\end{pmatrix}.
\]
Since
\[
\nabla^2\widehat U(x)=\eta I_r,
\]
and \(D_r,Q_r\) are the reduced-dimensional analogues of \(D,Q\), the same
Lyapunov-contraction argument in Appendix~\ref{Appendix-C1} applies. We thus get the following analogue of Theorem~\ref{thm:contr_S_C_s}: There exists
\(\widetilde S\in\mathbb R^{Kr\times Kr}\), with \(\widetilde S\succ0\), such that
\begin{equation}
\widetilde S\,\widehat J_b+\widehat J_b^\top \widetilde S
\preceq
-2C_{\widetilde S}\,\widetilde S ,
\label{eq:S_tilde_contraction}
\end{equation}
for some \(C_{\widetilde S} >0\).

\subsection{Ridge-separable analogue of Appendix~\ref{subsection:part2}}
We recall the ridge-separable setting. Let
\[
U(x)
=
\sum_{k=1}^r \phi_k(\theta_k^\top x)
+
\frac{\eta}{2}\|x\|_2^2,
\qquad x\in\mathbb R^d.
\]
With the notation
\[
\phi'(z)
:=
\bigl(\phi_1'(z_1),\ldots,\phi_r'(z_r)\bigr)^\top,
\qquad z\in\mathbb R^r,
\]
the gradient satisfies
\[
\nabla U(x)
=
\Theta^\top \phi'(\Theta x)+\eta x.
\]
Consequently, after projecting by \(\Theta\),
\[
\Theta\nabla U(x)
=
\Theta\Theta^\top \phi'(\Theta x)+\eta\,\Theta x .
\]

\begin{lemma}[Ridge-separable analogue of Lemma~\ref{lem:nice_events}]
Let \((B_t)_{t\ge 0}\) be a \(d\)-dimensional standard Brownian motion, \(h>0\), and \(N\ge 1\).  For \(C_\Theta\ge 0\), define
\begin{equation}
 \mathcal G_n(h,C_\Theta)
:=
\left\{
\sup_{0\le t\le h}
\left\|
\int_{nh}^{nh+t} \Theta\,dB_s
\right\|_2
\le C_\Theta
\right\},
\qquad n=0,1,\dots,N-1.
\label{eq:nice_events_ridge}
\end{equation}

Then
\[
\mathbb P\!\left(\bigcap_{n=0}^{N-1}\mathcal G_n(h,C_\Theta)\right)
\ge
1-3N\exp\!\left(-\frac{C_\Theta^2}{6 \, r\,\Lambda_\Theta\,h}\right).
\]
In particular, for any \(\delta\in(0,1)\), choosing
\[
C_\Theta=\sqrt{6 \, r\,\Lambda_\Theta\,h\log\frac{3N}{\delta}}
\]
ensures
\[
\mathbb P\!\left(\bigcap_{n=0}^{N-1}\mathcal G_n(h,C_\Theta)\right)\ge 1-\delta.
\]

\end{lemma}

\begin{proof}

\(\Theta (B_{nh+t}-B_{nh})\) is an \(r\)-dimensional Brownian motion with covariance matrix \(t\,\Theta \Theta^\top\). Thus it has the same law as
\[
(\Theta \Theta ^\top)^{1/2}W_t,
\]
where \(W_t\) is an \(r\)-dimensional standard Brownian motion. Since
\[
\|(\Theta  \Theta^\top)^{1/2}x\|_2\le  \sqrt{\Lambda_\Theta}\,\|x\|_2,
\]
we obtain
\[
\sup_{0\le t\le h}
\left\|\Theta\bigl(B_{nh+t}^{(K)}-B_{nh}^{(K)}\bigr)\right\|_2
\le
\sqrt{\Lambda_\Theta}\,
\sup_{0\le t\le h}\|W_t\|_2.
\]

Therefore
\[
\mathbb P\!\left(
\sup_{0\le t\le h}
\left\|
\int_{nh}^{nh+t}\Theta \,dB_s
\right\|_2
\ge C_\Theta
\right)
\le
\mathbb P\!\left(
\sup_{0\le t\le h}\|W_t\|_2
\ge
\frac{C_\Theta}{\sqrt{\Lambda_\Theta}}
\right).
\]
Thus we have
\[
\mathbb P\!\left(
\sup_{0\le t\le h}
\left\|
\int_{nh}^{nh+t}\Theta \,dB_s
\right\|_2
\ge C_\Theta
\right)
\le
3\exp\!\left(-\frac{C_\Theta^2}{6 \, r\,\Lambda_\Theta\,h}\right).
\]
A union bound over \(n=0,\dots,N-1\) yields the claim.
\end{proof}

\begin{lemma}[Ridge-separable analogue of Lemma~\ref{lem: First_Picard_increment_bound}]
Let \(\bar \Theta =  I_{K} \otimes  \Theta\), \(y\in\mathbb R^{Kd}\), \(z:=\bar \Theta y\in\mathbb R^{Kn}\), and
\[
\widehat Z_y^{[1]}(t):=\bar \Theta\,\widehat X_y^{[1]}(t).
\qquad 0\le t\le h,
\]
Here \(\widehat X_y^{[1]}\) is the first Picard iterate on \([0,h]\). Then on the nice event \eqref{eq:nice_events_ridge},  the following bound holds for \(h\le \tfrac{\log 2}{3 \gamma}\):
\[
\sup_{0\le t\le h}\|\widehat Z_y^{[1]}(t)-z\|
\le
e^{\|A_r\|h}
\left[
h\,\|D_r+Q_r\|
\Bigl(
\|\Theta \Theta^\top \phi'(z_1)\|+\eta\|z_1\|+\|z\|
\Bigr)
+\sqrt{2\gamma}\,C_\Theta
\right].
\]
In particular, if \(, \; \phi' \leq \tau_{\phi} \), then
\[
\sup_{0\le t\le h}\|\widehat Z_y^{[1]}(t)-z\|
\le
2\left[
h\,\|D_r+Q_r\|
\Bigl(
\sqrt{r} \Lambda_\Theta L_{\max} +\eta\|z_1\|+\|z\|
\Bigr)
+\sqrt{2\gamma}\,C_\Theta
\right].
\]
\label{lem: First_Picard_increment_bound_ridge}
\end{lemma}

\begin{proof}
For \(t\in[0,h]\), applying \(\bar \Theta\) to the first Picard equation and using
\[
\bar \Theta A= A_r\bar A,
\qquad
\bar \Theta(D+Q)=(D_r+Q_r)\bar \Theta,
\]
we obtain
\begin{align*}
\widehat Z_y^{[1]}(t)-z
&=
\int_0^t  A_r\bigl(\widehat Z_y^{[1]}(s)-z\bigr)\,ds
\;+\;
\int_0^t  A_r z\,ds \\
&\qquad
-
\int_0^t (D_r+Q_r)
\begin{bmatrix}
\Theta \Theta ^\top \phi^{'}(z) + \eta z_1\\
0\\
\vdots\\
0
\end{bmatrix}ds
\;+\;
\int_0^t \bar \Theta \sqrt{2D}\,dB_s .
\end{align*}

Taking norms and then supremum over \(0\le u\le t\) gives
\begin{align*}
    \sup_{0\le u\le t}\|\widehat Z_y^{[1]}(u)-z\|
&\le
\|A_r\|
\int_0^t
\sup_{0\le r\le s}\|\widehat Z_y^{[1]}(r)-z\|\,ds
\\
& + h\,\|D_r+Q_r\|
\Bigl(
\|\Theta \Theta ^\top \phi'(z_1)\|+\eta\|z_1\|+\|z\|
\Bigr)  +  \sqrt{2 \gamma} C_\Theta.
\end{align*}

Grönwall's inequality then yields
\[
\sup_{0\le t\le h}\|\widehat Z_y^{[1]}(t)-z\|
\le
e^{\|A_r\|h}
\left[
h\,\|D_r+Q_r\|
\Bigl(
\|\Theta \Theta ^\top \phi'(z_1)\|+\eta\|z_1\|+\|z\|
\Bigr)
+\sqrt{2\gamma}\,C_\Theta
\right].
\]
Using \(\|\Theta\Theta^\top\|_{\mathrm{op}}=\Lambda_\Theta\) and the uniform bound
\(\|\phi'(z_1)\|_2\le \sqrt r\,L_{\max}\), we obtain the claim.
\end{proof}

We now establish bounds on the projection of the discrete evolution of the algorithm across steps.
\begin{lemma}[Ridge-separable analogue of Lemma~\ref{lem:discrete_contraction}]
There exist constants \(0<h_1'\le h_1\) \eqref{eq:lem_32_h1}, depending only on the fixed problem
parameters, such that for every constant step size \(h\in[h_1',h_1]\), the
bounds used below hold on nice event \eqref{eq:nice_events_ridge}:
\begin{align}
\|\bar{\Theta }\widehat X((n+1)h)\|_{\widetilde S}^2 \lesssim r \log \left(\frac{3N}{\delta}\right).
\end{align}
\label{lem:discrete_contraction_ridge}
\end{lemma}

\begin{proof}
At each algorithmic step, we focus on the projection of the final Picard iteration $\nu$. Note that at the start of each algorithmic step we initialize the Picard iterations at the current state, i.e., set $y := \widehat X(nh)$ (For notational convenience, we omit $y$ in the proof.) Then the following relation holds:
\begin{align}
\widehat Z((n+1)h)
&:=
\bar \Theta \widehat X^{[\nu]}(h) \nonumber\\
&=
\bar \Theta y
-
\int_0^h
 (D_r+Q_r)
\begin{bmatrix}
\Theta P(s;\widehat X^{[\nu-1]})\\
\Theta \widehat X^{[\nu]}_2(s)\\
\vdots\\
\Theta\widehat X^{[\nu]}_K(s)
\end{bmatrix}\,ds +
\sqrt{2}\int_0^h \bar \Theta\sqrt D\,dB_s \nonumber\\
&=
\Biggl(
\widehat Z(nh)
-
\int_0^h
(D_r+Q_r)
\begin{bmatrix}
\Theta\nabla U(y_1)\\
\Theta y_2\\
\vdots\\
\Theta y_K
\end{bmatrix}\,ds
\Biggr) \nonumber \\
&\qquad
+
\Biggl(
-
\int_0^h
(D_r+Q_r)
\begin{bmatrix}
\Theta P(s;\widehat X^{[\nu-1]})- \Theta \nabla U(y_1)\\
\Theta \widehat X^{[\nu]}_2(s)- \Theta y_2\\
\vdots\\
\Theta \widehat X^{[\nu]}_K(s)- \Theta y_K
\end{bmatrix}\,ds
\Biggr) +
\sqrt{2}\int_0^h \bar \Theta \sqrt D\,dB_s\nonumber \\
&= \underbrace{\Biggl(
\widehat Z(nh)
+
\int_0^h
\widehat J_b
\widehat Z(nh)\,ds
\Biggr)}_{=: (\mathrm I)_{\mathrm{pr}}} + 
\underbrace{\int_0^h
-(D_r+Q_r)
\begin{bmatrix}
 \Theta \Theta^\top \phi' \\
0\\
\vdots\\
0
\end{bmatrix}\,ds}_{=: (\mathrm{II})_{\mathrm{pr}}} \nonumber\\
&+
\underbrace{\Biggl(
-
\int_0^h
(D_r+Q_r)
\begin{bmatrix}
\Theta P(s;\widehat X^{[\nu-1]})- \Theta \nabla U(y_1)\\
\Theta \widehat X^{[\nu]}_2(s)- \Theta y_2\\
\vdots\\
\Theta \widehat X^{[\nu]}_K(s)- \Theta y_K
\end{bmatrix}\,ds
\Biggr)}_{=: (\mathrm{III})_{\mathrm{pr}}}
+
\sqrt{2}\int_0^h \bar \Theta\sqrt D\,dB_s.\\
\nonumber
\end{align}

Consider these three terms separately (first term in the $\widetilde S$-norm \eqref{eq:S_tilde_contraction}):
\begin{align*}
\|(\mathrm I)\|_{\widetilde S}^2
&=\Biggl\|
   \widehat Z(nh)
+
h
\widehat J_b
\widehat Z(nh)
   \Biggr\|_{\widetilde S}^2\\
&= \|\widehat Z(nh)\|_{\widetilde S}^2
   + 2h \,\widehat Z(nh)^\top \widetilde S\widehat J_b\widehat Z(nh)
   + h^2 \|\widetilde S\|\|\widehat J_b\|^2 \|\widehat Z(nh)\|^2\\
&\le (1-2h\,C_{\widetilde S} + h^2\|\widetilde S\|\|\widetilde S^{-1}\|\|\widehat  J_b\|^2)\|\widehat Z(nh)\|_{\widetilde S}^2.
\end{align*}
For \(h \leq 
\min\!\left\{
\frac{C_{\widetilde S}}{4\bigl(\|\widetilde S\|\,\|\widetilde S^{-1}\|\,\|\widehat J_b\|^2+2C_{\widetilde S}^2\bigr)},
\;
\frac{1}{2\sqrt{\|\widetilde S\|\,\|\widetilde S^{-1}\|}\,\|\widehat J_b\|}
\right\}\) we have

\[(1 + C_{\widetilde S} h) \|\mathrm I\|_{\widetilde S}^2 \leq \left(1-\frac{C_{\widetilde S}}{2}h\right)\|\widehat Z(nh)\|_{\widetilde S}^2\]

Let us deal with the second term now:
\[
\left\|(\mathrm{II})
\right\|_{\widetilde S}
\le
\sqrt{r}h\|\widetilde S\|\|D_r+Q_r\| \Lambda_{\Theta} L_{max} = C_0 h \sqrt{r},\]
where we absorb the remaining constants in \(C_0\). We move on to the third term in $l_2$ norm.

\begin{align}
 \|(\mathrm{III})\|_2 &=  \left \| \int_{0}^{h} (D_r{+}Q_r)
      \begin{bmatrix}
        \Theta P\big(s;\,\widehat X^{[\nu-1]}\big) -  \Theta\nabla U\big(\widehat X_1(nh)\big)\\
        \Theta\widehat X^{[\nu]}_2(s) - \Theta\widehat X_2(nh)\\[-2pt]
        \vdots\\[-2pt]
        \Theta\widehat X^{[\nu]}_K(s) - \Theta\widehat X_K(nh)
      \end{bmatrix}\,\mathrm ds \right\|_2 \nonumber\\ 
      &\leq  \sum_{i=1}^M \int_{B_i} \left\| (D_r{+}Q_r)
      \begin{bmatrix}
        \Theta P_{B_i}\big(s;\,\widehat X^{[\nu-1]}\big) -  \Theta \nabla U\big(\widehat X_1(nh)\big)\\
         0\\[-2pt]
        \vdots\\[-2pt]
        0
      \end{bmatrix} \right\|_2 \,\mathrm ds \nonumber\\ 
      &\enspace \enspace\qquad+
        \int_0^h
        \left\|
        (D_r+Q_r)
        \begin{bmatrix}
        0\\
        \Theta\widehat X^{[\nu]}_2(s)-\Theta\widehat X_2(nh)\\
        \vdots\\
        \Theta\widehat X^{[\nu]}_K(s)-\Theta\widehat X_K(nh)
        \end{bmatrix}
        \right\|_2 ds \label{eq:ineq0_ridge}\\
      &\leq  h\Gamma_\phi(\Lambda_\Theta L_{\max}+\eta) \,\|D_r{+}Q_r\| \sup_{s \in [0,h]}  \big\|\Theta\widehat X^{[\nu-1]}(s) -  \Theta\widehat X(nh)\big\| \nonumber\\
      &\enspace \enspace\qquad +\; h\,\|D_r{+}Q_r\| \sup_{s \in [0,h]}\big\|  \Theta\widehat X^{[\nu]}(s) - \Theta\widehat X(nh)\big\|   \label{eq:ineq1_ridge} \\
      & \leq  \rho_{\Theta}\,\|D_r{+}Q_r\|  \sup_{s \in [0,h]}  \big\|\Theta\widehat X^{[\nu-1]}(s) -  \Theta\widehat X(nh)\big\| \nonumber\\
      &\enspace \enspace\qquad +\; h\,\|D_r{+}Q_r\| \sup_{s \in [0,h]} \big\|  \widehat \Theta X^{[\nu]}(s) - \Theta\widehat X(nh)\big\|  \label{eq:ineq2_ridge}\\     
      & \leq \rho_{\Theta}\,\|D_r{+}Q_r\| \frac{1-\rho_{\Theta}^{\nu -1}}{1-\rho_{\Theta}}\sup_{s \in [0,h]}  \big\|\widehat \Theta X^{[1]}(s) -  \Theta \widehat X(nh)\big\| \nonumber \\
      &\enspace \enspace\qquad +\; h\,\|D_r{+}Q_r\|\frac{1-\rho_{\Theta}^{\nu}}{1-\rho_{\Theta}}\sup_{s \in [0,h]} \big\|  \Theta \widehat X^{[1]}(s) - \Theta\widehat X(nh)\big\|  \label{eq:ineq3_ridge}\\
      & \leq  2\!\left[\rho_{\Theta}\,\|D_r{+}Q_r\| \frac{1-\rho_{\Theta}^{\nu-1}}{1-\rho_{\Theta}} + h\,\|D_r{+}Q_r\|\frac{1-\rho_{\Theta}^{\nu}}{1-\rho_{\Theta}}\right]\nonumber \\
      & \hspace{1cm}\times \left[
h\,\|D_r{+}Q_r\|\left((\eta+1) \|\widehat Z(nh)\| + \sqrt{r} \Lambda_\Theta \tau_{\phi} \right) + \sqrt{2\gamma}\,C_\Theta\right]. \label{eq:ineq4_ridge}
\end{align}

We now explain each inequality in detail:
\begin{itemize}
 \item \textbf{\eqref{eq:ineq0_ridge}} follows since \(P(s;X) = P_{B_i}(s;X)\) for \(s \in B_i\), which allows us to decompose the integral blockwise.
  \item \textbf{\eqref{eq:ineq1_ridge}} follows from the Lipschitz continuity of \(\phi^{'}\) in \(\Theta\nabla U\) and the definition of the interpolant \(P_{B_i}(s;\widehat X^{[\nu-1]})\). Using the identity \(\sum_j \ell_{i,j}(s)=1\), we rewrite the difference and then apply the triangle inequality along with the Lipschitz continuity of \(\phi^{'}\).
  
  \item \textbf{\eqref{eq:ineq2_ridge}} applies the definition of 
  \(\rho_{\Theta} = 2\, h\,\Gamma_\phi (\Lambda_{\Theta} L_{max} + \eta)\) used in Remark~\ref{rem:projected-geometric-picard}.
  
  \item \textbf{\eqref{eq:ineq3_ridge}} follows from the geometric convergence bound in Remark~\ref{rem:projected-geometric-picard}.
  
  \item \textbf{\eqref{eq:ineq4_ridge}} follows from Lemma~\ref{lem: First_Picard_increment_bound_ridge}, 
  which bounds the first Picard increment.
\end{itemize}

Next, we express $(\mathrm{III})$ in terms of the current projected state $\widehat Z(nh)$, noting that with $\rho_{\Theta} \leq  \tfrac{1}{2}$ we have
\begin{align*}
 \|(\mathrm{III})\|_2  &\leq 2h\,\|D_r{+}Q_r\|^2(1+\eta)\left[\rho_{\Theta}\frac{1-\rho_{\Theta}^{\nu-1}}{1-\rho_{\Theta}} + h\frac{1-\rho_{\Theta}^{\nu}}{1-\rho_{\Theta}}\right]\|\widehat Z(nh)\| \nonumber\\
 &\;+\;2h\,\|D_r{+}Q_r\|^2\left[\rho_{\Theta}\frac{1-\rho_{\Theta}^{\nu-1}}{1-\rho_{\Theta}} + h\frac{1-\rho_{\Theta}^{\nu}}{1-\rho_{\Theta}}\right]\sqrt{r} \Lambda_{\phi} \gamma_{\phi}\nonumber\\
& \qquad + 2\sqrt{2\gamma}\,C_{\Theta} \|D_r{+}Q_r\|\left[\rho_{\Theta} \, \frac{1-\rho_{\Theta}^{\nu-1}}{1-\rho_{\Theta}} + h \,\frac{1-\rho_{\Theta}^{\nu}}{1-\rho_{\Theta}}\right] \nonumber\\
&\leq 2h\,\|D_r{+}Q_r\|^2(1+\eta)(2\rho_{\Theta} +2h)\|\widehat Z(nh)\|\\
&+ 2\sqrt{r} \Lambda_{\phi} \gamma_{\phi}h\,\|D_r{+}Q_r\|^2(1+\eta)(2\rho_{\Theta} +2h)\\
 &\;+\; 2\sqrt{2\gamma}\,C_\Theta \,\|D_r{+}Q_r\|(2\rho_{\Theta} +2h)\, \qquad (\because\rho \leq 1/2)\nonumber \\
 &\leq 4h^2\,\|D_r{+}Q_r\|^2(1+\eta)(1+2(L_{\max} \Lambda_{\Theta}+1)\Gamma_{\phi})\|\widehat Z(nh)\|\\
& \;+\; 4\sqrt{r} \Lambda_{\phi} \gamma_{\phi}h^2\,\|D_r{+}Q_r\|^2(1+2(L_{\max} \Lambda_{\Theta}+1)\Gamma_{\phi})\\
 &\;+\; 4\sqrt{2\gamma}\,C_b \,\|D_r{+}Q_r\|(1+2(L_{\max} \Lambda_{\Theta}+1)\Gamma_{\phi}) h. \qquad (\because\rho = 2(L_{\max} \Lambda_{\Theta} +1) h\,\Gamma_{\phi})\nonumber \\
 &= C_1 h^2\|\widehat Z(nh)\|
   + C_2 h^2\sqrt r\,
   + C_3 C_\Theta h\, ,
\end{align*}
where we absorb the constants in \(C_1, C_2, C_3.\)
We now proceed to analyze the recursion in the $\widetilde{S}$-norm.  
Combining \eqref{eq: high_prob_1} with the bounds for terms $(\mathrm{I})$-$(\mathrm{II})$ (Converting into $\widetilde S$-norm), we obtain
\begin{align}
 \|\widehat Z((n+1)h)\|_{\widetilde{S}}^2 &  \leq  \left(1-\frac{C_{\widetilde S}}{2}h\right)\|\widehat Z(nh)\|_{\widetilde{S}}^2 + \left(3+ \frac{3}{C_{\widetilde{S}}h}\right )C_0^2 r h^2 \nonumber\\
 & \;+\; \left(3+ \frac{3}{C_{\widetilde{S}}h}\right)\|\widetilde{S}\| \left( 3C_1^2h^4 \|\widetilde S^{-1}\| \|\widehat Z(nh)\|_{\widetilde S}^2 + 3C_2^2 h^4 r + 3C_3^2 C_{\Theta}^2 h^2 \right)\nonumber \\
 & \;+\;  \left(3+ \frac{3}{C_{\widetilde{S}}h}\right)2\gamma C_\Theta^2 \|\widetilde S\| \nonumber \\
&\le
\Bigg[
1-\frac{C_{\widetilde S}}{2}h
+ \left(3+\frac{3}{C_{\widetilde S}h}\right) 3C_1^2 h^4 \|\widetilde S\|\,\|\widetilde S^{-1}\| \Bigg] \|\widehat Z(nh)\|_{\widetilde S}^2
\nonumber\\
&\quad + \left(3+\frac{3}{C_{\widetilde S}h}\right)\Bigl[
C_0^2 r h^2 + \|\widetilde S\|\bigl(3C_2^2 h^4 r + 3C_3^2 C_\Theta^2 h^2\bigr)
+ 2\gamma C_\Theta^2 \|\widetilde S\| \Bigr] \nonumber \\
& \lesssim \left(1-\frac{C_{\widetilde S}}{4}h\right)
\|\widehat Z(nh)\|_{\widetilde S}^2
\nonumber\\
&\quad+ 
C_0^2 r h
+
3\|\widetilde S\|C_2^2 h^3 r + 3\|\widetilde S\|C_3^2 C_\Theta^2 h
+
\frac{1}{h}2\gamma C_\Theta^2 \|\widetilde S\|  \nonumber
\end{align}

where the last inequality is obtained by choosing a sufficiently small stepsize \(h \in [h'_1, h_1]\) such that  
\begin{equation}
h_1 := \min\left\{
\left(
\frac{C_{\widetilde S}}
{72\,C_1^2\|\widetilde S\|\,\|\widetilde S^{-1}\|}
\right)^{1/3},
\quad
\frac{C_{\widetilde S}}
{\sqrt{72}\,C_1\,\|\widetilde S\|^{1/2}\|\widetilde S^{-1}\|^{1/2}}
\right\}, h'_1 = ch_1,
\label{eq:lem_32_h1}
\end{equation}
for some \(c \in (0,1)\). The  lower bound needed for this lemma (\( h'_1\)) will be fixed when \(h\) is
chosen in the proof of the main theorem.

Starting from the initialization $\widehat X(0)=0$, the recursion yields for all $n$,
\begin{align*}
\|\widehat Z((n+1)h)\|_{\widetilde S}^2 \lesssim   r  \log \frac{3N}{\delta}
\end{align*}

where we used the definition of \(C_{\Theta}\), together with the fact that \(h\) is bounded above and below by positive constants.
\end{proof}

Next we bound the projection of the Picard fixed point initialized at the current state $\widehat X(nh)$.

\begin{lemma}
\label{lem:high-prob-bound_fixed_point_ridge} [Ridge-separable analogue of Lemma~\ref{lem:high-prob-bound_fixed_point}]
Assume the step size condition in Lemma~\ref{lem: First_Picard_increment_bound_ridge} and ~\ref{lem:discrete_contraction_ridge}. Then the following bound holds with probability at least \(1-\delta\):
\begin{equation}
\label{eq:high-prob-bound}
\sup_{t \in [0,1]}\|\bar \Theta\widehat X^\ast_y(nh + th)\| 
\lesssim
\left[\, r \,\log\!\frac{3N}{\delta} \right]^{1/2},
\end{equation}
where $y=\widehat X(nh)$.
\end{lemma}

\begin{proof}
For notational convenience, we suppress the subscript $y$, though we emphasize that the Picard fixed point is initialized at $y=\widehat X(nh)$. Under the event $\bigcap_{n=0}^{N-1}\mathcal G_n(h,C_\Theta)$, we have
\begin{align}
\sup_{t \in [0,1]}\|\bar \Theta \widehat X^*(nh + th)\|^2 
&\leq  2\sup_{t \in [0,1]}\|\bar \Theta \widehat X^*(nh + th) - \bar \Theta \widehat X(nh)\|^2 + 2\|\bar \Theta \widehat X(nh)\|^2 \nonumber\\
&\leq  \frac{2}{(1-\rho_\Theta )^2}\sup_{t \in [0,1]}\|\bar \Theta \widehat X^{[1]}(nh + th) - \bar \Theta\widehat X(nh)\|^2 \nonumber\\
&\qquad + 2\|\bar \Theta \widehat X(nh)\|^2 \qquad \text{(by remark \eqref{rem:projected-geometric-picard})}\nonumber\\
&\lesssim (2+h^2)\|\bar \Theta \widehat X(nh)\|^2
+ h^2 r + C_\Theta^2.\qquad \text{(by Lemma~\ref{lem: First_Picard_increment_bound_ridge})}\nonumber
\end{align}
Converting $l_2$ norm into $\widetilde S$ norm and applying  Lemma~\ref{lem:discrete_contraction_ridge} we have

\begin{align}
\sup_{t \in [0,1]}\|\bar \Theta \widehat X^*(nh + th)\|^2 &\lesssim   (2+h^2) \|\widetilde S^{-1}\| \|\bar \Theta \widehat{X}(nh)\|_{\widetilde S}^2 + h^2r +  C_\Theta^2\nonumber\\
&\lesssim  (2+h^2) \|\widetilde S^{-1}\| \left( r  \log \frac{3N}{\delta} \right) + h^2 r + C_{\Theta}^2 \nonumber \\
& \lesssim  r  \log \frac{3N}{\delta}
\end{align}
where we used the definition of \(C_{\Theta}\).
\end{proof}

Finally, we state the moment bound needed to control the interpolation error for the ridge-separable case.

\begin{lemma}
\label{lem:expectaion_power2k_ridge}
Assume the step size condition in Lemma~\ref{lem: First_Picard_increment_bound_ridge} and ~\ref{lem:discrete_contraction_ridge}. Let
\[
Z_{\Theta} \;:=\; \sup_{t\in[0,1]}\bigl\|\bar \Theta \widehat X^\ast_y(nh+th)\bigr\|,
\]
for $y = \widehat X(nh).$ Then for every integer $k\ge1$,
\[
\mathbb{E}\bigl[Z_{\Theta}^{2k}\bigr]
\;\lesssim\; N r^k.
\]
\end{lemma}

\begin{proof}
For $k\ge1$,
\[
\mathbb{E}\bigl[Z_{\Theta}^{2k}\bigr]
= \int_{0}^{\infty}\mathbb{P}\bigl(Z_\Theta^{2k}>t\bigr)\,dt
= \int_{0}^{\infty}\mathbb{P}\bigl(Z_\Theta>u\bigr)\,2k\,u^{2k-1}\,du .
\]
By the tail estimate of Lemma~\ref{lem:high-prob-bound_fixed_point_ridge}, we obtain
\[
\mathbb{E}\bigl[Z_\Theta^{2k}\bigr]
\leq  6kN \int_{0}^{\infty} u^{2k-1}
    \exp\Bigl(-\frac{u^{2}}{Cr}\Bigr)\,du ,
\]
for some positive constant $C > 0$ independent of $d,r, M, N, \varepsilon$.
Let $a:=1/(C r)$. The standard integral
\(
\int_{0}^{\infty} u^{2k-1} e^{-a u^{2}}\,du
= \tfrac{1}{2} a^{-k}\Gamma(k)
\)
yields
\[
\mathbb{E}\bigl[Z_\Theta^{2k}\bigr]
\lesssim N r^k. \]
as claimed.
\end{proof}

\section{Auxiliary lemmas}
\label{Appendix-I}
\begin{lemma}\label{lemma:d_plus_q}
Let $D,Q$ be as defined in \eqref{eq:D-and-Q} with $\gamma>0$, and let $K\ge2$. Let $\|\cdot\|$ denote the operator norm (induced by Euclidean norm) for a matrix.
Then
\[
\sqrt{1+\gamma^2}\ \le\ \|D+Q\|
\;\le\;
\|D\|+\|Q\|
\;\le\;
\gamma+\max\{\,1+\gamma,\;2\gamma\,\}
\;=\;\max\{\,1+2\gamma,\;3\gamma\,\},
\]

Moreover, if
\[
J:=\operatorname{diag}(0,1,\ldots,1)\otimes I_d,
\qquad
A:=-(D+Q)J,
\]
then
\[
\sqrt{1+\gamma^2}
\;\le\;
\|A\|
=
\|(D+Q)J\|
\;\le\;
\|D+Q\|
\;\le\;
\max\{\,1+2\gamma,\;3\gamma\,\}.
\]
In particular, when \(\gamma\ge1\), \(\|A\|\le 3\gamma\) and \(\|D+Q\| \leq 3\gamma\).
\end{lemma}

\begin{proof}
By construction $D=\mathrm{diag}(0,\ldots,0,\gamma)\otimes I_d$, hence
$\|D\|=\gamma$.
Also $Q=T_\gamma\otimes I_d$, where $T_\gamma\in\mathbb{R}^{K\times K}$ is the
(skew–symmetric) tridiagonal matrix
\[
T_\gamma=
\begin{pmatrix}
0      & -1     & 0      & \cdots & 0      & 0 \\
1      & 0      & -\gamma& \ddots & \vdots & \vdots \\
0      & \gamma & 0      & \ddots & 0      & 0 \\
\vdots & \ddots & \ddots & \ddots & -\gamma& 0 \\
0      & \vdots & 0      & \gamma & 0      & -\gamma \\
0      & \vdots & 0      & \cdots & \gamma & 0
\end{pmatrix}.
\]
Kronecker structure gives $\|Q\|=\|T_\gamma\|$.
For any matrix $A$, $\|A\|\le \sqrt{\|A\|_1\|A\|_\infty}$; we now compute
$\|T_\gamma\|_1$ and $\|T_\gamma\|_\infty$. Calculating the row sums (in absolute value) of $T_\gamma$ yields: $\|T_\gamma\|_\infty=\max\{1+\gamma,\,2\gamma\}$.
By the same pattern for column sums, $\|T_\gamma\|_1=\max\{1+\gamma,\,2\gamma\}$.
Therefore
\[
\|Q\|=\|T_\gamma\|
\;\le\; \sqrt{\|T_\gamma\|_1\|T_\gamma\|_\infty}
\;=\; \max\{1+\gamma,\,2\gamma\}.
\]
Finally, by the triangle inequality,
\[
\|D+Q\|
\;\le\; \|D\|+\|Q\|
\;\le\; \gamma+\max\{1+\gamma,\,2\gamma\}
\;=\; \max\{1+2\gamma,\,3\gamma\}.
\]

For the lower bounds, note first that $\|B\otimes I_d\|=\|B\|$.
Write $D+Q=(\tilde D+\tilde Q)\otimes I_d$ with $\tilde D=\mathrm{diag}(0,\ldots,0,\gamma)$
and $\tilde Q=T_\gamma$. For any matrix $M$, $\|M\|\ge \max_j\|Me_j\|_2$.
Reading off the second column of $\tilde D+\tilde Q$ gives
\[
(\tilde D+\tilde Q)e_2=
\begin{cases}
-\,e_1+\gamma e_2,& K=2,\\
-\,e_1+\gamma e_3,& K\ge 3,
\end{cases}
\]
hence $\|(\tilde D+\tilde Q)e_2\|_2=\sqrt{1+\gamma^2}$ for all $K\ge2$, and therefore
$\|D+Q\|\ge \sqrt{1+\gamma^2}$.

Let $P:=\mathrm{diag}(0,1,\ldots,1)$. Since $\|(D+Q)P\otimes I_d\|=\|(\tilde D+\tilde Q)P\|$
and $P$ leaves column $2$ unchanged, the same column calculation yields
\[
\|(D+Q)P\otimes I_d\|
=\|(\tilde D+\tilde Q)P\|
\ \ge\ \|(\tilde D+\tilde Q)Pe_2\|_2
=\sqrt{1+\gamma^2}.
\]
\end{proof}

\begin{lemma}
\label{lem:J_b-operator}
Let
\[
b(x):=-(D+Q)\nabla H(x),
\qquad
H(x)=U(x_1)+\frac12\sum_{k=2}^K\|x_k\|^2.
\]
Assume that
\[
mI_d\preceq \nabla^2 U(x_1)\preceq LI_d,
\qquad \forall x_1\in\mathbb R^d,
\]
for some \(0<m\le L<\infty\). Then, for all \(x\in\mathbb R^{Kd}\), the Jacobian
\[
J_b(x):=\nabla b(x)=-(D+Q)\nabla^2 H(x)
\]
satisfies
\[
\|J_b(x)\|
\le
\max\{1,L\}\max\{1+2\gamma,3\gamma\}.
\]
\end{lemma}

\begin{proof}
Since
\[
\nabla^2 H(x)
=
\operatorname{diag}\!\big(\nabla^2 U(x_1),I_d,\ldots,I_d\big),
\]
we have
\[
\|\nabla^2 H(x)\|
\le
\max\{L,1\}.
\]
Therefore, by Lemma~\ref{lemma:d_plus_q},
\[
\|J_b(x)\|
=
\|(D+Q)\nabla^2 H(x)\|
\le
\|D+Q\|\,\|\nabla^2 H(x)\|
\le
\max\{1+2\gamma,3\gamma\}\max\{1,L\}.
\]
\end{proof}

\begin{lemma}\label{lem:bound-for-basis}
Let $K\ge 2$ and take uniform nodes
\[
c_j=\frac{j-1}{K-2}, \qquad j=1,\ldots,K-1,
\]
on $[0,1]$.
Let $\{\tilde\ell_j\}_{j=1}^{K-1}$ be the associated Lagrange basis and
\[
\Gamma_\phi:=\sup_{\tau\in[0,1]}\sum_{j=1}^{K-1}|\tilde\ell_j(\tau)|
\]
the Lebesgue constant. Then
\[
\Gamma_\phi \;\le\; \frac{2^{K-2}\,(K-2)^{K-2}}{(K-2)!}.
\]
\end{lemma}

\begin{proof}
Fix $\tau\in[0,1]$ and $j\in\{1,\ldots,K-1\}$. Using the product form,
\[
|\tilde\ell_j(\tau)|
=\prod_{\substack{k=1\\k\ne j}}^{K-1}
\frac{|\tau-c_k|}{|c_j-c_k|}
\le \prod_{\substack{k=1\\k\ne j}}^{K-1}
\frac{1}{|c_j-c_k|}.
\]
For the uniform grid, $|c_j-c_k|=|j-k|/(K-2)$, so
\[
\prod_{k\ne j}|c_j-c_k|
=\Big(\frac{1}{K-2}\Big)^{K-2}\prod_{k\ne j}|j-k|
=\Big(\frac{1}{K-2}\Big)^{K-2}(j-1)!\,(K-1-j)!.
\]
Hence
\[
|\tilde\ell_j(\tau)|
\le \frac{(K-2)^{K-2}}{(j-1)!\,(K-1-j)!}.
\]
Summing over $j$ and using
\[
\sum_{j=1}^{K-1} \binom{K-2}{j-1}=2^{K-2},
\]
we obtain
\[
\sum_{j=1}^{K-1} |\tilde\ell_j(\tau)|
\le \frac{(K-2)^{K-2}}{(K-2)!}\sum_{j=1}^{K-1} \binom{K-2}{j-1}
= \frac{2^{K-2}\,(K-2)^{K-2}}{(K-2)!}.
\]
Taking the supremum over $\tau\in[0,1]$ gives the claim.
\end{proof}

\begin{lemma}[Stationary moment bounds]\label{lem:StationaryMoments}
Let $H(x)=U(x_1)+\frac12\sum_{k=2}^K\|x_k\|^2$ and $\rho(dx)\propto e^{-H(x)}dx$ be the stationary law of the dynamics, with $U$ satisfying Assumption~\ref{assump:U-strong-smooth}. Let $x_\star\in\arg\min U$ and $X=(X_1, \ldots, X_K) \sim\rho$,
then:
\begin{align}
\mathbb{E}_\rho\!\big[\|X_1-x_\star\|^2\big] &\le \frac{d}{m}, \label{eq:BL-second-moment}\\
\mathbb{E}_\rho\!\big[\|\nabla U(X_1)\|^2\big] &\le L^2\,\mathbb{E}_\rho\!\big[\|X_1-x_\star\|^2\big] \;\le\; \frac{L^2}{m}\,d, \label{eq:grad-moment-both}\\
\mathbb{E}_\rho\!\big[\|X\|^2\big] = \mathbb{E}_\rho\!\big[\|X_1\|^2\big] &+ \sum_{k=2}^K\mathbb{E}_\rho\!\big[\|X_k\|^2\big]
\;\le\; \big(2\|x_\star\|^2+\tfrac{2d}{m}\big) \;+\; (K-1)\,d. \label{eq:full-state-moment}
\end{align} 
\end{lemma}

\begin{proof}
\textbf{Bound for $\mathbb{E}\|X_1-x_\star\|^2$.}
Write $\pi_U(dx_1)\propto e^{-U(x_1)}dx_1$ for the marginal stationary law of $X_1$; under $\rho$ we have the factorization
$\rho(dx)=\pi_U(dx_1)\otimes\big(\otimes_{k=2}^K\mathcal N(0,I_d)\big)$.
Because $U$ is $m$-strongly convex, $U(x_1)\ge U(x_\star)+\tfrac{m}{2}\|x_1-x_\star\|^2$, so $e^{-U}$ has Gaussian tails. Thus, for $g(x_1)=x_1-x_\star$ (whose divergence is $\nabla\!\cdot g \equiv d$), integration by parts yields
\[
\int_{\mathbb{R}^d} \!\!\langle \nabla U(x_1),\,x_1-x_\star\rangle\,e^{-U(x_1)}\,dx_1
= \int_{\mathbb{R}^d} \!\! \nabla\!\cdot g(x_1)\,e^{-U(x_1)}\,dx_1
= d \int_{\mathbb{R}^d}\! e^{-U(x_1)}\,dx_1,
\]
where the boundary term vanishes since $e^{-U}$ has tails. Normalizing gives
\[
\mathbb{E}_{\pi_U}\!\big[\langle \nabla U(X_1),\,X_1-x_\star\rangle\big]=d.
\]
By strong convexity and $\nabla U(x_\star)=0$,
$\langle \nabla U(x_1),x_1-x_\star\rangle
\ge m\|x_1-x_\star\|^2$. Taking expectation,
\(
d \ge m\,\mathbb{E}_{\pi_U}\|X_1-x_\star\|^2.
\)
Since the $X_1$-marginal of $\rho$ is $\pi_U$, the same bound holds under $\rho$, proving \eqref{eq:BL-second-moment}.

\textbf{Bound for $\mathbb{E}\|\nabla U(X_1)\|^2$.}
First, by $L$-smoothness and $\nabla U(x_\star)=0$,
\(
\|\nabla U(x_1)\|\le L\|x_1-x_\star\|.
\)
Squaring and taking expectations then using \eqref{eq:BL-second-moment} gives
\(
\mathbb{E}\|\nabla U(X_1)\|^2 \le (L^2/m)\,d.
\)

\textbf{Bound for $\mathbb{E}\|X\|^2$.}
Under $\rho$, $X_k\sim\mathcal N(0,I_d)$ for $k\ge2$, hence
$\mathbb{E}\|X_k\|^2=d$. For $X_1$,
\(
\|X_1\|^2=\|X_1-x_\star+x_\star\|^2
\le 2\|X_1-x_\star\|^2+2\|x_\star\|^2.
\)
Taking expectations and applying \eqref{eq:BL-second-moment} gives
\(
\mathbb{E}\|X_1\|^2 \le 2(d/m)+2\|x_\star\|^2.
\)
\end{proof}

\end{document}